\documentclass[a4paper,reqno]{amsart}

\usepackage[utf8]{inputenc}
\usepackage{amssymb}
\usepackage{enumitem}
\usepackage{mathrsfs}
\usepackage{mathtools}
\usepackage{amsmath}
\usepackage{xcolor}
\usepackage[colorlinks=true]{hyperref}
\hypersetup{linkcolor=violet,urlcolor=cyan,citecolor=cyan}

\setlist[enumerate]{label=\emph{(\roman*)}}

\usepackage{environ}
\usepackage{cancel}

\usepackage{amsmath,amssymb,mathrsfs}

\newtheorem{theorem}{Theorem}[section]
\newtheorem{corollary}[theorem]{Corollary}
\newtheorem{lemma}[theorem]{Lemma}
\newtheorem{proposition}[theorem]{Proposition}

\theoremstyle{definition}

\newtheorem{remark}[theorem]{Remark}
\numberwithin{equation}{section}
\newcommand{\R}{\mathbb{R}}

\def \d {{\rm{d}}}

\begin{document}

\parindent=0pt

	\title[2D KGZ]
	{Scattering for the Klein-Gordon-Zakharov system in two dimensions}
	
		\author[S.~Dong]{Shijie Dong}
	\address{Shenzhen International Center for Mathematics, and Department of Mathematics, Southern University of Science and Technology,  518055 Shenzhen, P.R. China.}
\email{dongsj@sustech.edu.cn, shijiedong1991@hotmail.com}

      \author[Z.~Guo]{Zihua Guo}
\address{Monash University, School of Mathematics,  Melbourne, Australia.}
\email{zihua.guo@monash.edu}
	
		\author[K.~Li]{Kuijie Li}
\address{Nankai University, School of Mathematical Sciences, Tianjin, 300071, P.R. China.}
\email{kuijiel@nankai.edu.cn}

\begin{abstract}

We study the Klein-Gordon–Zakharov system in two spatial dimensions, an important model in plasma physics. For small, smooth, and spatially localized initial data, we establish the global existence of solutions and characterize their sharp long-time behavior, including sharp time decay and scattering properties.
A particularly interesting phenomenon is that the Klein-Gordon component exhibits modified scattering for certain initial data, while for others it undergoes linear scattering—a dichotomy highlighting delicate long-range interaction effects.

The major obstacles are lack of symmetry and weak decay of the solution in two dimensions. To overcome these, we introduce a novel nonlinear transformation of the wave component and reinterpret the nonlinear coupling as a perturbation of the mass term in the Klein-Gordon equation. The proof employs a combination of physical space and frequency space methods.


	\end{abstract}
	\maketitle
	
\tableofcontents	
	
\section{Introduction}\label{sec:intro}
\subsection{Model problem and main result}

In this paper, we consider the two-dimensional Klein–Gordon–Zakharov system, an important model in plasma physics modeling the interaction between high-frequency Langmuir waves and low-frequency ion-acoustic waves in a plasma. The system of equations takes the form
\begin{equation}\label{eq:2D-KGZ}
\aligned
-\Box n &= \Delta |E|^2,
\\
-\Box E + E &= -nE,
\endaligned
\end{equation}
where $n: \mathbb{R}^{1+2} \to \mathbb{R}$ and $E: \mathbb{R}^{1+2} \to \mathbb{R}^2$ are the unknowns. The Cauchy problem is posed with initial data prescribed at $t_0 = 1$:
\begin{equation}\label{eq:ID}
(n, \partial_t n, E, \partial_t E)(t_0) = (n_0, n_1, E_0, E_1).
\end{equation}

Let $\nabla = (\partial_{x_1}, \partial_{x_2})$. We denote by $\|\cdot\|$ the $L^2$ norm on $\mathbb{R}^2$ and by $\|\cdot\|_{H^k}$ the Sobolev $H^k$ norm for $k \in \mathbb{N}$. The Japanese bracket is written as $\langle \cdot \rangle := \sqrt{1+|\cdot|^2}$. Throughout the paper, we use $A\lesssim B$ to indicate $A\leq CB$ with $C$ a universal constant, and $A \simeq B$ if both $A \lesssim B$ and $B \lesssim A$ hold.

\begin{theorem}[Global existence and decay]\label{thm:main1}
Let $N\geq 16$ be an integer.
There exists an $\epsilon_0 >0$, such that for all initial data satisfying the smallness conditions
\begin{align}
&\sum_{|I|\leq N+1}\big\| \langle x\rangle^{|I|+3} \nabla^I n_0 \big\|
+
\sum_{|I|\leq N}\big\| \langle x\rangle^{|I|+3} \nabla^I n_1 \big\|
\\
&+
\sum_{|I|\leq N+2}\big\| \langle x\rangle^{15} \nabla^I E_0 \big\|
+
\sum_{|I|\leq N+1}\big\| \langle x\rangle^{15} \nabla^I E_1 \big\|
< \epsilon \leq \epsilon_0,
\end{align}
the Cauchy problem \eqref{eq:2D-KGZ}-\eqref{eq:ID} admits a global solution $(n, E)$.
Moreover, the solution enjoys sharp time decay 
\begin{equation}\label{eq:thm-decay}
|n(t, x)| \leq C \epsilon \langle t+|x|\rangle^{-{1\over 2}} \langle t-|x|\rangle^{-{1\over 2}},
\qquad
|E(t, x)| \leq C \epsilon \langle t+|x|\rangle^{-1},
\end{equation}
and the uniform-in-time Sobolev norm bounds
\begin{align}\label{eq:thm-Sobolev}
    \|(\partial_t n, \nabla n)\|_{H^N} + \|E\|_{H^{N+2}} + \|\partial_t E\|_{H^{N+1}}
    \leq C\epsilon,
\end{align}
with $C>0$ a constant.
\end{theorem}

\begin{remark}
Recall that for linear homogeneous wave or Klein-Gordon equations in $\mathbb{R}^{1+2}$, the Sobolev norms remain bounded uniformly in time, and the optimal pointwise time decay rates are $t^{-{1\over 2}}$ for the wave component and $t^{-1}$ for the Klein–Gordon component. In this sense, the decay estimates \eqref{eq:thm-decay} and the uniform Sobolev bounds \eqref{eq:thm-Sobolev} are sharp.
\end{remark}

\begin{remark}
Several global existence results for the 2D Klein–Gordon–Zakharov system are known; see, for instance, \cite{Dong2006, Duan-Ma, DoMa, Cheng}. However, many of these works impose additional restrictions on the initial data, such as compact support, divergence-form structure, and others.
Theorem~\ref{thm:main1} removes these restrictions with sharp decay and Sobolev bounds, which will be crucial for our later analysis.
\end{remark}

The following theorem describes the scattering behavior of the solution $(n, E)$.

\begin{theorem}[Scattering results]\label{thm:main2}
Let the same assumptions in Theorem \ref{thm:main1} hold. Then, we have the following scattering results.
\begin{itemize}
    \item  If 
\begin{align}\label{eq:n1=+}
    \int_{\mathbb{R}^2} n_1 \, \d x \neq  0,
\end{align}
then the field $n$ scatters linearly, while the field $E$ undergoes nonlinear scattering (i.e., modified scattering) in the energy space.

    \item  If  
\begin{align}\label{eq:n1=0}
    \int_{\mathbb{R}^2} n_1 \, \d x = 0,
\end{align}
then both fields $(n, E)$ scatter linearly in the energy space.

\end{itemize}
\end{theorem}

\begin{remark}
Theorem~\ref{thm:main2} provides a complete classification of the scattering behavior of $(n, E)$ in the setting of smooth, small, and localized initial data.

In the nonlinear scattering case \eqref{eq:n1=+}, there exist a phase correction function $\Theta(t, \xi)$ (defined in \eqref{eq:ph-correction}) and an asymptotic profile $f_*(\infty)$ (defined in \eqref{eq:f-infty}) such that
    \begin{align*}
        \big\|(\partial_t - i\langle \Lambda\rangle) E(t) - e^{-it\langle \Lambda\rangle -i\Theta(t, -i\nabla)} f_*(\infty)\big\|
        \to 0,
        \quad
        \text{as } t\to +\infty,
    \end{align*}
    where $\Lambda$ is defined in Section~\ref{sec:pre}.
    Moreover, in this case we show in Proposition~\ref{prop:Theta} that for every fixed $\xi \in \mathbb{R}^2$,
\begin{align*}
\lim_{t\to+\infty} |\Theta(t, \xi)| = +\infty.
\end{align*}     
\end{remark}

\begin{remark}
The scattering analysis for the component $n$ is carried out in physical space.
To determine the scattering behavior of $E$, we primarily rely on the $Z$-norm method, in particular the framework developed in \cite{Ionescu-P2}, together with various estimates established in the proof of Theorem~\ref{thm:main1}.
This approach has proven to be robust and well-suited for problems of dispersive type.
\end{remark}

\begin{proposition}[Energy cascade]\label{prop:cascade}
    If \eqref{eq:n1=+} holds, then there exists $C>0$ such that for $t\gg 1$ we have
    \begin{align}
        \|n\| \geq C \log^{1\over 2} t \, \Big|\int_{\mathbb{R}^2} n_1 \, \d x\Big|,
        \qquad
        \|\partial n\| \leq C \epsilon.
    \end{align}
\end{proposition}
Its proof can be found in Appendix \ref{sec:App-2}.
\begin{remark}
    The results in Proposition \ref{prop:cascade} indicate the energy cascade of the component $n$ from high frequency to low frequency as time evolves.
\end{remark}

\subsection{Relevant results revisited}

\

The Klein-Gordon–Zakharov system originates from plasma physics and serves as an important model for laser–plasma interactions. It couples a wave equation with a Klein-Gordon equation, reflecting the interplay between low-frequency ion-acoustic waves and high-frequency Langmuir waves. The system blends aspects of classical plasma dynamics with relativistic field theory. Because of its rich mathematical structure and challenging nonlinear interactions, it has attracted significant attention over the past decades, particularly within the dispersive and hyperbolic PDE community.

In the 1980s, the breakthrough works of Klainerman \cite{Klainerman85, Klainerman852}, Shatah \cite{Shatah}, and Christodoulou \cite{Christodoulou} established that in $\mathbb{R}^{1+3}$, nonlinear wave equations with null structure and nonlinear Klein–Gordon equations admit global smooth solutions for smooth, small, and localized initial data.
The vector-field method \cite{Klainerman85, Klainerman852}, the normal-form method \cite{Shatah}, and the conformal compactification method \cite{Christodoulou} developed in these works have since become fundamental tools for the study of nonlinear wave-type equations; the first two play a central role in the present paper.

Following these advances for pure wave and Klein–Gordon equations, attention turned to coupled systems in $\mathbb{R}^{1+3}$. The first global results in this direction were due to Bachelot \cite{Bachelot} for the Dirac–Klein–Gordon system and Georgiev \cite{Georgiev} for the wave–Klein–Gordon system. Subsequent works include Ozawa–Tsutaya–Tsutsumi \cite{OTT, OTT2}, Katayama \cite{Katayama12a}, and Shi–Wang \cite{WangShu} on the Klein–Gordon–Zakharov system; Psarelli \cite{Psarelli}, Yang–Yu \cite{YaYu19}, and Fang–Klainerman–Wang–Yang \cite{Klainerman-QW-SY, FaWaYa} on the Maxwell–Klein–Gordon equations; Ionescu–Pausader \cite{Ionescu-P, Ionescu-P2}, LeFloch–Ma \cite{PLF-YM-cmp, PLF-YM-arXiv1}, and Wang \cite{Wang} on the Einstein– Klein– Gordon equations; and Huneau–Stingo–Wyatt \cite{Zoe23} on the Kaluza–Klein model, etc. Even in $\mathbb{R}^{1+3}$, proving global well-posedness for these systems is highly nontrivial. A key idea in many proofs is to exploit sharp pointwise decay of solutions. In lower dimensions, however, both wave and Klein–Gordon components decay more slowly, making the analysis more difficult; in $\mathbb{R}^{1+2}$, very few results about these models are known.

Before turning to the literature in $\mathbb{R}^{1+2}$, we note that the Klein–Gordon–Zakharov system, the Maxwell–Klein–Gordon system in Lorenz gauge, and the Einstein–Klein– Gordon system in wave gauge share structural similarities:
each is a coupled wave–Klein–Gordon system featuring $KG \times KG$ nonlinearities in the wave equations and $wave \times KG$ nonlinearities, without partial derivatives on the wave factor, in the Klein–Gordon equations.
This structure is particularly challenging. For example, the $KG \times KG$ terms in the wave equations typically force the wave component to decay at most like $t^{-1}$, with its $L^2$ norm growing like $t^{1\over 2}$; this in turn induces a logarithmic growth in the Klein–Gordon energy, obstructing a straightforward global argument.

In $\mathbb{R}^{1+2}$, the decay rates for the linearized system are already slow ($t^{-{1\over 2}}$ for the wave component and $t^{-1}$ for the Klein–Gordon component), so the nonlinear analysis requires greater delicacy.
For the pure Klein–Gordon equation, the pioneering result of Ozawa–Tsutaya–Tsutsumi \cite{OTT1} (1996) established global small-data solutions in two dimensions.
For the pure wave equation, a breakthrough was achieved by Alinhac \cite{Alinhac1, Alinhac2} in 2001, with more recent progress by Dong–LeFloch–Lei \cite{DoLeLe21} and Li \cite{Li21}.
In 2017, Ma \cite{MaH2} studied a wave–Klein–Gordon system; subsequent works have extended the range of admissible nonlinearities, including but not limited to Dong \cite{Dong2005}, Dong–Wyatt \cite{DoWy24}, Ifrim–Stingo \cite{Ifrim-Stingo}, Ma \cite{Ma19, Ma2008}, Stingo \cite{Stingo}, and Zhang \cite{Zh25}.

Returning to the Klein–Gordon–Zakharov system, several important results have been established. In $\mathbb{R}^{1+3}$, nonlinear stability was proved by Ozawa–Tsutaya– Tsutsumi \cite{OTT}; later on, for the case where the speeds are strictly different small data global well-posedness in the energy space was obtained in \cite{OTT2} and scattering was obtained under radial symmetry in \cite{Guo-N-W, Guo-N-W2}. In this case, due to the low-regularity setting, very different and delicate difficulties arise. In the celebrated works \cite{MaNa, MaNa1}, Masmoudi-Nakanishi studied the nonrelativistic limit, showing convergence from the Klein–Gordon–Zakharov system to the Zakharov system, a model of coupled wave and Schrödinger equations; see also the work of Colin-Ebrard-Gallice-Texier \cite{Colin}.
On the other hand in $\mathbb{R}^{1+2}$, due to the slow decay of the solution, little is known until the recent global existence result of Dong \cite{Dong2006}. In this result, the initial data are assumed to be compactly supported and of divergence form. Subsequent efforts by Duan–Ma \cite{Duan-Ma}, Dong–Ma \cite{DoMa}, and Cheng \cite{Cheng} have sought to relax these restrictions. Regarding scattering, only partial results are currently available in \cite{DoMa, Cheng}.

There exist many important models in plasma physics that are closely related to the Klein-Gordon-Zakharov model, among which is the Euler-Maxwell system, a basic model in plasma physics, as well as its models. In $\mathbb{R}^{1+3}$, the global well-posedness of the two-fluid Euler–Maxwell model was established in the celebrated work \cite{GuIoPa} by Guo-Ionescu-Pausader. In $\mathbb{R}^{1+2}$, the Euler-Poisson model and later the two-fluid Euler–Maxwell model were shown to admit global smooth solutions by Li–Wu \cite{Li-Wu14} and Ionescu-Pausader \cite{Ionescu-P4} and by Deng–Ionescu-Pausader \cite{DeIoPa}, respectively. These two models can be viewed as Klein–Gordon-type equations with nonlocal nonlinearities, and a common challenge, also as presented in our problem, is the weak decay of solutions.

\subsection{Key ingredients and structure of proof}

Our strategy is to first establish sharp estimates for the solution $(n, E)$ using a physical space approach, and then, based on these estimates, to derive the scattering of $(n, E)$ via a frequency space method.

In the course of proving Theorems \ref{thm:main1} and \ref{thm:main2}, another two main ideas play crucial roles in the proof.

\medskip
\noindent
\textbf{Main idea 1: Novel nonlinear transformation}

For the wave component $n$, its nonlinearity is $\Delta|E|^2$, whose decay rate is at best
\[
\|\Delta |E|^2\| \lesssim t^{-1},
\]
which is non-integrable in time and prevents global existence and scattering. To circumvent this obstacle, we reveal a hidden structure via a novel nonlinear transformation
\begin{equation}\label{eq:nonl-trans}
    \widetilde{n} = n + {1\over 4} \Delta |E|^2,
\end{equation}
leading to faster-decaying nonlinearities for the new unknown $\widetilde{n}$. More details can be found in Section \ref{sec:scatter}.

It is worth mentioning that the new nonlinear transformation in Section \ref{sec:scatter} can be used to study other wave-type systems with nonlinear terms of divergence form.

\medskip
\noindent
\textbf{Main idea 2: Perturbative view on nonlinearity $nE$.}

For $E$, we rewrite the equation as
\[
-\Box E + (1+n)E = 0,
\]
regarding the nonlinear term $nE$ as a perturbation of the mass term. This viewpoint effectively reduces the nonlinear Klein--Gordon equation to one that is close to its linear counterpart. It plays a vital role in deriving sharp estimates for the field $E$ and showing nonlinear scattering of $E$. 

\medskip
\noindent
\textbf{Other useful ideas.}

We use the decomposition
\[
n = \ell + \Delta \mathfrak{m},
\]
inspired by Katayama~\cite{Katayama12a}, where $\ell$ and $\mathfrak{m}$ solve
\begin{align*}
& -\Box \ell = 0, 
\qquad\qquad (\ell, \partial_t \ell)(t_0) = (n_0, n_1), 
\\
& -\Box \mathfrak{m} = |E|^2, 
\qquad\qquad (\mathfrak{m}, \partial_t \mathfrak{m})(t_0) = (0,0).
\end{align*}
The system \eqref{eq:2D-KGZ} then becomes
\begin{equation}\label{eq:2D-KGZ-2}
\aligned
-\Box \ell &= 0, \\
-\Box \mathfrak{m} &= |E|^2, \\
-\Box E + E &= -\ell E - \Delta \mathfrak{m}\, E.
\endaligned
\end{equation}
This decomposition is useful for obtaining sharp estimates on the component ~$n$.

Another important observation is that the nonlinear term $nE$ satisfies a Klein-Gordon equation with nonlinearities that decay faster than itself. This enables us to estimate $\partial_{\xi} \widehat{nE}$ directly in Section \ref{sec:S-E}, without using more complicated bilinear estimates. 




Below, we demonstrate the structure of proof with more detailed discussions.

\textbf{Step 1: Global existence.}

We apply the physical space method of (hyperboloidal) vector-field to obtain global existence with sharp decay and energy bounds. This approach, initiated by Klainerman~\cite{Klainerman85, Klainerman852, Klainerman93} for single wave and single Klein--Gordon equations, has since been developed further to treat more wave-type equations. Foundational contributions were made by H\"ormander \cite{Hormander}, Psarelli \cite{Psarelli} and more recently by LeFloch-Ma \cite{PLF-YM-cmp}, Klainerman-Wang-Yang \cite{Klainerman-QW-SY}, among others; see also the work of Tataru \cite{Tataru} on waves in hyperbolic space.

Besides the well-known challenges of weak decay and lack of symmetry, lack of partial derivatives on the wave component $n$ in the nonlinear term $nE$ causes a serious problem: the natural wave energy controls only $\partial n$, not $n$ itself. A Hardy inequality could in principle bridge this gap, but in $\mathbb{R}^{1+2}$ it fails unless one includes a logarithmic correction, which worsens the decay.

The strategy is to treat the interior and exterior regions separately: we first establish the estimates in the exterior region and then address the interior region. This approach to nonlinear wave-type equations was employed in the pioneering work of Christodoulou–Klainerman \cite{ChKl}, and more recently by Huneau–Stingo–Wyatt \cite{Zoe23}.



\textbf{Step 1.1: Exterior region.}

In the exterior region $\{ r \ge t - \frac{3}{2} \}$, we establish global existence along with pointwise and energy bounds. A key tool is the weighted Hardy inequality in Proposition~\ref{prop:Hardy-ex}, which controls $n$ in terms of~$\partial n$ at the expense of weighted norms on the initial data.

One important observation is that the energy on the boundary $\{r= t-1 \}$ has a positive sign, which will be used when treating the interior region.


\textbf{Step 1.2: Interior region.}

In this case, it seems impossible to rely on extra weights on the initial data to improve decay of the solution. To illustrate the key ideas, we first look at the equation of $E$. The best decay we can expect for the nonlinear term is
\[
\|nE\| \lesssim t^{-1},
\]
which is borderline non-integrable. To overcome this, we treat $nE$ as a perturbation of the mass term, allowing us to ignore it when conducting energy estimates.  
When applying higher-order (weighted) derivatives, the Leibniz rule produces many nonlinear terms. We adopt this strategy to treat the term where all (weighted) derivatives hit on the $E$ part, which turns out to be the worst term in the nonlinearities, and inductively we improve the energy bounds at all orders of (weighted) derivatives.

For the component $n$, we use the decomposition~\eqref{eq:2D-KGZ-2} together with the weighted energy estimates in Proposition~\ref{prop:EE-in} to obtain uniform-in-time bounds, albeit with unfavorable $(t-r)$ weights. Then, Proposition~\ref{prop:wave-extra}, combined with extra partial derivatives, removes these $(t-r)$ weights and closes the energy estimates.



\textbf{Step 2: Energy bounds on $t$-slices.}

Next, we transfer the energy estimates from curved hyperboloidal slices to constant-$t$ slices. This is done by applying divergence theorem on specific spacetime domains based on energy and pointwise estimates prepared in Sections \ref{sec:exterior} and \ref{sec:interior}. The sharp Sobolev norm bounds in \eqref{eq:thm-Sobolev} are derived in this part.

\textbf{Step 3: Scattering.}


In this part, based on the energy bounds on $t$-slices and pointwise estimates of the solution $(n, E)$, we mainly rely on the frequency space method of Z-norm (see for instance \cite{DeIoPa, Ionescu-P, Ionescu-P2, Ionescu-P3, Ionescu-P4}) and of spacetime resonance introduced by Germain-Masmoudi-Shatah and Gustafson-Nakanishi-Tsai (see for instance \cite{GeMaSh, GuNaTs}). 

\textbf{Step 3.1: Linear scattering of $n$.}

Due to the weak decay of nonlinearity in the $n$-equation $\Delta|E|^2$, we rely on the novel nonlinear transformation \eqref{eq:nonl-trans} to derive scattering of $n$. 
Based on the estimates prepared in \textbf{Step 2}, we show that the new unknown $\widetilde{n}$ scatters linearly. Since $n$ differs from this new unknown only by lower-order terms, $n$ also scatters linearly.



\textbf{Step 3.2: Modified scattering of $E$.}

If~\eqref{eq:n1=+} holds, the quantity $\widehat{n_1}(\eta)$ does not vanish at low frequencies, obstructing linear scattering of $E$. Following the seminal work~\cite{Ionescu-P2} and adopting the perturbation-of-mass perspective, we introduce a phase correction function $\Theta(t,\xi)$ and then estimate the residuals for the modified profile, where we decompose these residuals dyadically both in frequency and in time to get fine control of these terms. 

One part that requires delicate analysis lies in the proof of Proposition \ref{prop:Theta}, where we need to investigate the asymptotic behavior of $\Theta(t, \xi)$ to confirm that modified scattering is indeed the correct behavior. In the proof, we rely on technical analysis with the aid of properties enjoyed by the Bessel function of the first kind of order zero.




\textbf{Step 3.3: Linear scattering of $E$.}

Under the assumption~\eqref{eq:n1=0}, Lemma~\ref{lem:n1-low} yields a vanishing property of $\widehat{n_1}(\eta)$ in the low frequency regime. This allows us to bound the term related to $\Theta(t,\xi)$ that previously obstructed linear scattering, thereby recovering linear scattering for~$E$.








\subsection{Comparison with other spacetime dimensions}

As we will see in the proof in Section \ref{sec:interior}, one big challenge is to understand the behavior of $\|(L_1 n)E\|$ with $L_1 = t\partial_{x^1} + x^1\partial_t$ defined in Section \ref{sec:pre}.

To gain insight, we consider in $\mathbb{R}^{1+d}$, with a point denoted by $(t, x^1, \cdots, x^d)$, the linear homogeneous wave and Klein-Gordon equations
\begin{align*}
-\Box w = &0,
\quad
&(w, \partial_t w)(t_0)
=
(w_0, w_1),
\\
-\Box v + v = &0,
\quad
&(v, \partial_t v)(t_0)
=
(v_0, v_1).
\end{align*}
For simplicity, we assume $(w_0, w_1, v_0, v_1)$ all belong to $C^\infty_0$  with support in the unit ball $\{ |x|=\sqrt{(x^1)^2 + \cdots+ (x^d)^2}\leq 1\}$. The argument below also works for small, smooth initial data that decay sufficiently fast as $|x|\to +\infty$.

\subsubsection{High dimensions: $d\geq 3$.}

\

Using the rough bounds
\begin{align*}
    \|L_1 w\|\lesssim 1,
    \qquad
    |v(t, x)|\lesssim \langle t+|x|\rangle^{-{d\over 2}},
\end{align*}
we obtain
\begin{align*}
    \|(L_1 w) v \|
    \lesssim
    \|L_1 w\| \|v\|_{L^\infty}
    \lesssim
    \langle t\rangle^{-{3\over 2}}.
\end{align*}
The decay rate $\langle t\rangle^{-{3\over 2}}$ is integrable in time, which is favorable for establishing global well-posedness and scattering in the corresponding nonlinear problems.

\subsubsection{Low dimensions: $d=1$.}

\

In $\mathbb{R}^{1+1}$, one in general does not expect any integrable bound for $\|(L_1 w) v\|$ as the solutions only have weak (or no) decay rates
\begin{align*}
    |L_1 w(t, x)|\lesssim 1,
    \qquad
    |v(t, x)|\lesssim \langle t+|x|\rangle^{-{1\over 2}}.
\end{align*}
However, $(1+1)-$dimensional spacetime setting has special structural features that allow for sharper bounds.

First, the Klein–Gordon component enjoys the following (non-trivial) improved decay estimate (cf. Proposition~\ref{prop:KG-extra} for the two-dimensional analogue)
\begin{align*}
    |v(t, x)|
    \lesssim
    \langle t+|x|\rangle^{-{5\over 2}} \langle t-|x|\rangle^2,
\end{align*}
which indicates that 
\begin{align*}
    |v(t, x)|
    \lesssim
    \langle t+|x|\rangle^{-{5\over 2}},
    \qquad
    \text{in } \{|x|\geq t-2\}.
\end{align*}

Next, the d’Alembert formula gives
\begin{align*}
&L_1 w(t,x)
\\
= 
&\frac{1}{2} \Big( (x + t)\left(w_0'(x + t) + w_1(x + t)\right)
+ (x - t)\left(-w_0'(x - t) + w_1(x - t)\right) \Big),
\end{align*}
by which we see $L_1 w$ is supported in $\{|t-x| \leq 1, \text{ or } |t+x| \leq 1\}$. This implies
\begin{align*}
    \|(L_1 w) v\|
    \lesssim
    \langle t\rangle^{-{3\over 2}},
\end{align*}
an integrable decay rate that is, perhaps, unexpected in $(1+1)-$dimensions.

\subsubsection{Low dimensions: $d=2$.}

\

In $\mathbb{R}^{1+2}$, the situation is more subtle. In this case, we have the  bounds
\begin{align*}
    \|L_1 w\| \lesssim 1,
    \qquad
    |v(t, x)|
    \lesssim
    \langle t+|x|\rangle^{-3} \langle t-|x|\rangle^2.
\end{align*}
Due to the fact that $L_1 w$ is supported in the full interior $\{|x|\leq t+1 \}$ not only near the light cone,
consequently, we obtain
\begin{align}
    \|(L_1 w) v\|
    \lesssim
    \langle t\rangle^{-1},
\end{align}
which is non-integrable and also creates serious difficulties in the proof of Section~\ref{sec:interior}.

\subsubsection{Summary}

\

Based on the computations above, we observe that the term $\|(L_1 w) v\|$ is integrable in time in all spacetime dimensions except $\mathbb{R}^{1+2}$. While this does not imply that the Klein-Gordon–Zakharov system is necessarily more difficult in $\mathbb{R}^{1+2}$ than in $\mathbb{R}^{1+1}$, since our analysis here concerns only the linear equations for $w$ and $v$, it does suggest (as indeed turns out to be the case) that the system in $\mathbb{R}^{1+2}$ presents unique challenges and requires a more delicate analysis. Moreover, we conjecture that the dichotomy in scattering behavior observed in Theorem \ref{thm:main2} may be a phenomenon specific to the two-dimensional case.

\subsection{Perspectives}

\subsubsection{A cousin problem: Zakharov system}

The Zakharov system, a cousin system to the Klein-Gordon-Zakharov system, can be regarded as a system coupling two types of fundamental equations of wave and Schrodinger. There exist extensive studies on this system.

The nonlinearities in the Zakharov system are of the same form as those in the Klein-Gordon-Zakharov system. In addition, these two systems share many similar properties: the Schrodinger field enjoys the same time decay rate as the Klein-Gordon field in the same spacetime dimensions.

\subsubsection{Other wave-type dispersive equations: Maxwell-Klein-Gordon system}

As we discussed before, there are lots of similarities between the Klein-Gordon-Zakharov system and the Maxwell-Klein-Gordon system. Considering the weak decay of the solution to the Maxwell-Klein-Gordon system, it is a very challenging problem to study its long-time behavior. Our method with a particular choice of gauge might be useful to treat Maxwell-Klein-Gordon system in $\mathbb{R}^{1+2}$.

\subsubsection{Low regularity scattering}

An important and open problem is the scattering behavior of the Klein–Gordon–Zakharov system with initial data of low regularity, such as data in the energy space. This is a different problem from the one studied in the present paper. In $\mathbb{R}^{1+3}$, the low-regularity scattering for the radial case has been tackled by Guo–Nakanishi–Wang \cite{Guo-N-W}. However, to the best of our knowledge, the non-radial case in $\mathbb{R}^{1+3}$ remains unaddressed, and no results are available for the case in $\mathbb{R}^{1+2}$.

Our findings in Theorem~\ref{thm:main2} suggest that different scattering phenomena may arise in $\mathbb{R}^{1+2}$ compared to the three-dimensional setting. The nonlinear transformation introduced in Section~\ref{sec:scatter} may prove useful in tackling the low-regularity problem; say, for $(n_0, n_1, E_0, E_1) \in H^{1}(\mathbb{R}^2)\times L^2(\mathbb{R}^2)\times H^{2}(\mathbb{R}^2) \times H^{1}(\mathbb{R}^2)$.

\subsection{Organisation}

In Section~\ref{sec:pre}, we introduce the basic notation, develop various weighted energy estimates, and review fundamental properties of the wave-type components. In Sections~\ref{sec:exterior} and~\ref{sec:interior}, we establish the global existence of solutions along with uniform energy bounds and sharp time decay estimates for the Klein-Gordon-Zakharov system in the exterior and interior regions, respectively. Building on these results, we prove the linear scattering of $n$ in Section~\ref{sec:scatter} using physical space analysis with a key novel nonlinear transformation, and the (non)linear scattering of $E$ in Section~\ref{sec:S-E}, relying primarily on frequency space methods. Several technical arguments are deferred to the appendix.


\section{Preliminaries}\label{sec:pre}
\subsection{(Hyperboloidal) Vector-field setting}

We work in the $\mathbb{R}^{1+2}$ Minkowski spacetime with metric $m=\text{diag}(-1, 1, 1)$. For a spacetime point $(t, x_1, x_2)$, we denote $x^0 = t$ and $r=\sqrt{(x_1)^2 + (x_2)^2}$. The indices are raised or lowered by the metric $m$, and for repeated upper and lower indices we adopt the Einstein summation convention.

We recall the classical vector fields \cite{Klainerman852}.
\begin{itemize}
\item Translations: $\partial_\alpha = \partial_{x^\alpha}$, 
\qquad
$\alpha \in \{ 0, 1, 2\}$.

\item Rotation: $\Omega = \Omega_{12} = x_1\partial_2 - x_2\partial_1$.

\item Lorentz boosts: $L_a = x_a \partial_t + t\partial_a$,
\qquad
$a\in\{ 1, 2\}$.
\end{itemize}
In the sequel, we use Greek letters $\alpha, \beta, \cdots \in\{0, 1, 2\}$ to denote spacetime indices and Roman letters $a, b, \cdots \in \{ 1, 2\}$ to represent space indices.

We use $\Gamma$ to denote an arbitrary vector field in the set
$$
V = \{  \partial_0, \partial_1, \partial_2, \Omega_{12}, L_1, L_2  \},
$$
while we adopt the same convention for $\partial \in \{\partial_0, \partial_1, \partial_2\} $ and $L \in \{ L_1, L_2  \}$.

We also recall the scaling vector field
$$
L_0 = t\partial_t + x^a \partial_a,
$$
which is compatible with pure wave equations, and the good derivatives from ghost weight method \cite{Alinhac1}
\begin{align*}
    G = (G_1,  G_2),
\end{align*}
with $G_a = {\partial_a} + {x_a\over r}\partial_t$.

We have for a smooth function $\phi(t, x)$ that (see for instance \cite{Hormander, Sogge, LiZh})
\begin{equation}\label{eq:wave-partial}
\begin{aligned}
\langle t-r\rangle |\partial \phi|
\lesssim
&|\Gamma \phi| + |L_0 \phi|,
\\
|\Omega_{12} \phi |
\lesssim
&|L \phi|,
\qquad
\text{for  }|x|\leq 2t.
\end{aligned}
\end{equation}


We will rely on a hyperboloidal foliation to treat the interior region $\{ |x|\leq t-1 \}$ of the spacetime $\mathbb{R}^{1+2}$, and now revisit the vector-field method adapted to this particular foliation. We set
\begin{align*}
\tau_+ = t+|x|,
\qquad
\tau_- = t - |x|,
\qquad
\tau = \sqrt{\tau_+ \tau_-}.
\end{align*}
We use $\mathcal{H}_\tau := \{ (t, x) : t^2 - |x|^2 = \tau^2, \, |x|\leq t-1 \}$ to denote a hyperboloid with hyperbolic time $\tau \geq \tau_0=1$ truncated by the cone $\{ |x|=t-1 \}$. 
We use $\mathcal{D}_{[\tau_1, \tau_2]}:= \{ (t, x) : \tau_1^2 \leq t^2 - |x|^2 \leq \tau_2^2, \, |x|\leq t-1 \}$ to represent the spacetime region limited by two hyperboloids $\mathcal{H}_{\tau_1}, \mathcal{H}_{\tau_2}$ and the cone $\{ |x|=t-1 \}$. Within the region $\mathcal{D}_{[\tau_0, +\infty)}$, we remind one of the useful relations
\begin{align*}
|x|\leq t-1,
\qquad
\tau \leq t \leq \tau^2,
\qquad
t\leq \tau_+ \leq 2t.
\end{align*}

We denote
\begin{align*}
\cancel{\partial}_0 = {\tau \over t} \partial_0,
\qquad
\cancel{\partial}_a = {L_a \over t} = \partial_a + {x_a\over t} \partial_0.
\end{align*}
We find that
\begin{align*}
\partial_0 = {t\over \tau} \cancel{\partial}_0,
\qquad
\partial_a = \cancel{\partial}_a - {x_a\over \tau} \cancel{\partial}_0.
\end{align*}

Given a multi-index $I=(I_1, \cdots, I_6) \in \mathbb{N}^6$, we denote 
\begin{align*}
    \Gamma^I
    =
    \partial_0^{I_1} \partial_1^{I_2} \partial_2^{I_3} \Omega_{12}^{I_4} L_1^{I_5} L_2^{I_6},
\end{align*}
and $\partial^J$ with $J=(J_1, J_2, J_3 )\in \mathbb{N}^3$ and $L^K$ with $K=(K_1, K_2)\in \mathbb{N}^2$ are defined in a similar way.

\subsection{Basics in Fourier analysis}

We define the Fourier transform of a nice function $f$ as
$$
\mathcal{F}f(\xi) = \widehat{f}(\xi) = \int_{\mathbb{R}^2} f(x) \, e^{-ix \cdot \xi} \, \d x,
$$
with its inverse formula
$$
f(x) = \mathcal{F}^{-1} \widehat{f} = {1\over 4\pi^2} \int_{\mathbb{R}^2} \widehat{f}(\xi) e^{ix \cdot \xi} \, \d\xi.
$$
The Plancherel’s identity is given by
$$
\| f\| = {1\over 2\pi} \| \widehat{f} \|.
$$
We have the derivative rule 
$$
\widehat{\partial_a f} = i \xi_a \widehat{f},
\qquad
\widehat{x_a f} = i \partial_{\xi_a} \widehat{f}. 
$$
We denote $\Lambda$ such that $\widehat{\Lambda f}(\xi) = |\xi| \widehat{f}(\xi)$ where $f$ is a sufficiently nice function.

We recall the Littlewood-Paley decomposition. Select a smooth radial decreasing (in $r$) function $\psi: \mathbb{R}^2\to [0, 1]$ with
\begin{equation}
\psi(x)
=
\begin{cases}
       1, 
     \qquad
     &|x|\in [-1, 1],
     \\[1ex]
      0,
     \qquad
     &|x|\in  [-2, 2]^c.
\end{cases}
\end{equation}
For any integer $k\in\mathbb{Z}$ and interval $I\subseteq \mathbb{R}$, define 
$$
\phi(x) = \psi({x}) - \psi(2x),
\qquad
\phi_k(x) = \phi({x\over 2^k}),
\qquad
\phi_{I}(x) = \sum_{l\in I\cap \mathbb{Z}} \phi_l(x).
$$
By a slight abuse of notation, $\psi$ and $\phi$ can also be regarded as one-dimensional functions.
One notes that 
$$
\sum_{k\in\mathbb{Z}} \phi_k(x) 
 =1,
 \qquad
 \text{for } |x|\neq 0.
$$
Define 
$$
P_k = \phi_k({\Lambda}),
\qquad
P_{I} = \phi_{I}(\Lambda),
$$
which are interpreted by
$$
\widehat{P_k f}(\xi) = \phi_k(\xi) \widehat{f}(\xi),
\qquad
\widehat{P_{I} f}(\xi) = \phi_{I}(\xi) \widehat{f}(\xi).
$$

We denote 
$$
k^- = \min \{k, 0 \},
\qquad
k^+ = \max\{k, 0\},
$$
for $k\in\mathbb{Z}$. One checks that $k\leq l$ implies $k^-\leq l^-$ and $k^+\leq l^+$.

Following Ionescu-Pausader \cite{Ionescu-P, Ionescu-P2}, let 
\begin{equation}\label{eq:J}
\mathcal{J} = \{ (l, j)\in \mathbb{Z}\times \mathbb{Z}: l+j\geq 0, \, j\geq 0  \}.
\end{equation}
For $(l, j)\in \mathcal{J}$, we define
\begin{eqnarray}
\widetilde{\phi}^{(l)}_j(x)
=
\left\{
\begin{array}{lll}
      \phi_{\leq j}(x),
     \qquad
     &\text{for }l+j=0 \text{ and } l\leq 0,
     \\[1ex]
      \phi_{\leq j}(x),
     \qquad
     &\text{for }j=0 \text{ and } l\geq 0,
     \\[1ex]
      \phi_j(x),
        \qquad
        &\text{for }l+j\geq 1 \text{ and } j\geq 1.
\end{array}
\right. 
\end{eqnarray}
One easily checks that $\sum_{j\geq -\min\{l, 0\}} \widetilde{\phi}_j^{(l)}=1$ for any given $l\in \mathbb{Z}$.

The following result in \cite[Lemma 5.1]{Ionescu-P2} will be used in the proofs of Lemma \ref{lem:d-xi-L2} and Proposition \ref{prop:M3}. 
\begin{lemma}[\cite{Ionescu-P2}]\label{lem:d-xi}
Let $w_+, v_+$ solve
$$
(\partial_t + i\Lambda)w_+ = \mathcal{N}_1,
\qquad\qquad
(\partial_t + i\langle\Lambda\rangle)v_+ = \mathcal{N}_2,
$$
respectively.
Let $g_+ = e^{it|\nabla|} w_+, f_+ = e^{it\langle\nabla\rangle} v_+$.
    Then, we have
    \begin{equation}
    \begin{aligned}
    \widehat{L_a w_+}
=
&|\xi| e^{-it|\xi|} \partial_{\xi_a} \widehat{g_+}(\xi)
        +{\xi_a \over |\xi|} e^{-it|\xi|} \widehat{g_+}(\xi)
        +i\partial_{\xi_a} \widehat{N}_1,
        \\
    \widehat{L_a v_+}
=
&\langle\xi\rangle e^{-it\langle\xi\rangle} \partial_{\xi_a} \widehat{f_+}(\xi)
        +{\xi_a \over \langle\xi\rangle} e^{-it\langle\xi\rangle} \widehat{f_+}(\xi)
        +i\partial_{\xi_a} \widehat{N}_2.    
    \end{aligned}
    \end{equation}

\end{lemma}

The following phase functions will appear frequently in Section \ref{sec:S-E}
\begin{align*}
    \Phi_{1\pm} 
    = &\langle \xi\rangle + \langle \xi-\eta\rangle \pm |\eta|,
    \\
    \Phi_{2\pm} 
    = &\langle \xi\rangle - \langle \xi-\eta\rangle \pm |\eta|.
\end{align*}
We have the following estimates for the phase functions $\Phi_{1\pm}, \Phi_{2\pm}$.
\begin{lemma}\label{lem:phase}
    The following four groups of estimates hold.

    \begin{enumerate}
        \item 
        \begin{align*}
        |\Phi_{1\pm}| 
        \gtrsim 
        {1\over \langle \xi\rangle} + {1\over \langle \xi-\eta\rangle},
        \qquad\quad
        |\Phi_{2\pm}| 
        \gtrsim
        {|\eta| \over (\langle\eta\rangle + \langle \xi-\eta\rangle)^2}.
        \end{align*}

        \item 
        \begin{align*}
        |\nabla_\eta \Phi_{1\pm}| 
        \gtrsim 
        \langle \xi-\eta\rangle^{-2},
        \qquad\quad
        |\nabla_\eta \Phi_{2\pm}| 
        \gtrsim
        \langle \xi-\eta\rangle^{-2}.
        \end{align*}

        \item 
        \begin{align*}
        |\nabla_\eta \nabla_\eta \Phi_{1\pm}| 
        +
        |\nabla_\eta \nabla_\eta \Phi_{2\pm}| 
        \lesssim
        {1\over \langle\xi-\eta\rangle} + {1\over |\eta|}.
        \end{align*}

        \item 
        \begin{align*}
        |\nabla_\eta \nabla_\eta \nabla_\eta \Phi_{1\pm}| 
        +
        |\nabla_\eta \nabla_\eta \nabla_\eta \Phi_{2\pm}| 
        \lesssim
        {1\over \langle\xi-\eta\rangle^2} + {1\over |\eta|^2}.
        \end{align*}

    \end{enumerate}
\end{lemma}
\begin{proof}
    The proof follows from elementary but tedious calculations, and we omit it.
\end{proof}

\subsection{Commutator estimates}

We first recall, see for instance \cite{LiZh, Sogge, Hormander}, the commutators between $\Box$ and $\Gamma$
\begin{align*}
    [-\Box, \Gamma] =0,
    \qquad
    [-\Box + 1, \Gamma] = 0,
    \qquad
    [-\Box, L_0] = C \Box.
\end{align*}
In addition, given a smooth function $f(t,x)$ and $\Gamma_1, \Gamma_2 \in V$, the following holds:
\begin{equation}
\begin{aligned}
    &|[\partial, \Gamma] f|
    \lesssim
    |\partial f|,
    \qquad
    |[{\Gamma_1}, {\Gamma_2}]f|
    \lesssim
    |\Gamma f|;
    \\
    &|\partial(\tau/t)|
    \lesssim
    \tau^{-1},
    \qquad
    |L(\tau/t)|
    \lesssim
    \tau/t, \quad\text{  for } r\leq t-1.
\end{aligned}
\end{equation}

\begin{lemma}\label{lem:comm-Fourier}
Let $w_+ = (\partial_t - i\Lambda) w$ and $v_+ = (\partial_t - i \langle \Lambda\rangle) v$.
Then, it holds that
\begin{align}
\| L_a w_+ - (\partial_t - i\Lambda) L_a w\|_{H^7}
\lesssim &\|\partial w\|_{H^7},
\\
\| L_a v_+ - (\partial_t - i\langle\Lambda\rangle) L_a v\|_{H^7}
\lesssim &\|\partial v\|_{H^7}.
\end{align}
\end{lemma}
\begin{proof}
We only prove the estimate regarding $w_+$.

We note 
\begin{align*}
    L_a w_+ - (\partial_t - i\Lambda) L_a w
    =
    -\partial_a w - ix_a \partial_t \Lambda w + i\Lambda (x_a \partial_t w).
\end{align*}
In frequency space, we find
\begin{align*}
    \mathcal{F}(x_a \partial_t \Lambda w - \Lambda (x_a \partial_t w))
    =
    i\partial_{\xi_a} \partial_t (|\xi| \widehat{w})
    - i |\xi| (\partial_{\xi_a} \partial_t \widehat{w})
    =
    i{\xi_a \over |\xi|} \partial_t \widehat{w}.
\end{align*}
    Therefore, we arrive at
    \begin{align}
        \|L_a w_+ - (\partial_t - i\Lambda) L_a w\|
        \leq 2 \|\partial w\|.
    \end{align}

    The higher order cases can be estimated in a similar way, and therefore we complete the proof.
\end{proof}

\subsection{Various energy estimates for wave-type equations}

We consider a general inhomogeneous wave-type equation
\begin{align}\label{eq:wave-type-eqn}
-\Box \phi + m^2 \phi = F.
\end{align}

Let $\eta_0 \in [1, {5\over 4}]$ and  denote 
\begin{align*}
\Sigma^{ex}_t = \{ (s, x) : s=t, \,r\geq s-\eta_0 \},
\qquad
&\Sigma^{bd}_t = \{(s, x): t_0\leq s\leq t,\, r=s-\eta_0 \},
\\
\mathcal{D}^{ex}_t = \{ (s, x) : &t_0\leq s\leq t, \,r\geq s-\eta_0 \}.
\end{align*}
We define the exterior energy at time $t\geq t_0$ by
\begin{equation} 
\begin{aligned}
\mathcal{E}^{ex}_m (\phi, t, \gamma)
:=
&\int_{\Sigma^{ex}_t}  \langle r-t\rangle^{2\gamma} \big( |\partial \phi|^2 + m^2 |\phi|^2 \big) \, \d x
\\
&+
\gamma \int_{t_0}^t \int_{\Sigma^{ex}_s} \langle r-s\rangle^{2\gamma-1} \big( |G \phi|^2 + m^2 |\phi|^2 \big) \, \d x \d s
\\
&+
\int_{\Sigma^{bd}_t} \big( \sum_{a=1}^2(\partial_a \phi + \frac{x_a}{r}\partial_t \phi)^2 +m^2 |\phi|^2  \big) \d S.
\end{aligned}
\end{equation}

\begin{proposition}[Weighted energy estimates in the exterior regions I]\label{prop:EE-ex}
Let $\phi$ solve \eqref{eq:wave-type-eqn}, and $\gamma \geq 0$. Then we have
\begin{align}
\mathcal{E}^{ex}_m (\phi, t, \gamma)
\lesssim
\mathcal{E}^{ex}_m (\phi, t_0, \gamma)
+
\int_{\mathcal{D}^{ex}_t} \langle r-s\rangle^{2\gamma} \big| F \cdot \partial_t \phi \big| \, \d x \d s.
\end{align}

\end{proposition}
\begin{proof}
Multiplying by $(r-t+\frac 32)^{2\gamma} \partial_t \phi$ on both sides of \eqref{eq:wave-type-eqn} and Leibniz rule, we can get 
\begin{align*}
& \partial_t \bigg(\frac 12 (r-t+\frac 32)^{2\gamma}(|\partial \phi|^2 + m^2 |\phi|^2) \bigg)-\partial^{i}\big((r-t+\frac 32)^{2\gamma}\partial_i \phi \partial_t \phi  \big) 
\\
&  +\gamma (r-t+\frac 32)^{2\gamma-1} \big(\sum_{a=1}^2|G_a\phi|^2 + m^2\phi^2 \big) 
\\
= &(r-t+\frac 32)^{2\gamma}F \partial_t \phi.
\end{align*}
We integrate this identity over the region $\mathcal{D}^{ex}_t$ and apply the divergence theorem to derive desired energy estimates.
\end{proof}

\begin{proposition}[Weighted energy estimates in the exterior regions II]\label{prop:EE-ex2}
Let $\phi$ solve \eqref{eq:wave-type-eqn} with $m=1$, and $\gamma \geq 0$. Then we have
\begin{equation}
\begin{aligned}
\mathcal{E}^{ex}_1 (\phi, t, \gamma)
\lesssim
&\mathcal{E}^{ex}_1 (\phi, t_0, \gamma)
+
\int_{\mathcal{D}^{ex}_t} \langle r-t\rangle^{2\gamma} \big| (F+h\phi) \cdot \partial_t \phi \big| \, \d x \d s
\\
&+
\int_{\mathcal{D}^{ex}_t} \langle r-t \rangle^{2\gamma} \big| |\phi|^2  \partial_t h  \big| + \gamma \langle r-t\rangle^{2\gamma-1} \big| h  |\phi|^2 \big| \, \d x \d s,
\end{aligned}
\end{equation}
in which $h$ is a smooth function with $|h|\ll {1\over 10}$.
\end{proposition}

\begin{proof}
Since $\phi$ solves \eqref{eq:wave-type-eqn}, we have 
\begin{align*}
\partial_t^2 \phi - \Delta \phi +  \phi + h \phi = F+ h\phi.
\end{align*}
Then multiplying by $(r-t+\frac 32)^{2\gamma} \partial_t \phi$ on both sides of the above equation, one can obtain
\begin{align*}
& \partial_t \Big(\frac 12 (r-t+\frac 32)^{2\gamma}\big(|\partial \phi|^2 + (1+h) |\phi|^2\big) \Big)-\partial^{i}\big((r-t+\frac 32)^{2\gamma}\partial_i \phi \partial_t \phi  \big) \\
& \qquad +\gamma (r-t+\frac 32)^{2\gamma-1} \Big(\sum_{a=1}^2|G_a\phi|^2 + |\phi|^2 \Big)  \\ 
& = (r-t+\frac 32)^{2\gamma}(F+h\phi) \partial_t \phi + \frac 12  (r-t+\frac 32)^{2\gamma} |\phi|^2 \partial_t h  - \gamma (r-t+\frac 32)^{2\gamma-1} h |\phi|^2.
\end{align*}
Then one can argue similarly as that in Proposition \ref{prop:EE-ex} to get the desired result.  The proof is done.
\end{proof}

We turn to the energy in the interior region $\{ r\leq t-1 \}$, and we always take $\eta_0 =1$ in this region. For a function $\phi$ defined on $\mathcal{H}_\tau$, we define its weighted energy by
\begin{equation}
\begin{aligned}
&\mathcal{E}^{in}_m (\phi, \tau, \gamma)
\\
:=
&\int_{\mathcal{H}_\tau} \tau_-^{-2\gamma} \Big( \sum_\alpha |\cancel{\partial}_\alpha \phi|^2 + m^2\phi^2 \Big) \, \d x
+
\gamma \int_{\tau_0}^\tau \int_{\mathcal{H}_{\widetilde{\tau}}} (\widetilde{\tau}/t) { |G \phi |^2 +m^2 \phi^2 \over \tau_-^{1+2\gamma}}   \, \d x \d \widetilde{\tau},
\end{aligned}
\end{equation}
in which $\gamma\geq 0$ and
\begin{align*}
\int_{\mathcal{H}_\tau} \psi(t, x) \, \d x
:=
\int_{\mathbb{R}^2} \psi(\sqrt{\tau^2 + |x|^2}, x) \, \d x,
\end{align*}
for a smooth function $\psi(t, x)$ defined on $\mathcal{H}_\tau$.
We use the notation for $p \in [1, +\infty)$
\begin{align}
\| \psi \|_{L^p(\mathcal{H}_\tau)}
=
\Big( \int_{\mathcal{H}_\tau} |\psi(t, x)|^p \, \d x \Big)^{1\over p}.
\end{align}
For simplicity, we tacitly denote
\begin{align}
    \mathcal{E}^{in}(\phi, \tau, \gamma)
    =
    \mathcal{E}^{in}_0(\phi, \tau, \gamma)
\end{align}
as well as
\begin{align}
    \mathcal{E}^{in}_m(\phi, \tau)
    =
    \mathcal{E}^{in}_m(\phi, \tau, 0).
\end{align}
One also checks that
\begin{align}
    \big\| \tau_-^{-\gamma} (\tau/t) \partial \phi\big\|_{L^2(\mathcal{H}_\tau)}^2
    \lesssim
    \mathcal{E}^{in}_m(\phi, \tau, \gamma).
\end{align}

\begin{proposition}[Weighted energy estimates in the interior regions I]\label{prop:EE-in}
Let $\phi$ solve \eqref{eq:wave-type-eqn}. Then, we have (below $\gamma \geq 0$)
\begin{equation}
\begin{aligned}
\mathcal{E}^{in}_m (\phi, \tau, \gamma)
\lesssim
&\mathcal{E}^{in}_m (\phi, \tau_0, \gamma)
+
\int_{\tau_0}^\tau \big\|\tau_-^{-2\gamma} F \cdot (\widetilde{\tau} / t) \partial_t \phi \big\|_{L^1(\mathcal{H}_{\widetilde{\tau}})} \, \d \widetilde{\tau}
\\
&+
\int_{\Sigma^{bd}_{t^{bd}}}  \sum_a ({x_a \over r} \partial_t \phi + \partial_a \phi)^2 + m^2 \phi^2  \, \d S,
\end{aligned}
\end{equation}
in which $t^{bd} = {\tau^2+1\over 2}$.
\end{proposition}
\begin{proof}
One derives the following identity
\begin{equation}\label{eq:div1}
 \begin{aligned}   
&{1\over 2} \partial_t \Big( \tau_-^{-2\gamma} \big( |\partial \phi|^2 + m^2 \phi^2 \big) \Big)
-
\partial_a \big( \tau_-^{-2\gamma} \partial^a \phi \partial_t \phi  \big)
\\
&+
\gamma \tau_-^{-2\gamma-1} \Big( \sum_a \big( {x_a \over r }\partial_t \phi + \partial_a \phi \big)^2 + m^2 \phi^2 \Big)
=
\tau_-^{-2\gamma}  F\partial_t \phi,
\end{aligned}
\end{equation}
and the proof follows from the divergence theorem.
\end{proof}

\begin{proposition}[Weighted energy estimates in the interior regions II]\label{prop:EE-in2}
Let $\phi$ solve \eqref{eq:wave-type-eqn} with $m=1$, and $\gamma \geq 0$. Then we have
\begin{equation}
\begin{aligned}
\mathcal{E}^{in}_1 (\phi, \tau, \gamma)
\lesssim
&\mathcal{E}^{in}_1 (\phi, \tau_0, \gamma)
+
\int_{\tau_0}^\tau \big\| \tau_-^{-2\gamma}  (F+h\phi) \cdot (\widetilde{\tau} / t)\partial_t \phi \big\|_{L^1(\mathcal{H}_{\widetilde{\tau}})} \,  \d \widetilde{\tau}
\\
&+
 \int_{\tau_0}^\tau \big\| \tau_-^{-2\gamma}  |\phi|^2  (\widetilde{\tau} / t)\partial_t h  \big\|_{L^1(\mathcal{H}_{\widetilde{\tau}})} 
+ \gamma \big\| \tau_-^{-2\gamma-1}  (\widetilde{\tau} / t) h  |\phi|^2 \big\|_{L^1(\mathcal{H}_{\widetilde{\tau}})}  \,  \d \widetilde{\tau}
\\
&+
\int_{\Sigma^{bd}_{t^{bd}}}  \sum_a ({x_a \over r} \partial_t \phi + \partial_a \phi)^2 + |\phi|^2  \, \d S,
\end{aligned}
\end{equation}
in which $t^{bd} = {\tau^2+1\over 2}$ and $h$ is a smooth function with $|h|\ll {1\over 10}$.
\end{proposition}
\begin{proof}
From
\begin{align*}
\partial_t^2 \phi - \Delta \phi +  \phi + h \phi = F+ h\phi,
\end{align*}
one gets
\begin{equation}\label{eq:div2}
 \begin{aligned}  
& \partial_t \Big(\frac 12 \tau_-^{-2\gamma}\big(|\partial \phi|^2 + (1+h) |\phi|^2\big) \Big)-\partial^{i}\big(\tau_-^{-2\gamma}\partial_i \phi \partial_t \phi  \big) \\
& \qquad +\gamma \tau_-^{-2\gamma-1} \Big(\sum_{a=1}^2|G_a\phi|^2 + \phi^2 \Big)  \\ 
& = \tau_-^{-2\gamma}(F+h\phi) \partial_t \phi + \frac 12  \tau_-^{-2\gamma} \phi^2 \partial_t h  - \gamma \tau_-^{-2\gamma-1} h \phi^2,
\end{aligned}
\end{equation}
and the proof follows from the divergence theorem.
\end{proof}

\begin{proposition}[Conformal energy estimates for free wave] \label{prop:con-EE}
Consider \eqref{eq:wave-type-eqn} with $m=0, F=0$. Then one has
\begin{align}
\mathcal{E}_{con} (\phi, \tau)
\leq
2\mathcal{E}_{con} (\phi, 1),
\end{align}
in which (with $K_0 = (t^2+r^2) \partial_t + 2rt \partial_r$)
\begin{align*}
\mathcal{E}_{con} (\phi, \tau)
=
&\int_{\mathcal{H}_\tau} \big( (K_0 \phi /t) + \phi  \big)^2 + (\tau / t)^2 \sum_a (L_a \phi)^2 \, \d x
\\
&+
\int_{\{ t={1+\tau^2 \over 2}, \, r\geq 0 \}} (L_0 \phi + \phi)^2 + \sum_a (L_a \phi)^2 + (\Omega \phi)^2 \, \d x.
\end{align*}
\end{proposition}

\begin{proof}
Multiplying $2(K_0 \phi + t\phi)$ on both sides of \eqref{eq:wave-type-eqn} with $m=0, F=0$, one obtains
\begin{align*}
&2(K_0 \phi + t\phi) (-\Box \phi)
\\
=
&\partial_t \big((r^2+t^2)|\partial \phi|^2 + 4rt\partial_t \phi \partial_r \phi + 2t\phi \partial_t\phi - \phi^2 + \partial_a(x^a \phi^2)\big)
\\
&+
2 \partial^a \Big( tx_a \big(-(\partial_t \phi)^2 + \partial_b \phi \partial^b \phi\big) - (r^2+t^2)\partial_t \phi \partial_a \phi -2rt \partial_r \phi \partial_a \phi 
\\
&\qquad -t\phi \partial_a \phi - {1\over 2} \partial_t (x_a \phi^2)\Big).
\end{align*}
We first integrate the above identity over the spacetime region $\mathcal{D}_{[\tau_0, \tau]} \bigcup \{ t_0\leq t\leq {\tau^2 + 1\over 2},\, r\geq t-1  \}$ and  $\{ t_0\leq t\leq {\tau^2 + 1\over 2},\, r\geq 0  \}$, respectively. Then, applying the divergence theorem and taking the summation of these two estimates to get the desired result. 
\end{proof}

\begin{remark}
In the interior region, one notes that
\begin{align*}
{1\over t} K_0 
=
L_0 + {x^a \over t} L_a,
\end{align*}
and therefore, one has
\begin{align}
(K_0 \phi /t) + \phi
=
L_0 \phi + {x^a \over t} L_a \phi + \phi.
\end{align}
This means that we can bound $L_0 \phi$ once we have estimates on $\phi$ as the bounds on $L_a \phi$ already appear in the conformal energy.
\end{remark}

\subsection{Additional structure of wave-type equations}

\begin{proposition}[\cite{Klainerman93}]\label{prop:KG-extra}
Consider the wave-type (Klein-Gordon) equation \eqref{eq:wave-type-eqn} with $m=1$.
For $t>0$ it holds
\begin{align}
|\phi| \lesssim
{\langle\tau_-\rangle \over \langle t \rangle} |\partial \partial \phi| + {1\over \langle t \rangle} |\partial L \phi| + {1\over \langle t \rangle} |\partial \phi| + |F|,
\qquad
\text{for } r \leq 3t.
\end{align}

\end{proposition}

\begin{proof}
Considering $t\geq 1$ first,  we can rewrite the wave operator $-\Box$ as 
\begin{align}\label{reprewave}
-\Box = \frac{\tau_- \tau_+}{t^2} \partial_t\partial_t -\frac{1}{t} \partial^a L_a + \frac{x^a}{t^2} \partial_t L_a + \frac{2}{t}\partial_t -\frac{x^a}{t^2}\partial_a.
\end{align}
Then following equation \eqref{eq:wave-type-eqn} with $m=1$, one can see
\begin{align*}
\phi = \Box \phi + F.
\end{align*}
This combining with \eqref{reprewave} and $r\leq 3t$ implies the desired result. 
While for $0 < t\leq 1$,  we can use the equation \eqref{eq:wave-type-eqn} to get
\begin{align*}
|\phi| \leq |\Box \phi| +|F| \leq |\partial \partial \phi| + |F|.
\end{align*}
The proof is completed.
\end{proof}

\begin{proposition}[\cite{MaH2}]\label{prop:wave-extra}
Consider the wave-type (wave) equation \eqref{eq:wave-type-eqn} with $m=0$.
For $t>0$ it holds
\begin{align}
|\partial \partial\phi| \lesssim
 {1\over \langle  \tau_-\rangle } |\partial \Gamma \phi| + {1\over \langle \tau_- \rangle} |\partial \phi| +  {t\over \langle \tau_-\rangle}|F|,
\qquad
\text{for } |x| \leq 3t.
\end{align}

\end{proposition}

\begin{proof}
For $|t-r|\leq 1$, it is trivial. For $|t-r|\geq 1$, one concludes from the identity \eqref{reprewave}.
\end{proof}

\subsection{Sobolev-type inequality}

We first recall the standard Klainerman-Sobolev inequality in $\mathbb{R}^{1+2}$.
\begin{proposition}[\cite{Hormander, Sogge, LiZh}]\label{prop:KS}
    Let $\phi = \phi(t, x)$ be a smooth function decaying sufficiently fast at space infinity for every $t \geq 1$. 	Then it holds 
\begin{equation}
\langle \tau_+\rangle^{1\over 2} \langle \tau_-\rangle^{1\over 2} |\phi(t,x)|
\lesssim 
\sum_{|I|+ |J|\leq 2}\| \Gamma^I L_0^J \phi(t,\cdot)\|.
\end{equation}
\end{proposition}

We will mainly apply the following version of Sobolev-type inequality to get decay for a function when restricted to the exterior region (say $\{r\geq t-1\}$) .
\begin{proposition}\label{prop:Sobolev-ex}
Let $\phi = \phi(t, x)$ be a smooth function decaying sufficiently fast at space infinity for every $t \geq 1$. 	Then it holds 
\begin{equation}
\langle r\rangle^{1\over 2} \big|\langle \tau_-\rangle^{\eta} \phi(t,x)\big|
\lesssim 
\sum_{|I|\leq 1, |J|\leq 1}\big\|\langle \tau_-\rangle^{\eta} Z^I \Omega^J \phi(t,\cdot)\big\|,
\qquad
\eta\in \mathbb{R},
\end{equation}
in which $Z$ represents  the vector fields in $\{\partial_r,\Omega = x^1\partial_2-x^2\partial_1\}$.
\end{proposition}

To treat the interior region $\{ r\leq t-1\}$, we need a version of Klainerman-Sobolev inequality, which is stated now.
\begin{proposition}[\cite{Hormander, Psarelli}]\label{prop:Sobolev-in}
Let $\phi = \phi(t, x)$ be a smooth function defined on $\mathcal{H}_\tau$. 	Then it holds 
\begin{equation}
\sup_{\mathcal{H}_\tau} |(1+t) \phi(t, x)|
\lesssim 
\sup_{t={\tau^2 + 1\over 2}, \, |x|={\tau^2-1 \over 2}} |\phi(t, x)| + \sum_{|I|\leq 2}\| L^I \phi\|_{L^2(\mathcal{H}_\tau)},
\end{equation}
in which $L$ represents  the Lorentz boosts in $\{L_a = x_a \partial_t + t \partial_a\}_{a=1, 2}$.
\end{proposition}
Its proof can be found in the appendix.

We will use the following Hardy-type inequality to handle the wave component $n$ in exterior region. This is motivated by the three-dimensional version of weighted Hardy inequality in \cite{Zoe23}.
\begin{proposition}\label{prop:Hardy-ex}
Let $\gamma > 0$, then we have
\begin{align}
\big\| \langle \tau_-\rangle^\gamma \phi(t, x) \big\|_{L^2(\Sigma^{ex}_t)}
\lesssim
{1\over \gamma} \big\| \langle \tau_-\rangle^{\gamma+1} \nabla \phi(t, x) \big\|_{L^2(\Sigma^{ex}_t)}.
\end{align}
\end{proposition}
The proof can be found in the appendix.

\subsection{Estimates for linear homogeneous wave}

We consider the simple case of a linear homogeneous wave equation
\begin{align*}
    -\Box \phi = 0.
\end{align*}
We have the following result.
\begin{lemma}[\cite{LiZh}]\label{lem:freewave}
The following $L^2$ type and pointwise estimates hold.
\begin{enumerate}
\item \emph{$L^2$ bounds.}
\begin{align}
\|\Gamma^I L_0^J \phi\|(t) 
\lesssim
&\big( \|\Gamma^I L_0^J \phi\|(t_0) + \|\partial_t \Gamma^I L_0^J \phi\|_{L^1 \cap L^2}(t_0)  \big) \log \langle t\rangle.  \label{eq:wave-L-L2}
\end{align}

\item \emph{Pointwise bounds.}
\begin{align}
| \phi(t, x)|
\lesssim
&\sum_{|I|+|J|\leq 2}\big( \|\Gamma^I L_0^J \phi\|(t_0) + \|\partial_t \Gamma^I L_0^J \phi\|_{L^1\cap L^2}(t_0)  \big)  \langle \tau_+\rangle^{-{1\over 2}} \langle \tau_-\rangle^{-{1\over 2}} \log \langle t\rangle.  \label{eq:wave-L-point}
\end{align}
\end{enumerate}
\end{lemma}

\begin{proof}
    The proof of \eqref{eq:wave-L-L2} is based on the following formula in frequency space
    \begin{align*}
        \partial_t \partial_t \widehat{\phi}(t, \xi) + |\xi|^2 \widehat{\phi}(t, \xi) = 0,
    \end{align*}
    with solution
    $$
    \widehat{\phi}(t, \xi) 
    =
    \cos((t-t_0)|\xi|) \widehat{\phi}(t_0, \xi) 
    + 
    {\sin((t-t_0)|\xi|)\over |\xi|} \partial_t \widehat{\phi}(t_0, \xi).
    $$
    Note that
    \begin{align*}
        &\Big\| {\sin((t-t_0)|\xi|)\over |\xi|} \partial_t \widehat{\phi}(t_0, \xi)  \Big\|^2
        \\
        \lesssim
        &\|\partial_t {\phi}(t_0, x)\|^2 + \|\partial_t {\phi}(t_0, x)\|^2_{L^1} \int_{|\xi|\leq 1} {\sin^2((t-t_0)|\xi|)\over |\xi|^2} \, \d\xi,   
    \end{align*}
and that
\begin{align*}
    &\int_{|\xi|\leq 1} {\sin^2((t-t_0)|\xi|)\over |\xi|^2} \, \d\xi
    \\
    \lesssim
    &\int_{0}^{1} {\sin^2((t-t_0)|\xi|)\over |\xi|} \, \d|\xi|
    \lesssim
    \int_{0}^{t} {\sin^2(\rho)\over \rho} \, \d\rho
    \lesssim
    \log \langle t\rangle.
\end{align*}
These lead to the desired result \eqref{eq:wave-L-L2}.

    The proof of \eqref{eq:wave-L-point} follows from \eqref{eq:wave-L-L2} and the standard Klainerman-Sobolev inequality in Proposition \ref{prop:KS}.
\end{proof}

\subsection{Pointwise estimates for Klein-Gordon}

Suppose $\phi$ satisfies the Klein-Gordon equation
\begin{align}
-\Box \phi + \phi = F,
\end{align}
and we want to derive pointwise bounds for $\phi$ adapted to our problem, i.e., the equation of $E$. We only focus on the interior region $\mathcal{D}_{[\tau_0, +\infty)} = \{(t, x) : |x|\leq t-1, \, t^2 - |x|^2 \geq \tau_0^2 \}$. This part is based on earlier work of \cite{Klainerman85, PLF-YM-cmp, DoWy}.

We define integral curves $\mathcal{V}$  parameterized by $\lambda$ as:
\begin{equation} \label{eq:integral-curve}
\begin{aligned}
{dt\over d\lambda} = {t\over \tau},
\\
{dx^a \over d\lambda} = {x^a \over \tau}.
\end{aligned}
\end{equation}
One derives that
\begin{equation}
\begin{aligned}
&\tau(\lambda) = \lambda + \tau(0),
\\
&t(\lambda) = {t(0) \over \tau(0)} (\lambda + \tau(0)),
\\
&x_a(\lambda) = {x_a(0) \over \tau(0)} (\lambda + \tau(0)),
\end{aligned}
\end{equation}
in which $\tau(0) = \sqrt{t^2(0) - x_1^2(0) - x_2^2(0)} \geq \tau_0$ .

We will consider functions along this integral curve.

\begin{lemma}\label{lem:ODE-00}
Along the integral curve $\mathcal{V}$, the solution $\phi(t(\lambda), x(\lambda))$ satisfies the following second-order ODE:
\begin{align}
{d^2\over d\lambda^2} \psi(t, x) + \psi
=
\tau \cancel{\partial}_a \cancel{\partial}^a \phi + {x^a x^b \over \tau} \cancel{\partial}_a \cancel{\partial}_b \phi
+
2{x^a \over \tau} \cancel{\partial}_a \phi + \tau F,
\end{align}
in which $\psi = \tau \phi$.
\end{lemma}

\begin{proof}

We note along the integral curve $\mathcal{V}$ that 
\begin{align}\label{eq:d-lambda}
{d\over d\lambda} h(t, x)
=
{L_0 \over \tau} h
=
(\cancel{\partial}_0 + {x^a \over \tau} \cancel{\partial}_a ) h,  
\end{align}
as well as
\begin{align*}
{d^2\over d\lambda^2} h(t, x)
=
(\cancel{\partial}_0 + {x^a \over \tau} \cancel{\partial}_a )^2 h
=
(\cancel{\partial}_0 \cancel{\partial}_0  + 2{x^a \over \tau} \cancel{\partial}_0 \cancel{\partial}_a + {x^a x^b \over \tau^2} \cancel{\partial}_a \cancel{\partial}_b  )h.
\end{align*}

On the other hand,  we rewrite the wave operator in the frame $\{ \cancel{\partial}_\alpha \}$
\begin{align*}
-\Box 
=
\cancel{\partial}_0 \cancel{\partial}_0 + 2{x^a \over \tau} \cancel{\partial}_0 \cancel{\partial}_a - \cancel{\partial}_a \cancel{\partial}^a + {2\over \tau} \cancel{\partial}_0.
\end{align*}
Set $\psi = \tau \phi$, and we have
\begin{align*}
{d^2\over d\lambda^2} \psi(t, x)
=
(\cancel{\partial}_0 \cancel{\partial}_0  + 2{x^a \over \tau} \cancel{\partial}_0 \cancel{\partial}_a + {x^a x^b \over \tau^2} \cancel{\partial}_a \cancel{\partial}_b  )\psi.
\end{align*}
After a tedious computation, we find
\begin{align*}
\tau (-\Box \phi)
=
\tau (-\Box {\psi \over \tau})
=
\cancel{\partial}_0 \cancel{\partial}_0 \psi
+
2{x^a \over \tau} \cancel{\partial}_0 \cancel{\partial}_a \psi 
- \cancel{\partial}_a \cancel{\partial}^a \psi - 2{x^a \over \tau^2} \cancel{\partial}_a \psi.
\end{align*}
Therefore, we get
\begin{align*}
{d^2\over d\lambda^2} \psi(t, x)
=
&\cancel{\partial}_a \cancel{\partial}^a \psi + {x^a x^b \over \tau^2} \cancel{\partial}_a \cancel{\partial}_b \psi
+
2{x^a \over \tau^2} \cancel{\partial}_a \psi
+
\tau (-\Box \phi)
\\
=
&\tau \cancel{\partial}_a \cancel{\partial}^a \phi + {x^a x^b \over \tau} \cancel{\partial}_a \cancel{\partial}_b \phi
+
2{x^a \over \tau} \cancel{\partial}_a \phi
+
\tau (-\Box \phi).
\end{align*}

Employing the equation of $\phi$, we derive that
\begin{align}
{d^2\over d\lambda^2} \psi(t, x) + \psi
=
\tau \cancel{\partial}_a \cancel{\partial}^a \phi + {x^a x^b \over \tau} \cancel{\partial}_a \cancel{\partial}_b \phi
+
2{x^a \over \tau} \cancel{\partial}_a \phi + \tau F.
\end{align}

The proof is completed.
\end{proof}

We need to use the following ODE estimate, which can be found in \cite{DoWy}.
\begin{lemma}
\label{lem:ODE-01}
Consider the second-order ODE
\begin{equation}\label{eq:ode1} 
\begin{aligned}
& z''(\lambda) + \big( 1 - G(\lambda) \big) z(\lambda) =  k(\lambda),
\\
& z(0) = z_0, \quad z'(0) =  z_1, \quad  | G(\lambda) | \leq {1\over 10},
\end{aligned}
\end{equation}
in which $k$ is assumed to be integrable, then we have the following pointwise estimate
\begin{equation}\label{eq:ode2}
\begin{aligned} 
&\big( (z')^2(\lambda) + (1 - G(\lambda)) z^2(\lambda) \big)^{1\over 2} 
\\
\lesssim 
& \big( (z')^2(0) + z^2(0) \big)^{1\over 2} + \int_{0}^\lambda \big|k(\widetilde{\lambda})\big| + \big| G'(\widetilde{\lambda})  z(\widetilde{\lambda})\big| \, \d\widetilde{\lambda }.
\end{aligned}
\end{equation}
\end{lemma}

\begin{proof}
We set $Y(\lambda) = \big((z')^2(\lambda) + (1 - G(\lambda)) z^2(\lambda) \big)^{1\over 2}$, and then by multiplying $z'(\lambda)$ in \eqref{eq:ode1}, we get
\begin{align*}
{d \over d \lambda} Y^2(\lambda) 
&= 2 z'(\lambda) k(\lambda) - G'(\lambda) z^2(\lambda)
\\
&\leq 2 Y(\lambda) \big( |k(\lambda)| + | G'  z(\lambda)| \big).
\end{align*}
In order to proceed, we divide $Y(\lambda)$ in the above inequality and, integrate to get
\begin{align*}
Y(\lambda) \leq Y(0) + \int_{0}^\lambda \big( \big|k(\widetilde{\lambda})\big| + \big| G'  z(\widetilde{\lambda})\big| \big) \, \d \widetilde{\lambda}.
\end{align*}

The proof is thus completed.
\end{proof}

\begin{proposition}\label{prop:E-sharp-decay}
Let $h(t, x)$ be a smooth function with small amplitude $|h|\leq {1\over 10}$. Then we have
\begin{equation} \label{eq:E-sharp-decay}
\begin{aligned} 
&|\tau \phi(t, x)|
\\
\lesssim
&|\tau(0) \phi(t(0), x(0))| + | L_0 \phi(t(0), x(0))|
\\
 & + 
  \int_{0}^\lambda |\tau \cancel{\partial}_a \cancel{\partial}^a \phi + {x^a x^b \over \tau} \cancel{\partial}_a \cancel{\partial}_b \phi
+
2{x^a \over \tau} \cancel{\partial}_a \phi| + |\tau F -h \tau \phi| + | {d\over d\lambda} h  \tau \phi(t, x)| \, \d\widetilde{\lambda}.
\end{aligned}
\end{equation}
\end{proposition}
\begin{proof}
The proof follows from Lemmas \ref{lem:ODE-00}--\ref{lem:ODE-01}.
\end{proof}

\subsection{Properties of Bessel function}

We denote $J_0(s)$ the Bessel function of the first kind of order zero, which solves the differential equation
\begin{align*}
    s^2 J_0^{''}(s) + s J_0^{'}(s) + s^2 J_0(s) = 0.
\end{align*}

We will rely on the following property of this function, which can be found in \cite[Chapter 13.4]{Wa}.
\begin{lemma}\label{lem:Bessel}
Let $a>b>0$, then it holds that
\begin{align}
    \int_0^{+\infty} \sin(as) J_0(bs) \, \d s
    =
     (a^2 - b^2)^{-{1\over 2}}.
    \end{align}
\end{lemma}

\begin{lemma}\label{lem:Bessel2}
    For $s\gg 1$, one has
    \begin{align*}
        \big|J_0(s) - \sqrt{2\over s \pi } \cos\big(s-{\pi \over 4}\big) \big|
        \lesssim
        s^{-{3\over 2}}.
    \end{align*}
\end{lemma}
\begin{proof}
    Its proof can be found in \cite[Chapter 7]{Wa}.
\end{proof}

\section{Global existence: the exterior region}\label{sec:exterior}
This section is devoted to proving global existence for the Klein-Gordon-Zakharov system \eqref{eq:2D-KGZ} in the exterior region $\bigcup_{t\geq t_0} \Sigma^{ex}_t$ valid for all $\eta_0 \in [1, {5\over 4}]$. We recall that in Theorem \ref{thm:main1} we assumed suitable decay for the initial data which is mainly due to the multiple use of Lorentz boosts for the Klein-Gordon component $E$. For ease of reading, we will ignore this issue and use Lorentz boosts as many times as we want. Without loss of generality, we might regard the field $E$ as a scalar-valued function which can be treated in the same way as it is the vector-valued one.

The proof  relies on a bootstrap argument. We carefully formulate bootstrap assumptions for the solution over a given interval and demonstrate that these estimates can be strictly improved. This iterative refinement of the bootstrap assumptions ultimately establishes the global existence of solutions to the Klein-Gordon-Zakharov system.     Here, the bootstrap assumptions for $t\in [t_0, T)$ are as follows:
\begin{align}
\|\partial^I\Gamma^J n \|_{L^2(\Sigma^{ex}_t)}
\leq 
&C_1 \epsilon t^{\delta},
\qquad
&|I|+|J| \leq N+1,   \label{eq:BA-ex}  
\\
\mathcal{E}^{ex}_1 (\partial^I\Gamma^J E, t, \gamma_{e_1})
\leq 
&(C_1 \epsilon)^2,
\qquad
&|I|+|J|\leq N, \label{eq:BA-ex01}     
\\
\mathcal{E}^{ex}_1 (\partial^I\Gamma^J E, t, \gamma_{e_2} )
\leq 
&(C_1 \epsilon)^2,
\qquad
&|I|+|J|\leq N+1.   \label{eq:BA-ex02}  
\end{align}

In the above, $\epsilon \ll 1$ measures the size of the initial data, $\epsilon \ll \delta \ll 1$, and  $C_1$ is a large number such that $C_1 \epsilon \ll 1$ to be fixed. In the sequel, the implied constants in $\lesssim$ should be independent of $\epsilon$ and $C_1$. The parameters with admissible ranges $\gamma_{e_1} \geq {3\over 2}, \gamma_{e_2} \geq {1\over 2}+\delta$ are fixed, and we take $\gamma_{e_1} = {3\over 2}, \gamma_{e_2} = {1\over 2}+\delta$.

After employing the Sobolev-type inequality and commutator estimates, we get pointwise bounds for the solution $(E, n)$ for $t \in [t_0, T)$.

\begin{proposition}
Under the bootstrap assumptions in \eqref{eq:BA-ex}--\eqref{eq:BA-ex01}, the following pointwise estimates hold.
\begin{enumerate}
		\item \emph{Rough decay for $n$.}
\begin{align}
|\partial^I\Gamma^J n | 
\lesssim
C_1 \epsilon \langle t+r\rangle^{-{1\over 2} + \delta},
\qquad
|I|+|J| \leq N-1. \label{eq:n-rough} 
\end{align}

\item \emph{Rough decay for $E$.}
\begin{align}
|\partial \partial^I\Gamma^J E| + |\partial^I\Gamma^J E|
\lesssim
C_1 \epsilon \langle t+r\rangle^{-{1\over 2}} \langle t-r\rangle^{-\gamma_{e_1}},
\qquad
|I|+|J| \leq N-2. 
\end{align}
\end{enumerate}
\end{proposition}

Acting $\partial^I \Gamma^J$ to \eqref{eq:2D-KGZ}, we have
\begin{align}\label{eq:KG-high}
-\Box \partial^I \Gamma^J E + \partial^I \Gamma^J E = - \partial^I \Gamma^J (n E).
\end{align}

\begin{proposition}[Refined decay for $E$] \label{prop:KG-refine01}
It holds that
\begin{align}\label{eq:KG-refine01}
|\partial^I\Gamma^J E|
\lesssim
C_1 \epsilon \langle t+r\rangle^{-{1\over 2}-\gamma_{e_1}},
\qquad
|I|+|J| \leq N-5. 
\end{align}
\end{proposition}

\begin{proof}
Note that 
\begin{align*}
\langle t+r\rangle
\simeq 
\langle t-r\rangle,
\qquad
r\geq 2t,
\end{align*}
thus we only consider the region $\{ r \leq 2t \}$.

By \eqref{eq:KG-high} with $|I| + |J| = 0$ and Proposition \ref{prop:KG-extra}, we have
\begin{align}
|E| \lesssim
C_1 \epsilon \langle t+r \rangle^{-{3\over 2}} \langle t-r \rangle^{-\gamma_{e_1}+1} + |nE|.
\end{align}
The smallness of $n$ yields
\begin{align}
|E| \lesssim
C_1 \epsilon \langle t+r \rangle^{-{3\over 2}} \langle t-r \rangle^{-\gamma_{e_1}+1}.
\end{align}

Inductively, we get
\begin{align}\label{eq:KG-p00}
|\partial^I \Gamma^J E|
\lesssim
C_1 \epsilon \langle t+r \rangle^{-{3\over 2}} \langle t-r \rangle^{-\gamma_{e_1}+1},
\qquad
|I| + |J|\leq N-3.
\end{align}

We repeat this process with refined pointwise \eqref{eq:KG-p00} to arrive at
\begin{align}\label{eq:KG-p01}
|\partial^I \Gamma^J E|
\lesssim
C_1 \epsilon \langle t+r \rangle^{-{5\over 2}} \langle t-r \rangle^{-\gamma_{e_1}+2},
\qquad
|I| + |J|\leq N-5.
\end{align}

Finally, we interpolate \eqref{eq:KG-p00} and \eqref{eq:KG-p01} to get \eqref{eq:KG-refine01}.

\end{proof}

\subsection{Refined bounds}

Recall we decompose $n = \ell + \Delta \mathfrak{m}$. The component $\ell$ satisfies a free wave equation, so its estimates are relatively easier to get. 

\begin{lemma}\label{lem:n0}
We have the following bounds for $\ell$.
\begin{enumerate}
\item \emph{$L^2$ bounds.}
\begin{align}
\big\|\langle r-t\rangle^{\gamma_{e_1}} \partial^I \Gamma^J \ell\big\|_{L^2( \Sigma^{ex}_t)}
\lesssim
\epsilon,
\qquad
|I| + |J| \leq N,  \label{eq:n0-L2}
\\
\| \partial^I \Gamma^J \ell\|
\lesssim
\epsilon \log \langle t\rangle ,
\qquad
|I| + |J| \leq N+1.  \label{eq:n0-L2-a}
\end{align}

\item \emph{Pointwise bounds in $\Sigma^{ex}_t$.}
\begin{align}
|\partial^I \Gamma^J \ell |
\lesssim
\epsilon \langle t+r\rangle^{-{1\over 2}}  \langle t-r\rangle^{-\gamma_{e_1}},
\qquad
|I| + |J| \leq N-2.  \label{eq:n0-point}
\end{align}

\end{enumerate}

\end{lemma}
\begin{proof}
We first derive bounds for $\partial \partial^I \Gamma^J \ell$ with $|I| + |J| \leq N$. We get by applying Proposition \ref{prop:EE-ex} that
\begin{align*}
\mathcal{E}^{ex}(\partial^I \Gamma^J \ell, t, \gamma_{e_1}+1)
=
\mathcal{E}^{ex}(\partial^I \Gamma^J \ell, t_0, \gamma_{e_1}+1)
\lesssim
\epsilon^2.
\end{align*}
Then, the Hardy-type inequality in Proposition \ref{prop:Hardy-ex} with above estimates leads to \eqref{eq:n0-L2}.

For the estimate in \eqref{eq:n0-L2-a}, one refers to Proposition \ref{prop:con-EE} and Lemma \ref{lem:freewave}.  

Finally, we can derive pointwise bounds \eqref{eq:n0-point} using Sobolev inequality in Proposition \ref{prop:Sobolev-ex} and commutator estimates.
\end{proof}

Now, we derive the bounds for the component $\mathfrak{m}$. We act $\partial^I \Gamma^J $ to the equation of $\mathfrak{m}$ and get
\begin{align}\label{eq:wave-high}
-\Box \partial^I \Gamma^J \mathfrak{m}
=
\partial^I \Gamma^J |E|^2.
\end{align}

\begin{lemma}\label{lem:nDelta-ex}
We have
\begin{equation}\label{eq:nDelta-ex}
\begin{aligned}
\mathcal{E}^{ex} (\partial \partial^I\Gamma^J \mathfrak{m}, t, \gamma_{e_1})
\lesssim
&\epsilon^2 + (C_1 \epsilon)^4,
\qquad
|I| + |J|\leq N+1,
\\
\mathcal{E}^{ex} ( \partial^I\Gamma^J \mathfrak{m}, t, \gamma_{e_2})
\lesssim
&\epsilon^2 + (C_1 \epsilon)^4,
\qquad
|I| + |J|\leq N+2, \, |J|\leq N+1. 
\end{aligned}
\end{equation}
\end{lemma}

\begin{proof}
By the energy estimate in Proposition \ref{prop:EE-ex}, we have
\begin{equation}
\begin{aligned}
&\mathcal{E}^{ex} (\partial \partial^I\Gamma^J \mathfrak{m}, t, \gamma_{e_1})
\\
\lesssim
&\mathcal{E}^{ex} (\partial \partial^I\Gamma^J \mathfrak{m}, t_0, \gamma_{e_1})
\\
&+
\int_{t_0}^t \big\| \langle r-s\rangle^{\gamma_{e_1}} \partial \partial^I \Gamma^J |E|^2 \big\|_{L^2(\Sigma^{ex}_s)} 
\mathcal{E}^{ex} (\partial \partial^I\Gamma^J \mathfrak{m}, s, \gamma_{e_1})^{1\over 2} \, \d s.
\end{aligned}
\end{equation}

\textbf{Case 1: $|I|+|J|\leq N$.}
Applying the product rule, we find that
\begin{align*}
&\big\| \langle r-s\rangle^{\gamma_{e_1}} \partial \partial^I \Gamma^J |E|^2 \big\|_{L^2(\Sigma^{ex}_s)}
\\
\lesssim
&\sum_{\substack{I_1+I_2=I, J_1+J_2=J \\|I_1| + |J_1|\geq |I_2|+|J_2|}} \big\| \langle r-s\rangle^{\gamma_{e_1}} \partial \partial^{I_1} \Gamma^{J_1} E \big\|_{L^2(\Sigma^{ex}_s)}
\big\|  \partial^{I_2} \Gamma^{J_2} E \big\|_{L^\infty(\Sigma^{ex}_s )}
\\
&+\sum_{\substack{I_1+I_2=I, J_1+J_2=J \\|I_1| + |J_1|\leq |I_2|+|J_2|}}\big\| \partial \partial^{I_1} \Gamma^{J_1} E \big\|_{L^\infty(\Sigma^{ex}_s)}
\big\| \langle r-s\rangle^{\gamma_{e_1}}  \partial^{I_2} \Gamma^{J_2} E \big\|_{L^2(\Sigma^{ex}_s)}
\\
\lesssim
&(C_1 \epsilon)^2 s^{-{3\over 2}},
\end{align*}
in which we used the estimates from \eqref{eq:BA-ex01} and \eqref{eq:KG-refine01}.

\textbf{Case 2: $|I|+|J|\leq N+1$.}

Similarly, we apply the product rule to get
\begin{align*}
&\big\| \langle r-s\rangle^{\gamma_{e_1}} \partial \partial^I \Gamma^J |E|^2 \big\|_{L^2(\Sigma^{ex}_s)}
\\
\lesssim
&\sum_{\substack{I_1+I_2=I, J_1+J_2=J \\|I_1| + |J_1|\geq |I_2|+|J_2|}} \big\| \langle r-s\rangle^{\gamma_{e_2}} \partial \partial^{I_1} \Gamma^{J_1} E \big\|_{L^2(\Sigma^{ex}_s)}
\\
&\qquad\qquad\qquad
\times
\big\| \langle r-s\rangle^{\gamma_{e_1} -\gamma_{e_2}}  \partial^{I_2} \Gamma^{J_2} E \big\|_{L^\infty(\Sigma^{ex}_s)}
\\
&+\sum_{\substack{I_1+I_2=I, J_1+J_2=J \\|I_1| + |J_1|\leq |I_2|+|J_2|}} \big\| \langle r-s\rangle^{\gamma_{e_1} -\gamma_{e_2}} \partial \partial^{I_1} \Gamma^{J_1} E \big\|_{L^\infty(\Sigma^{ex}_s)}
\\
&\qquad\qquad\qquad
\times
\big\| \langle r-s\rangle^{\gamma_{e_2}}  \partial^{I_2} \Gamma^{J_2} E \big\|_{L^2(\Sigma^{ex}_s)}
\\
\lesssim
&(C_1 \epsilon)^2 s^{-1-\delta},
\end{align*}
in which we used the estimates from \eqref{eq:BA-ex02} and \eqref{eq:KG-refine01}.

Thus, in either case we arrive at
\begin{align*}
&\mathcal{E}^{ex} (\partial \partial^I\Gamma^J \mathfrak{m}, t, \gamma_{e_1})
\\
\lesssim
&\epsilon^2 + (C_1 \epsilon)^2 \int_{t_0}^t s^{-1-\delta}  \mathcal{E}^{ex} (\partial \partial^I\Gamma^J \mathfrak{m}, s, \gamma_{e_1})^{1\over 2} \, \d s
\\
\lesssim
&\epsilon^2 + (C_1 \epsilon)^4 +  \int_{t_0}^t s^{-1-\delta}  \mathcal{E}^{ex} (\partial \partial^I\Gamma^J \mathfrak{m}, s, \gamma_{e_1}) \, \d s, 
\end{align*}
which further yields, combining the Gronwall inequality, that
\begin{align}
\mathcal{E}^{ex} (\partial \partial^I\Gamma^J \mathfrak{m}, t, \gamma_{e_1})
\lesssim
\epsilon^2 + (C_1 \epsilon)^4,
\qquad
|I|+|J|\leq N+1.
\end{align}

The bounds for $\mathcal{E}^{ex} ( \partial^I\Gamma^J \mathfrak{m}, t, \gamma_{e_2})$ with $|I|+|J|\leq N+2, \, |J|\leq N+1$ can be derived in a similar way. The proof is completed.
\end{proof}

\begin{corollary}\label{cor:nDelta-point}
In the region $\{ r\geq t - \eta_0 \}$ with $t_0\leq t\leq T$, we have
\begin{align}
|\partial\partial \partial^{I} \Gamma^J \mathfrak{m}|
\lesssim
\big(\epsilon + (C_1 \epsilon)^2\big) \langle t+r\rangle^{-{1\over 2}} \langle t-r\rangle^{-\gamma_{e_1}},
\quad
|I| + |J| \leq N-1.
\end{align}
\end{corollary}
\begin{proof}
The proof follows from the estimates from Lemma \ref{lem:nDelta-ex} as well as the Sobolev inequality in Proposition \ref{prop:Sobolev-ex} and commutator estimates.
\end{proof}

\begin{proposition}\label{prop:wave-est}
We have the following refined bounds for $n$.

\begin{enumerate}
\item \emph{Refined $L^2$ bounds.}
\begin{eqnarray}
\| \partial^I \Gamma^J n\|_{L^2(\Sigma^{ex}_t)}
\lesssim
\left\{ \begin{array}{lll}
\epsilon + (C_1 \epsilon)^2,
\qquad
&|I| + |J| \leq N,
\\[1ex]
\big(\epsilon + (C_1 \epsilon)^2\big) \log \langle t\rangle,
\qquad
&|I| + |J| \leq N+1.
\end{array} \right. 
\end{eqnarray}

\item \emph{Refined decay bounds.}
\begin{align}
|\partial^I \Gamma^J n |
\lesssim
&(\epsilon + (C_1 \epsilon)^2) \langle t+r\rangle^{-{1\over 2}} \langle t-r\rangle^{-\gamma_{e_1}},
\quad
|I| + |J| \leq N-2.
\end{align}
\end{enumerate}
\end{proposition}

\begin{proof}
The proof follows from Lemmas \ref{lem:n0}, \ref{lem:nDelta-ex}, Corollary \ref{cor:nDelta-point}, and the relation $n = \ell + \Delta \mathfrak{m}$.
\end{proof}

Next, we improve the bounds on the Klein-Gordon component $E$.
\begin{proposition}[Refined lower-order energy bounds for $E$] \label{prop:KG-refine-ex}
We have
\begin{align}
\mathcal{E}^{ex}_1 (\partial^I\Gamma^J E, t, \gamma_{e_1})
\lesssim
&\epsilon^2 + (C_1 \epsilon)^3,
\qquad
|I|+|J|\leq N. \label{eq:KG-refine-ex}
\end{align}
\end{proposition}

\begin{proof}
We apply the energy estimate in Proposition \ref{prop:EE-ex2} on \eqref{eq:KG-high} to get
\begin{align*}
&\mathcal{E}^{ex}_1 (\partial^I\Gamma^J E, t, \gamma_{e_1})
\\
\lesssim
&\mathcal{E}^{ex}_1 (\partial^I\Gamma^J E, t_0, \gamma_{e_1})
+
\int_{\mathcal{D}^{ex}_t} \langle r-s\rangle^{2\gamma_{e_1}} \big| (-\partial^I \Gamma^J (nE) + n \partial^I \Gamma^J E) |\partial_t \partial^I\Gamma^J E| \big| \, \d x \d s
\\
&+
\int_{\mathcal{D}^{ex}_t} \langle r-s\rangle^{2\gamma_{e_1}} \big| \partial_t n |\partial^I \Gamma^J E|^2  \big| +  \langle r-s\rangle^{2\gamma_{e_1} -1} \big| n |\partial^I \Gamma^J E|^2 \big| \, \d x \d s.
\end{align*}
Recall the decomposition $n = \ell + \Delta \mathfrak{m}$, and we get
\begin{align}\label{eq:E-ex-001}
&\mathcal{E}^{ex}_1 (\partial^I\Gamma^J E, t, \gamma_{e_1})
\lesssim
\epsilon^2
+
\mathcal{A}_1 + \mathcal{A}_2 + \mathcal{B},
\end{align}
in which
\begin{align*}
\mathcal{A}_1
&=
\int_{\mathcal{D}^{ex}_t} \langle r-s\rangle^{2\gamma_{e_1}} \big| (-\partial^I \Gamma^J (\Delta \mathfrak{m} E) + \Delta \mathfrak{m} \partial^I \Gamma^J E) |\partial_t \partial^I\Gamma^J E| \big| \, \d x \d s,
\\
\mathcal{A}_2
&=
\int_{\mathcal{D}^{ex}_t} \langle r-s\rangle^{2\gamma_{e_1}} \big| (-\partial^I \Gamma^J (\ell E) + \ell \partial^I L^J E) |\partial_t \partial^I\Gamma^J E| \big| \, \d x \d s,
\end{align*}
and 
\begin{align*}
\mathcal{B}
=
\int_{\mathcal{D}^{ex}_t} \langle r-s\rangle^{2\gamma_{e_1}} \big| \partial_t n |\partial^I \Gamma^J E|^2  \big| +  \langle r-s\rangle^{2\gamma_{e_1}-1} \big| n |\partial^I \Gamma^J E|^2 \big| \, \d x \d s.
\end{align*}

\textbf{Estimate on $\mathcal{A}_1$.}
By the product rule, we find 
\begin{align*}
&\big|-\partial^I \Gamma^J (\Delta \mathfrak{m} E) + \Delta \mathfrak{m} \partial^I \Gamma^J E \big|
\\
=
&\big| -\sum_{\substack{I_1 + I_2 = I, J_1 + J_2 = J\\ |I_2|+|J_2|\leq |I| + |J|-1}} \partial^{I_1} \Gamma^{J_1} \Delta \mathfrak{m} \partial^{I_2} \Gamma^{J_2} E \big|
\\
\lesssim
& \sum_{\substack{I_1 + I_2 = I, J_1 + J_2 = J\\ |I_2|+|J_2|\leq |I| + |J|-1}} \big|\partial\partial \partial^{I_1} \Gamma^{J_1}  \mathfrak{m} \partial^{I_2} \Gamma^{J_2} E \big|
\\
\lesssim
& \sum_{\substack{I_1 + I_2 = I, J_1 + J_2 = J\\|I_1|+|J_1|\leq |I_2|+|J_2|\\ |I_2|+|J_2|\leq |I| + |J|-1}} \big|\partial\partial \partial^{I_1} \Gamma^{J_1}  \mathfrak{m} \partial^{I_2} \Gamma^{J_2} E \big|
\\
&+
 \sum_{\substack{I_1 + I_2 = I, J_1 + J_2 = J\\|I_1|+|J_1|\geq |I_2|+|J_2|}} \big|\partial\partial \partial^{I_1} \Gamma^{J_1}  \mathfrak{m} \partial^{I_2} \Gamma^{J_2} E \big|
=: \mathcal{T}_1 +\mathcal{T}_2 .
\end{align*}

Using Corollary \ref{cor:nDelta-point}, we take pointwise bounds for $\mathfrak{m}$ part to find
\begin{align*}
\mathcal{T}_1
\lesssim
C_1 \epsilon \langle s+r\rangle^{-{1\over 2}} \langle s-r\rangle^{-\gamma_{e_1}} \sum_{|I_2|+|J_2|\leq |I|+|J|-1} |\partial^{I_2} \Gamma^{J_2} E|.
\end{align*}
Recall Proposition \ref{prop:KG-extra} and we get that
\begin{equation}\label{eq:ex001}
\begin{aligned}
&\big| \partial^{I_2} \Gamma^{J_2} E \big|
\\
\lesssim
&{|\tau_-| \over \langle s\rangle} |\partial \partial \partial^{I_2} \Gamma^{J_2} E|
+
{1\over \langle s\rangle} |\partial \Gamma \partial^{I_2} \Gamma^{J_2} E|
+
{1\over \langle s\rangle} |\partial \partial^{I_2} \Gamma^{J_2} E|
+
|\partial^{I_2} \Gamma^{J_2} (n E)|,
\end{aligned}
\end{equation}
which, by smallness of the component $n$ in \eqref{eq:n-rough}, further leads to
\begin{equation}\label{eq:ex002}
\begin{aligned}
&\sum_{|I_2|+|J_2|\leq |I|+|J|-1 \leq N-1} |\partial^{I_2} \Gamma^{J_2}  E|
\\
\lesssim
&\sum_{|I_2|+|J_2|\leq N-1} {|\tau_-| \over \langle s\rangle} |\partial \partial \partial^{I_2} \Gamma^{J_2} E|
+
{1\over \langle s\rangle} |\partial \Gamma \partial^{I_2} \Gamma^{J_2} E|
+
{1\over \langle s\rangle} |\partial \partial^{I_2} \Gamma^{J_2} E|.
\end{aligned}
\end{equation}
Therefore, we have
\begin{align*}
\mathcal{T}_1
\lesssim
C_1 \epsilon\langle \tau_+\rangle^{-{3\over 2}} \langle \tau_-\rangle^{-\gamma_{e_1} +1} \sum_{|I_2|+|J_2|\leq |I|+|J|} |\partial\partial^{I_2} \Gamma^{J_2} E|.
\end{align*}

For the term $\mathcal{T}_2$, applying the pointwise estimates in \eqref{eq:KG-refine01}, we have
\begin{align*}
\mathcal{T}_2
\lesssim
C_1 \epsilon \langle s+r\rangle^{-1/2 - \gamma_{e_1}} \sum_{|I_1|+|J_1|\leq |I|+|J|} |\partial\partial \partial^{I_1} \Gamma^{J_1} \mathfrak{m}|.
\end{align*}

Thus, we bound
\begin{align*}
&\big|-\partial^I \Gamma^J (\Delta \mathfrak{m} E) + \Delta \mathfrak{m} \partial^I \Gamma^J E \big|
\\
\lesssim
& C_1 \epsilon \langle s+r\rangle^{-{1\over 2} - \gamma_{e_1}} \sum_{|I_1|+|I_2|\leq |I|+|J|} |\partial\partial \partial^{I_1} \Gamma^{J_1} \mathfrak{m}|
\\
&+
C_1 \epsilon \langle \tau_+\rangle^{-3/2} \langle \tau_-\rangle^{-\gamma_{e_1}+1} \sum_{|I_2|+|J_2|\leq |I|+|J|} |\partial\partial^{I_2} \Gamma^{J_2} E|.
\end{align*}

We insert these bounds to derive
\begin{equation}\label{eq:E-ex-02}
\begin{aligned}
\mathcal{A}_1
\lesssim
&\int_{t_0}^t C_1 \epsilon \langle s\rangle^{-{3\over 2} } \Big(\sum_{|I_1|+|J_1|\leq |I|+|J|}  \big\| \langle r-s \rangle^{\gamma_{e_1}} \partial\partial \partial^{I_1} \Gamma^{J_1} \mathfrak{m} \big\|_{L^2(\Sigma^{ex}_s)}   
\\
&\qquad\qquad\qquad
+
 \big\|\langle r-s\rangle^{\gamma_{e_1}} \partial\partial^{I_2} \Gamma^{J_2} E \big\|_{L^2(\Sigma^{ex}_s)} \Big)
 \\
&\qquad\qquad\qquad
\times
\big\| \langle r-s\rangle^{\gamma_{e_1}} \partial_t \partial^I L^J E \big\|_{L^2(\Sigma^{ex}_s)} 
\, \d s
\\
\lesssim
&(C_1 \epsilon)^3 \int_{t_0}^t  \langle s\rangle^{-{3\over 2} } \, \d s
\lesssim
(C_1 \epsilon)^3.
\end{aligned}
\end{equation}

\textbf{Estimate on $\mathcal{A}_2$.}
We first apply the product rule to get 
\begin{align*}
&\big|-\partial^I \Gamma^J ( \ell E) + \ell \partial^I \Gamma^J E \big|
\\
=
&\Big| \sum_{\substack{I_1 + I_2 = I, J_1 + J_2 = J\\ |I_2|+|J_2|\leq |I| + |J|-1}} \partial^{I_1} \Gamma^{J_1}  \ell \partial^{I_2} \Gamma^{J_2} E \Big|
\\
\lesssim
& \sum_{\substack{I_1 + I_2 = I, J_1 + J_2 = J\\|I_1|+|J_1|\leq |I_2|+|J_2|\\ |I_2|+|J_2|\leq |I| + |J|-1}} \big| \partial^{I_1} \Gamma^{J_1}  \ell \partial^{I_2} \Gamma^{J_2} E \big|
\\
&+
 \sum_{\substack{I_1 + I_2 = I, J_1 + J_2 = J\\|I_1|+|J_1|\geq |I_2|+|J_2|}} \big| \partial^{I_1} \Gamma^{J_1}  \ell \partial^{I_2} \Gamma^{J_2} E \big|
=: \mathcal{T}_3 +\mathcal{T}_4 .
\end{align*}

By taking pointwise bounds of $\ell$ component in \eqref{eq:n0-point}, we have
\begin{align*}
\mathcal{T}_3
\lesssim
C_1 \epsilon \langle s+r\rangle^{-{1\over 2}} \langle s-r\rangle^{-\gamma_{e_1}} \sum_{|I_2|+|J_2|\leq |I|+|J|-1} |\partial^{I_2} \Gamma^{J_2} E|.
\end{align*}
We then apply the refined decay estimate for $E$ component in Proposition \ref{prop:KG-extra}, as similarly done in \eqref{eq:ex001} and \eqref{eq:ex002}, to get
\begin{align*}
\mathcal{T}_3
\lesssim
C_1 \epsilon \langle s+r\rangle^{-{3\over 2}} \langle s-r\rangle^{-\gamma_{e_1} +1} \sum_{|I_2|+|J_2|\leq |I|+|J|} |\partial\partial^{I_2} \Gamma^{J_2} E|.
\end{align*}

By taking pointwise bounds of the component $E$ in \eqref{eq:KG-refine01}, we have
\begin{align*}
\mathcal{T}_4
\lesssim
C_1 \epsilon \langle s+r\rangle^{-{1\over 2} - \gamma_{e_1}} \sum_{|I_1|+|J_1|\leq |I|+|J|} | \partial^{I_1} \Gamma^{J_1} \ell|.
\end{align*}

The estimates for $\mathcal{T}_3$ and $\mathcal{T}_4$ yield that
\begin{align*}
&\big|-\partial^I \Gamma^J ( \ell E) + \ell \partial^I \Gamma^J E \big|
\\
\lesssim
& C_1 \epsilon \langle s+r\rangle^{-{1\over 2} - \gamma_{e_1}} \sum_{|I_1|+|J_1|\leq |I|+|J|} | \partial^{I_1} \Gamma^{J_1} \ell|
\\
&+
C_1 \epsilon \langle s+r\rangle^{-{3\over 2}} \langle s-r\rangle^{-\gamma_{e_1}+1} \sum_{|I_2|+|J_2|\leq |I|+|J|} |\partial\partial^{I_2} \Gamma^{J_2} E|.
\end{align*}

Therefore, similarly to what we did for $\mathcal{A}_1$ in \eqref{eq:E-ex-02}, we derive
\begin{align}
\mathcal{A}_2 \lesssim (C_1 \epsilon)^3.
\end{align}

\textbf{Estimate on $\mathcal{B}$.}
We then estimate $\mathcal{B}$, and we take $L^\infty$-norm on the wave component $n$ to get
\begin{align*}
\mathcal{B}
\lesssim
\epsilon \int_{\mathcal{D}^{ex}_t} \langle r-s\rangle^{2\gamma_{e_1} -1}  |\partial^I \Gamma^J E|^2  \, \d x \d s.
\end{align*}
By smallness of $\epsilon$, this term can be absorbed by the energy in the left hand side of \eqref{eq:E-ex-001}.

Finally, we conclude \eqref{eq:KG-refine-ex}.
\end{proof}

\begin{proposition}[Refined top-order energy bounds for $E$] \label{prop:KG-refine-ex-a}
We have
\begin{align}
\mathcal{E}^{ex}_1 (\partial^I\Gamma^J E, t, \gamma_{e_2})
\lesssim
&\epsilon^2 + (C_1 \epsilon)^3,
\qquad
|I|+|J|\leq N+1. \label{eq:KG-refine-ex-a}
\end{align}
\end{proposition}

\begin{proof}
Let $|I|+|J|\leq N+1$. We apply the energy estimate in Proposition \ref{prop:EE-ex2} on \eqref{eq:KG-high} to get
\begin{align*}
&\mathcal{E}^{ex}_1 (\partial^I\Gamma^J E, t, \gamma_{e_2})
\\
\lesssim
&\mathcal{E}^{ex}_1 (\partial^I\Gamma^J E, t_0, \gamma_{e_2})
+
\int_{\mathcal{D}^{ex}_t} \langle r-s\rangle^{2\gamma_{e_2}} \big| (-\partial^I \Gamma^J (nE) + n \partial^I \Gamma^J E) |\partial_t \partial^I\Gamma^J E| \big| \, \d x \d s
\\
&+
\int_{\mathcal{D}^{ex}_t} \langle r-s\rangle^{2\gamma_{e_2}} \big| \partial_t n |\partial^I \Gamma^J E|^2  \big| +  \langle r-s\rangle^{2\gamma_{e_2}-1} \big| n |\partial^I \Gamma^J E|^2 \big| \, \d x \d s
\\
\lesssim
& \epsilon^2 + \mathcal{C} + \mathcal{D},
\end{align*}
in which 
\begin{align*}
\mathcal{C}
=
&\int_{\mathcal{D}^{ex}_t} \langle r-s\rangle^{2\gamma_{e_2}} \big| (-\partial^I \Gamma^J (nE) + n \partial^I \Gamma^J E) |\partial_t \partial^I\Gamma^J E| \big| \, \d x \d s,
\\
\mathcal{D}
=
&\int_{\mathcal{D}^{ex}_t} \langle r-s\rangle^{2\gamma_{e_2}} \big| \partial_t n |\partial^I \Gamma^J E|^2  \big| +  \langle r-s\rangle^{2\gamma_{e_2}-1} \big| n |\partial^I \Gamma^J E|^2 \big| \, \d x \d s.
\end{align*}

\textbf{Estimate on $\mathcal{C}$.}

We apply the product rule to get 
\begin{align*}
&\big|-\partial^I \Gamma^J ( n E) + n \partial^I \Gamma^J E \big|
\\
=
&\Big| \sum_{\substack{I_1 + I_2 = I, J_1 + J_2 = J\\ |I_2|+|J_2|\leq |I| + |J|-1}} \partial^{I_1} \Gamma^{J_1}  n \partial^{I_2} \Gamma^{J_2} E \Big|
\\
\lesssim
& \sum_{\substack{I_1 + I_2 = I, J_1 + J_2 = J\\|I_1|+|J_1|\leq |I_2|+|J_2|\\ |I_2|+|J_2|\leq |I| + |J|-1}} \big| \partial^{I_1} \Gamma^{J_1}  n \partial^{I_2} \Gamma^{J_2} E \big|
\\
&+
 \sum_{\substack{I_1 + I_2 = I, J_1 + J_2 = J\\|I_1|+|J_1|\geq |I_2|+|J_2|}} \big| \partial^{I_1} \Gamma^{J_1}  n \partial^{I_2} \Gamma^{J_2} E \big|
=: \mathcal{T}_5 +\mathcal{T}_6 .
\end{align*}

We first bound $\mathcal{T}_5$. 
Employing the pointwise bounds of the $n$ component in Proposition \ref{prop:wave-est}, we find
\begin{align*}
\mathcal{T}_5
\lesssim
(\epsilon + (C_1\epsilon)^2) \langle s+r\rangle^{-{1\over 2}} \langle s-r\rangle^{-\gamma_{e_1}}  \sum_{|I_2|+|J_2|\leq |I|+|J|-1} \big| \partial^{I_2} \Gamma^{J_2} E \big|.
\end{align*}
Based on Proposition \ref{prop:KG-extra}, as similarly done in \eqref{eq:ex001} and \eqref{eq:ex002}, we get
\begin{align*}
\mathcal{T}_5
\lesssim
(\epsilon + (C_1\epsilon)^2) \langle s\rangle^{-{3\over 2}} \langle s-r\rangle^{-\gamma_{e_1}+1} \sum_{|I_2|+|J_3|\leq |I|+|J|} \big| \partial \partial^{I_2} \Gamma^{J_3} E \big|.
\end{align*}
By interpolating the above two estimates, we have
\begin{align*}
\mathcal{T}_5
\lesssim
(\epsilon + (C_1\epsilon)^2) \langle s\rangle^{-{7\over 6}} \langle s-r\rangle^{-\gamma_{e_1}+{2\over 3}} 
&\Big(\sum_{|I_2|+|J_2|\leq |I|+|J|-1} \big| \partial^{I_2} L^{J_2} E \big| \Big)^{1\over 3}
\\
\times
&\Big( \sum_{|I_2|+|J_3|\leq |I|+|J|} \big| \partial \partial^{I_2} \Gamma^{J_3} E \big| \Big)^{2\over 3}.
\end{align*}

For the term $\mathcal{T}_6$, we insert the pointwise bounds of the  component $E$ in \eqref{eq:KG-refine01} to derive
\begin{align*}
\mathcal{T}_6
\lesssim
C_1 \epsilon \sum_{|I_1|+|J_1|\leq |I|+|J|}  \langle s+r\rangle^{-{1\over 2} - \gamma_{e_1}}  \big| \partial^{I_1} \Gamma^{J_1}  n  \big|.
\end{align*}

Thus, we have
\begin{align*}
\mathcal{C}
\lesssim
\int_{t_0}^t 
&(\epsilon + (C_1\epsilon)^2) \langle s\rangle^{-{7\over 6}} \Big( \sum_{|I_2|+|J_2|\leq |I|+|J|-1} \big\| \langle s-r\rangle^{\gamma_{e_2} }  \partial^{I_2} \Gamma^{J_2} E \big\|^{1\over 3}_{L^2(\Sigma^{ex}_s)} \Big) 
\\
\times
&
\Big( \sum_{|I_2|+|J_3|\leq |I|+|J|} \big\| \langle s-r\rangle^{\gamma_{e_2} } \partial \partial^{I_2} \Gamma^{J_3} E \big\|^{2\over 3}_{L^2(\Sigma^{ex}_s)} \Big) \big\|\langle r-s \rangle^{\gamma_{e_2}} \partial_t \partial^{I} \Gamma^{J} E\big\|_{L^2(\Sigma^{ex}_s)}
\\
+
&\langle s\rangle^{-{1\over 2} - \gamma_{e_1} + \gamma_{e_2}} C_1 \epsilon \sum_{|I_1|+|J_1|\leq |I|+|J|} \| \partial^{I_1} \Gamma^{J_1} n\|_{L^2(\Sigma^{ex}_s)} \big\|\langle r-s \rangle^{\gamma_{e_2}} \partial_t \partial^{I} \Gamma^{J} E\big\|_{L^2(\Sigma^{ex}_s)} \, \d s
\\
\lesssim
& (\epsilon + (C_1\epsilon)^2) (C_1 \epsilon)^2.
\end{align*}

\textbf{Estimate on $\mathcal{D}$.}
We take $L^\infty$ norm on $\partial_t n$ and $n$ to get
\begin{align*}
\mathcal{D}
\lesssim
C_1 \epsilon \int_{\mathcal{D}^{ex}_t} \langle r-s \rangle^{2\gamma_{e_2}-\gamma_{e_1} -{1\over 2}}  |\partial^I \Gamma^J E|^2  \, \d x \d s.
\end{align*}
By smallness of $\epsilon$, this term can be absorbed by the left hand side.

Therefore, we get \eqref{eq:KG-refine-ex-a}, and the proof is completed.
\end{proof}

\begin{proposition}\label{prop:ex-improved}
The following improved bounds on $E, n$ hold for all $t\in [t_0, T)$:
\begin{equation}\label{eq:ex-improved}
\begin{aligned}
\|\partial^I\Gamma^J n \|_{L^2(\Sigma^{ex}_t)}
\leq 
&{1\over 2} C_1 \epsilon,
\qquad
&|I|+|J| \leq N,   
\\
\|\partial^I\Gamma^J n \|_{L^2(\Sigma^{ex}_t)}
\leq 
&{1\over 2} C_1 \epsilon \log \langle t\rangle,
\qquad
&|I|+|J| \leq N+1, 
\\
\mathcal{E}^{ex}_1 (\partial^I\Gamma^J E, t, \gamma_{e_1})
\leq 
&{1\over 2}(C_1 \epsilon)^2,
\qquad
&|I|+|J|\leq N, 
\\
\mathcal{E}^{ex}_1 (\partial^I\Gamma^J E, t, \gamma_{e_2} )
\leq 
&{1\over 2} (C_1 \epsilon)^2,
\qquad
&|I|+|J|\leq N+1.  
\end{aligned} 
\end{equation}
Consequently, the solution pair $(E, n)$ exists globally in the exterior region $\{ t\geq t_0, \, r\geq t - \eta_0 \}$.
\end{proposition}
\begin{proof}
The first two bounds for $n$ are from Proposition \ref{prop:wave-est}, by taking $C_1$ sufficiently large and $\epsilon$ small enough such that $\lesssim \epsilon + (C_1 \epsilon)^2$ and $\lesssim (\epsilon + (C_1 \epsilon)^2) \log \langle t\rangle$ give $\leq {1\over 2} C_1 \epsilon$ and $\leq {1 \over 2} C_1 \epsilon \log \langle t\rangle$, respectively.

Similarly, the last two bounds for $E$ are from Propositions \ref{prop:KG-refine-ex} and \ref{prop:KG-refine-ex-a}.
\end{proof}

%
%
%

\section{Global existence: the interior region}\label{sec:interior}

With the global existence and various bounds on the solution $(E, n)$ in the exterior region established in Section \ref{sec:exterior}, we now prove global existence in the interior region $\{ r\leq t-1 \}$.

We recall the following  estimates from the analysis in the exterior region in  Section \ref{sec:exterior} with $\eta_0 = 1$.
\begin{proposition}[Boundary estimates]\label{prop:BDRY-4}
The following estimates are valid for all $t \geq t_0$.

\

\begin{enumerate}
\item \emph{Pointwise estimates.}
\begin{align}
\sup_{|x| = t-1} \big|\langle\tau_+\rangle^{1\over 2} \langle\tau_-\rangle^{3}  \partial\partial \partial^I L^J \mathfrak{m}(t, x)\big|
\lesssim
C_1 \epsilon,
\qquad
|I|+|J|\leq N-1,
\\
\sup_{|x| = t-1} \big|\langle\tau_+\rangle^2 \langle\tau_-\rangle^{3} \partial^I L^J E(t, x)\big|
\lesssim
C_1 \epsilon,
\qquad
|I| + |J|\leq N-5.
\end{align}

\item \emph{$L^2$ estimates (below with $\eta_0=1$).}
\begin{equation}
\begin{aligned}
&\int_{\Sigma^{bd}_t} \big( \sum_{a=1}^2(\partial_a \partial^I L^J \mathfrak{m} + \frac{x_a}{r}\partial_t \partial^I L^J \mathfrak{m})^2   \big) \d S
\\
\lesssim
& \epsilon^2+ (C_1 \epsilon)^4,
\qquad\qquad
|I|+|J|\leq N+2, \, |J|\leq N+1,
\end{aligned}
\end{equation}
\begin{equation}
\begin{aligned}
&\int_{\Sigma^{bd}_t} \big( \sum_{a=1}^2(\partial_a \partial\partial^I L^J \mathfrak{m} + \frac{x_a}{r}\partial_t \partial\partial^I L^J \mathfrak{m})^2   \big) \d S
\\
\lesssim
&\epsilon^2 + (C_1 \epsilon)^4,
\qquad\qquad
|I|+|J|\leq N+1,
\end{aligned}
\end{equation}
\begin{equation}
\begin{aligned}
&\int_{\Sigma^{bd}_t} \big( \sum_{a=1}^2(\partial_a \partial^I L^J \ell + \frac{x_a}{r} \partial_t \partial^I L^J \ell)^2   \big) \d S
\\
\lesssim
&\epsilon^2 + (C_1 \epsilon)^4,
\qquad\qquad
|I|+|J|\leq N,
\end{aligned}
\end{equation}
\begin{equation}
\begin{aligned}
&\int_{\Sigma^{bd}_t} \big( \sum_{a=1}^2|\partial_a \partial^I L^J E + \frac{x_a}{r}\partial_t \partial^I L^J E|^2 + |\partial^I L^J E|^2  \big) \d S
\\
\lesssim
&\epsilon^2+ (C_1 \epsilon)^3,
\qquad\qquad
|I|+|J|\leq N+1.
\end{aligned}
\end{equation}
\end{enumerate}
\end{proposition}
\begin{proof}
One finds the pointwise estimates from Proposition \ref{prop:KG-refine01} and Corollary \ref{cor:nDelta-point}, together with the fact that $1\lesssim \langle t-r\rangle \lesssim 1$ on the boundary. For the $L^2$ type estimates, one refers to Lemma \ref{lem:nDelta-ex} and Proposition \ref{prop:KG-refine-ex-a}.
\end{proof}

We make the following bootstrap assumptions for $\tau \in [\tau_0, \tau_1 )$
\begin{align}
|\partial^I L^J n(t, x)| 
\leq
&C_1 \epsilon t^{-{1\over 2}},
\qquad
&|I| + |J| \leq N-2, \label{eq:BA-in}
\\
\mathcal{E}^{in}(\partial^I L^J n, \tau)
\leq
&(C_1 \epsilon)^{2},
\qquad
&|I| + |J| \leq N, \label{eq:BA-in2}
\\
\mathcal{E}^{in}_1(\partial^I L^J E, \tau)
\leq
&(C_1 \epsilon)^2 \tau^{2\delta},
\qquad
&|I| + |J| \leq N+1. \label{eq:BA-in3}
\end{align}

The constant $C_1$ here (such that $C_1^2 \epsilon \ll 1$) might not be identical to that in Section \ref{sec:exterior}, but in the end we can take the maximum between these two numbers. We treat similarly for $\delta, \epsilon$. In addition, we set $\tau_1 = \sup \{ \tau: \eqref{eq:BA-in}-\eqref{eq:BA-in3} \text{ all hold} \}$.

To improve the bounds in bootstrap assumptions \eqref{eq:BA-in}--\eqref{eq:BA-in3}, we consider the equations of $\mathfrak{m}, E$ with higher order derivatives:
\begin{align}
-\Box \partial^I L^J \mathfrak{m} = &\partial^I L^J |E|^2, \label{eq:n-Delta-h}
\\
-\Box \partial^I L^J E + \partial^I L^J E = &- \partial^I L^J (n E). \label{eq:E-h}
\end{align}

Based on the estimates in \eqref{eq:BA-in3}, we have some pointwise decay estimates for the field $E$.
\begin{proposition}
Under the bootstrap assumptions in \eqref{eq:BA-in3}, the following pointwise estimates hold.
\begin{enumerate}
		\item \emph{Rough decay for $E$.}
\begin{align}
(\tau/t)|\partial \partial^I\Gamma^J E|  + |\partial^I\Gamma^J E | 
\lesssim
C_1 \epsilon  \tau_+^{-1 + \delta},
\qquad
|I|+|J| \leq N-1.  \label{eq:E-rough-in}
\end{align}

\item \emph{Improved decay for $E$.}
\begin{align}
|\partial^I\Gamma^J E|
\lesssim
C_1 \epsilon \tau_-  \tau_+^{-2 + \delta},
\qquad
|I|+|J| \leq N-3.   \label{eq:E-improve-in}
\end{align}
\end{enumerate}
\end{proposition}
\begin{proof}
    The proof of \eqref{eq:E-rough-in} follows from $L^2$-type bounds in \eqref{eq:BA-in3}, Klainerman-Sobolev inequality in Proposition \ref{prop:Sobolev-in}, and commutator estimates.

    Relying further on \eqref{eq:BA-in} and Proposition \ref{prop:KG-extra}, one obtains \eqref{eq:E-improve-in}. 
\end{proof}

\subsection{Improved bounds}

First, recall that $\ell$ satisfies a linear homogeneous wave equation, so its $L^2$ bounds and pointwise decay can be  derived in a relatively easier manner.
\begin{lemma}\label{lem:n0-in}
\begin{enumerate}
\item \emph{$L^2$ bounds.}
\begin{align*}
\| (\tau/t)\partial^I L^J \ell \|_{L^2(\mathcal{H}_\tau)}
\lesssim
&\epsilon,
\qquad
&1\leq |I|+|J| \leq N+1,
\\
\big\| (\tau/t)\tau_-^{1\over 2} \partial\partial^I L^J \ell \big\|_{L^2(\mathcal{H}_\tau)}
\lesssim
&\epsilon,
\qquad
&1\leq |I|+|J| \leq N.
\end{align*}

\item \emph{Pointwise bounds.}
\begin{align*}
|\partial^I L^J \ell|
\lesssim
&\epsilon \tau^{-1},
\qquad
&1\leq |I| + |J|\leq N-1,
\\
\tau_-^{-1}| \ell| + |\partial \partial^I L^J \ell|
\lesssim
&\epsilon \tau^{-1} \tau_-^{-{1\over 2}},
\qquad
&|I| + |J|\leq N-3.
\end{align*}

\end{enumerate}

\end{lemma}
\begin{proof}
The $L^2$-type estimates follow from Propositions \ref{prop:EE-in} and \ref{prop:con-EE} as well as the relation \eqref{eq:wave-partial}.

The $L^2$-type estimates together with Sobolev inequality \ref{prop:Sobolev-in} give the first pointwise decay.
For the second pointwise decay, one can refer to \cite{Dong2005} for the proof.
\end{proof}

\begin{lemma}[Estimates for $\mathfrak{m}$: I]\label{lem:nDelta-in}

\begin{enumerate}
\item \emph{$L^2$ bounds.}
\begin{align}
\big\| \tau_-^{-{1\over 2}} (\tau/t)\partial\partial^I L^J \mathfrak{m} \big\|_{L^2(\mathcal{H}_\tau)}
\lesssim
&\epsilon + (C_1 \epsilon)^2,
\qquad
&|I|+|J| \leq N+2, \, |J|\leq N+1, \label{eq:nD-L2-in01}
\\
\big\| \tau_-^{1\over 2} (\tau/t)\partial\partial\partial^I L^J \mathfrak{m} \big\|_{L^2(\mathcal{H}_\tau)}
\lesssim
&\epsilon + (C_1 \epsilon)^2,
\qquad
&|I|+|J| \leq N+1, \, |J|\leq N.  \label{eq:nD-L2-in02}
\end{align}

\item \emph{Pointwise bounds.}
\begin{align}
|\partial \partial^I L^J \mathfrak{m}|
\lesssim
&(\epsilon + (C_1 \epsilon)^2) \tau^{-1} \tau_-^{1\over 2},
\qquad
&|I| + |J|\leq N, \, |J|\leq N-1 \label{eq:nD-decay-in01}
\\
|\partial \partial \partial^I L^J \mathfrak{m}|
\lesssim
&(\epsilon + (C_1 \epsilon)^2) \tau^{-1} \tau_-^{-{1\over 2}},
\qquad
&|I| + |J|\leq N-2. \label{eq:nD-decay-in02}
\end{align}
\end{enumerate}
\end{lemma}

\begin{proof}

Let $|I| + |J| \leq N+2$ with $|J|\leq N+1$. We perform the weighted energy estimates in Proposition \ref{prop:EE-in} with $\gamma={1\over 2}$ to the equation of $\mathfrak{m}$ \eqref{eq:n-Delta-h}, and we find
\begin{equation}
\begin{aligned}
&\mathcal{E}^{in} (\partial^I L^J \mathfrak{m}, \tau, {1\over 2})
\\
\lesssim
&\mathcal{E}^{in} (\partial^I L^J \mathfrak{m}, \tau_0, {1\over 2})
+
\int_{\tau_0}^\tau \big\| \tau_-^{-{1\over 2}} \partial^I L^J|E|^2 \big\|_{L^2(\mathcal{H}_{\widetilde{\tau}})} \big\| \tau_-^{-{1\over 2}} (\widetilde{\tau} /t) \partial_t \partial^I L^J \mathfrak{m} \big\|_{L^2(\mathcal{H}_{\widetilde{\tau}})} \, \d\widetilde{\tau}
\\
&+
\int_{\Sigma^{bd}_{t^{bd}}}  \sum_a ({x_a \over r} \partial_t  \partial^I L^J \mathfrak{m} + \partial_a  \partial^I L^J \mathfrak{m})^2  \, \d S
\\
\lesssim
&\epsilon^2 + (C_1 \epsilon)^4
+
 \int_{\tau_0}^\tau \big\| \tau_-^{-{1\over 2}} \partial^I L^J|E|^2 \big\|_{L^2(\mathcal{H}_{\widetilde{\tau}})} \mathcal{E}^{in} (\partial^I L^J \mathfrak{m}, \widetilde{\tau}, {1\over 2})^{1\over 2} \, \d\widetilde{\tau},
\end{aligned}
\end{equation}
in which $t^{bd} = {\tau^2+1\over 2}$ and we used the bounds for the boundary term in Proposition \ref{prop:BDRY-4}.

Next, we bound $\| \tau_-^{-{1\over 2}} \partial^I L^J|E|^2 \|_{L^2(\mathcal{H}_{\widetilde{\tau}})}$. 
By the Leibniz rule, we have
\begin{align*}
&\big\| \tau_-^{-{1\over 2}} \partial^I L^J|E|^2 \big\|_{L^2(\mathcal{H}_{\widetilde{\tau}})}
\\
\lesssim
&\sum_{\substack{|I_1|+|J_1|\leq |I_2|+|J_2| \\|I_1|+|I_2|\leq |I|, |J_1|+|J_2|\leq |J|}} \big\|\tau_-^{-{1\over 2}} (\tau_+/\widetilde{\tau}) \partial^{I_1} L^{J_1} E \big\|_{L^\infty(\mathcal{H}_{\widetilde{\tau}})}   \big\| (\widetilde{\tau}/\tau_+) \partial^{I_2} L^{J_2} E \big\|_{L^2(\mathcal{H}_{\widetilde{\tau}})} 
\\
\lesssim
&(C_1 \epsilon)^2 \widetilde{\tau}^{-{5\over 4}},
\end{align*}
in which we used the pointwise bounds in \eqref{eq:E-improve-in} and energy bounds in \eqref{eq:BA-in3}.

Therefore, we get (below $C$ is a generic constant independent of $C_1, \epsilon$)
\begin{align*}
&\mathcal{E}^{in} (\partial^I L^J \mathfrak{m}, \tau, {1\over 2})
\\
\leq
&C\epsilon^2
+
C(C_1 \epsilon)^4
+
C \int_{\tau_0}^\tau \big\| \tau_-^{-{1\over 2}} \partial^I L^J|E|^2 \big\|_{L^2(\mathcal{H}_{\widetilde{\tau}})} \mathcal{E}^{in} (\partial^I L^J \mathfrak{m}, \widetilde{\tau}, {1\over 2})^{1\over 2} \, \d\widetilde{\tau}
\\
\leq
&C \epsilon^2 + C (C_1 \epsilon)^4
+
C \int_{\tau_0}^\tau \widetilde{\tau}^{-{5\over 4}} \mathcal{E}^{in} (\partial^I L^J \mathfrak{m}, \widetilde{\tau}, {1\over 2}) \, \d\widetilde{\tau},
\end{align*}
which further proves \eqref{eq:nD-L2-in01} by using the Gronwall inequality.

The proof of \eqref{eq:nD-L2-in02} relies on \eqref{eq:nD-L2-in01} and Proposition \ref{prop:wave-extra}. By Proposition \ref{prop:wave-extra}, for $|I|+|K|\leq N+1, \, |K|\leq N$ we have
\begin{align*}
&|\partial\partial \partial^I L^K \mathfrak{m}|
\\
\lesssim
&{1\over \langle  \tau_-\rangle } \Big(\sum_{|J_1|\leq |K|+1} |\partial \partial^I L^{J_1} \mathfrak{m}| 
+ \sum_{|J_2|\leq |K|}|\partial\partial \partial^I L^{J_2} \mathfrak{m}| \Big) 
+  {t\over \langle \tau_-\rangle}\big|\partial^I L^K |E|^2\big|.
\end{align*}
To proceed, one obtains
\begin{align*}
&\big\|\tau_-^{1\over 2} (\tau/t)\partial\partial \partial^I L^K \mathfrak{m}\big\|_{L^2(\mathcal{H}_\tau)}
\\
\lesssim
& \sum_{|J_1|\leq |K|+1} \big\|\tau_-^{-{1\over 2}} (\tau/t)\partial \partial^I L^{J_1} \mathfrak{m}\big\|_{L^2(\mathcal{H}_\tau)} 
\\
&+
\sum_{|J_2|\leq |K|} \big\|\tau_-^{-{1\over 2}} (\tau/t)\partial \partial \partial^I L^{J_2} \mathfrak{m}\big\|_{L^2(\mathcal{H}_\tau)} 
+ \big\|\tau_+ \tau_-^{-{1\over 2}}\partial^I L^K |E|^2\big\|_{L^2(\mathcal{H}_\tau)},
\end{align*}
and one derives \eqref{eq:nD-L2-in02} with the estimates for $\mathfrak{m}$ in \eqref{eq:nD-L2-in01} and the estimates for $E$ in \eqref{eq:BA-in3} and \eqref{eq:E-improve-in}.

Finally, by the Sobolev inequality in Proposition \ref{prop:Sobolev-in} and the $L^2$-type bounds on $\mathfrak{m}$ \eqref{eq:nD-L2-in01} and \eqref{eq:nD-L2-in02}, we have the decay estimates \eqref{eq:nD-decay-in01} and \eqref{eq:nD-decay-in02}.
\end{proof}

\begin{proposition}[Refined estimates for $n$: I] \label{prop:refine-n-in001}
We have following refined bounds for the component $n$.

\begin{enumerate}
\item \emph{$L^2$ bounds.}
\begin{align*}
\| (\tau/t)\partial^I L^J n \|_{L^2(\mathcal{H}_\tau)}
\lesssim
&\epsilon,
\qquad
&1\leq |I|+|J| \leq N+1, \quad |J|\leq N,
\\
\big\| (\tau/t)\tau_-^{1\over 2} \partial\partial^I L^J n \big\|_{L^2(\mathcal{H}_\tau)}
\lesssim
&\epsilon,
\qquad
&1\leq |I|+|J| \leq N.
\end{align*}

\item \emph{Pointwise bounds.}
\begin{align*}
|\partial^I L^J n|
\lesssim
&\epsilon \tau^{-1},
\qquad
&1\leq |I| + |J|\leq N-1, \, |J|\leq N-2, 
\\
\tau_-^{-1} |n| + |\partial \partial^I L^J n|
\lesssim
&\epsilon \tau^{-1} \tau_-^{-{1\over 2}},
\qquad
&|I| + |J|\leq N-3.
\end{align*}
\end{enumerate}
\end{proposition}

\begin{proof}
The proof follows from the relation $n = \ell + \Delta \mathfrak{m}$ and the estimates established in Lemmas \ref{lem:n0-in}--\ref{lem:nDelta-in}.
\end{proof}

\begin{proposition}[Refined pointwise decay for $E$] \label{prop:refine-E-decay}
It holds that
\begin{align}
|\partial^I E |
\lesssim
C_1 \epsilon t^{-{3\over 2}}\tau_-^{1\over 2},
\qquad
|I|\leq N-6.
\end{align}
\end{proposition}

\begin{proof}
We rely on Proposition \ref{prop:E-sharp-decay} to derive the sharp time decay for the Klein-Gordon component $E$. Therefore, we need to bound the source terms in the right hand side of \eqref{eq:E-sharp-decay}. Since the boundary terms can be easily bounded by the estimates from Section \ref{sec:exterior}, we only  estimate the integral terms in the right hand side of \eqref{eq:E-sharp-decay}.

For $|I|\leq N-4$, recall that the equation of $\partial^I E$ is
\begin{align*}
-\Box \partial^I E + \partial^I E = - \partial^I (n E)
= -\sum_{I_1+I_2=I, \, |I_1|\geq 1} \partial^{I_1} n \partial^{I_2} E - n \partial^I E = F.
\end{align*}

We now estimate 
\begin{align*}
&\big|\tau \cancel{\partial}_a \cancel{\partial}^a \partial^I E + {x^a x^b \over \tau} \cancel{\partial}_a \cancel{\partial}_b \partial^I E
+
2{x^a \over \tau} \cancel{\partial}_a \partial^I E\big| 
+ \big|\tau F + n \tau \partial^I E\big| 
+ \big| {d\over d\lambda} n \cdot \tau \partial^I E\big|
\\
=: & \mathcal{A}_1 + \mathcal{A}_2 + \mathcal{A}_3.
\end{align*}

For the term $\mathcal{A}_1$, we note
\begin{align*}
&\mathcal{A}_1
\\
\leq
&\big| {\tau \over t^2} L_a L^a \partial^I E - {\tau x_a \over t^3} L^a \partial^I E\big|
+
\big| {x^a x^b \over \tau t^2} L_a L_b \partial^I E - {x^a x^b x_a \over \tau t^3} L_b \partial^I E\big|
+
2 \big| {x^a \over \tau t} L_a \partial^I E\big|
\\
\lesssim
& C_1 \epsilon \tau^{-2+\delta}.
\end{align*}

Next, we estimate $\mathcal{A}_2$, and we find
\begin{align*}
\mathcal{A}_2
\lesssim
\tau \sum_{I_1+I_2=I, \, |I_1|\geq 1} |\partial^{I_1} n \partial^{I_2} E|
\lesssim
(C_1 \epsilon)^2 \tau^{-{3\over 2} + \delta},
\end{align*}
in which we used the estimates in Proposition \ref{prop:refine-n-in001} and \eqref{eq:E-improve-in} in the last inequality.

Now, we bound $\mathcal{A}_3$. Recall that on a integral curve $\Gamma$ defined in \eqref{eq:integral-curve} one has
\begin{align*}
\big|{d\over d\lambda} n\big|
=
\big|\cancel{\partial}_0 n + {x^a \over \tau} \cancel{\partial}_a n \big|
\lesssim
|\partial n| + \tau^{-1} \sum_a |L_a n|
\lesssim
C_1 \epsilon \tau^{-1} \tau_-^{-{1\over 2}},
\end{align*}
in which we used the relation \eqref{eq:d-lambda} in the first inequality and the estimates in Proposition \ref{prop:refine-n-in001} in the last inequality.
Thus, we proceed to have
\begin{align*}
\mathcal{A}_3
\lesssim
(C_1 \epsilon)^2 \tau^{-{3\over 2}+\delta}.
\end{align*}

Gathering these estimates, we arrive at (for $|I|\leq N-4$)
\begin{align*}
|\tau \partial^I E |
\lesssim
C_1 \epsilon + C_1 \epsilon \int_{0}^{+\infty} (1+\lambda)^{-{3\over 2}+\delta} \, \d\lambda
\lesssim
C_1 \epsilon.
\end{align*}

With the aid of Proposition \ref{prop:KG-extra}, for $|K|\leq N-6$ one gets
\begin{align*}
|\partial^K E|
\lesssim
\sum_{|I|\leq N-4} {\tau_- \over t} |\partial^I E|
+
{1\over t} \sum_{|I_1|\leq N-5, \, |J|\leq 1} |\partial^{I_1} L^J E| 
+
 | \partial^K (n E)  |.
\end{align*}
As the quadratic terms decay faster, we have
\begin{align*}
 | \partial^K (n E)  |
 \lesssim
 &\sum_{K_1 + K_2 =K, \, |K_1|\geq 1}   | \partial^{K_1} n   |   | \partial^{K_2}  E | + |n \partial^K E|
\\
\lesssim
&(C_1 \epsilon)^2 t^{-2} + C_1 \epsilon |\partial^K E|,
\end{align*}
in which we used the estimates from \eqref{eq:E-improve-in} and Proposition \ref{prop:refine-n-in001},
 and which further yields
\begin{align}
|\partial^K E|
\lesssim
C_1 \epsilon t^{-{3\over 2}} \tau_-^{1\over 2},
\qquad
|K|\leq N-6.
\end{align}

The proof is done.
\end{proof}

\begin{proposition}[Refined estimates for $E$: I] \label{prop:refine-E-in001}
We have
\begin{align}
\mathcal{E}^{in}_1 (\partial^I  E, \tau)
\lesssim
\epsilon^2 + (C_1 \epsilon)^3,
\qquad
|I|  \leq N+1.
\end{align}

\end{proposition}

\begin{proof}
\textbf{Case I: $|I| \leq N$.} 
We first consider the case of $|I| \leq N$. By energy estimates in Proposition \ref{prop:EE-in2}, we have
\begin{equation}
\begin{aligned}
&\mathcal{E}^{in}_1 (\partial^I E, \tau, 0)
\\
\lesssim
&\mathcal{E}^{in}_1 (\partial^I E, \tau_0, 0)
+
\int_{\tau_0}^\tau \big\| \big( -\partial^I (nE) + n \partial^I E\big) \cdot (\widetilde{\tau}/t) \partial_t \partial^I E \big\|_{L^1(\mathcal{H}_{\widetilde{\tau}})} \, \d\widetilde{\tau}
\\
&+
\int_{\tau_0}^\tau \big\| | \partial^I E|^2 \cdot (\widetilde{\tau}/t) \partial_t n \big\|_{L^1(\mathcal{H}_{\widetilde{\tau}})} \, \d\widetilde{\tau} 
\\
&+ \int_{\Sigma^{bd}_{t^{bd}}} \Big( \sum_a ({x_a \over r} \partial_t \partial^I E + \partial_a \partial^I E)^2 + |\partial^I E|^2\Big)  \, \d S
\\
=
&\mathcal{G}_{11} + \mathcal{G}_{12} + \mathcal{G}_{13} + \mathcal{G}_{14}.
\end{aligned}
\end{equation}

Since $\mathcal{G}_{11}, \mathcal{G}_{14}$ can be bounded by initial data and the boundary estimates in Proposition \ref{prop:BDRY-4} established in Section \ref{sec:exterior}, we only consider terms in $\mathcal{G}_{12}, \mathcal{G}_{13}$ now.

We estimate the terms in $\mathcal{G}_{12}$. By the Leibniz rule, we have
\begin{align*}
-\partial^I (nE) + n \partial^I E
=
-\sum_{I_1+I_2=I, \, |I_1|\geq 1, |I_2|\leq N-1} \partial^{I_1} n \partial^{I_2} E.
\end{align*}
Applying Proposition \ref{prop:KG-extra}, we get
\begin{align*}
| \partial^{I_2} E|
\lesssim
{\tau^2 \over t^2} |\partial \partial \partial^{I_2} E|
+ 
{1\over t} \sum_{|I_3|\leq N-1, \,|J_1|\leq 1} |\partial \partial^{I_3} L^{J_1} E|
+
|\partial^{I_2} (nE)|.
\end{align*}
By the pointwise estimates in Proposition \ref{prop:refine-n-in001},  we further get
\begin{align*}
| \partial^{I_2} E|
\lesssim
{\tau^2 \over t^2} |\partial \partial \partial^{I_2} E|
+ 
{1\over t^{1\over 2}} \sum_{|I_3|\leq N-1, \,|J_1|\leq 1} |\partial \partial^{I_3} L^{J_1} E| ,
\end{align*}
which further yields
\begin{align*}
&\big\| -\partial^I (nE) + n \partial^I E\big\|_{L^2(\mathcal{H}_{\widetilde{\tau}})} 
\\
\lesssim
&\sum_{\substack{I_1+I_2=I\\ 1\leq |I_1|\leq N-3 }} \Big( \big\|{\widetilde{\tau}^2 \over t^2} \partial^{I_1} n \big\|_{L^\infty(\mathcal{H}_{\widetilde{\tau}})}  \|\partial \partial \partial^{I_2}  E \|_{L^2(\mathcal{H}_{\widetilde{\tau}})}  
\\
&\qquad\qquad\quad
+
\big\|t^{-{1\over 2}} \partial^{I_1} n \big\|_{L^\infty(\mathcal{H}_{\widetilde{\tau}})}  \sum_{\substack{|I_3|\leq N-1\\|J_1|\leq 1}}\|\partial \partial^{I_3} L^{J_1} E \|_{L^2(\mathcal{H}_{\widetilde{\tau}})}  \Big)
\\
&+
\sum_{I_1+I_2=I, \, |I_1|\geq 3} \big\|(\widetilde{\tau} / t) \tau_-^{1\over 2} \partial^{I_1} n \big\|_{L^2(\mathcal{H}_{\widetilde{\tau}})}  \big\|(t / \widetilde{\tau} ) \tau_-^{-{1\over 2}} \partial^{I_2} E \big\|_{L^\infty(\mathcal{H}_{\widetilde{\tau}})} 
\\
\lesssim
& (C_1 \epsilon)^2 \widetilde{\tau}^{-{3\over 2}+\delta},
\end{align*}
in which we used the bounds for $n$ from Proposition \ref{prop:refine-n-in001} and the estimates for $E$ from \eqref{eq:BA-in3} and \eqref{eq:E-improve-in}.
Therefore, we obtain
\begin{align*}
&\big\| \big( -\partial^I (nE) + n \partial^I E\big) \cdot (\widetilde{\tau}/t) \partial_t \partial^I E \big\|_{L^1(\mathcal{H}_{\widetilde{\tau}})} 
\\
\lesssim
 &\big\|  -\partial^I (nE) + n \partial^I E \big\|_{L^2(\mathcal{H}_{\widetilde{\tau}})}   \|(\widetilde{\tau}/t) \partial_t \partial^I E \|_{L^2(\mathcal{H}_{\widetilde{\tau}})} 
\\
\lesssim
&(C_1 \epsilon)^3\widetilde{\tau}^{-{3\over 2}+2\delta}.
\end{align*}
As a consequence, we derive
\begin{align}
\mathcal{G}_{12}
\lesssim
(C_1 \epsilon)^3 \int_{\tau_0}^\tau \widetilde{\tau}^{-{3\over 2}+2\delta} \, \d\widetilde{\tau}
\lesssim
(C_1 \epsilon)^3.
\end{align}

Next, we bound $\mathcal{G}_{13}$. Applying Proposition \ref{prop:KG-extra}, we get
\begin{align*}
| \partial^I E|
\lesssim
{\tau^2 \over t^2} |\partial \partial \partial^I E|
+ 
{1\over t} \sum_{|I_1|\leq |I|, \,|J_1|\leq 1} |\partial \partial^{I_1} L^{J_1} E|
+
|\partial^I (nE)|,
\end{align*}
which, together with the pointwise estimates in Proposition \ref{prop:refine-n-in001},  further implies
\begin{align*}
| \partial^I E|
\lesssim
{\tau^2 \over t^2} |\partial \partial \partial^I E|
+ 
{1\over t^{1\over 2}} \sum_{|I_1|\leq |I|, \,|J_1|\leq 1} |\partial \partial^{I_1} L^{J_1} E|.
\end{align*}
Consequently, we obtain
\begin{align*}
| \partial^I E|^2
\lesssim
{\tau^4 \over t^4} |\partial \partial \partial^I E|^2
+ 
{1\over t} \sum_{|I_1|\leq |I|, \,|J_1|\leq 1} |\partial \partial^{I_1} L^{J_1} E|^2,
\end{align*}
which leads us to
\begin{align*}
&\big\| | \partial^I E|^2 \cdot (\widetilde{\tau}/t) \partial_t n \big\|_{L^1(\mathcal{H}_{\widetilde{\tau}})}
\\
\lesssim
&C_1 \epsilon \widetilde{\tau}^{-{3\over 2}} \big\| (\widetilde{\tau}/t) \partial \partial \partial^I E \big\|^2_{L^2(\mathcal{H}_{\widetilde{\tau}})} 
+
C_1 \epsilon \widetilde{\tau}^{-{3\over 2}} \sum_{|I_1|\leq |I|, \,|J_1|\leq 1} \big\| (\widetilde{\tau}/t) \partial \partial^{I_1} L^{J_1} E \big\|^2_{L^2(\mathcal{H}_{\widetilde{\tau}})} 
\\
\lesssim
&(C_1 \epsilon)^3 \widetilde{\tau}^{-{3\over 2} + 2\delta},
\end{align*}
in which we used pointwise decay of $\partial_t n$ in Proposition \ref{prop:refine-n-in001} and the estimates in bootstrap assumption \eqref{eq:BA-in3}.
Therefore, we arrive at
\begin{align}
\mathcal{G}_{13}
\lesssim
(C_1 \epsilon)^3 \int_{\tau_0}^{\tau} \widetilde{\tau}^{-{3\over 2} + 2\delta} \, \d\widetilde{\tau}
\lesssim
(C_1 \epsilon)^3.
\end{align}

\textbf{Case II: $|I| = N+1$.} 
To show uniform energy bound in this case, we first derive a weaker one. We apply Proposition \ref{prop:EE-in2} with $\gamma = 2\delta$, and we get
\begin{equation}
\begin{aligned}
&\mathcal{E}^{in}_1 (\partial^I E, \tau, 2\delta)
\\
\lesssim
&\mathcal{E}^{in}_1 (\partial^I E, \tau_0, 2\delta)
+
\int_{\tau_0}^\tau \big\| \tau_-^{-4\delta} (-\partial^I (n E) + n \partial^I E) \cdot (\widetilde{\tau}/t) \partial_t \partial^I E \big\|_{L^1(\mathcal{H}_{\widetilde{\tau}})} \, \d \widetilde{\tau}
\\
&+
\int_{\tau_0}^\tau \big\| \tau_-^{-4\delta} |\partial^I E|^2 (\widetilde{\tau}/t) \partial_t n \big\|_{L^1(\mathcal{H}_{\widetilde{\tau}})} 
+ \big\| \tau_-^{-4\delta-1} |\partial^I E|^2 (\widetilde{\tau}/t) n \big\|_{L^1(\mathcal{H}_{\widetilde{\tau}})} 
\, \d \widetilde{\tau}
\\
&+ 
 \int_{\Sigma^{bd}_{t^{bd}}}  \Big( \sum_a ({x_a \over r} \partial_t \partial^I E + \partial_a \partial^I E)^2 + |\partial^I E|^2 \Big) \, \d S
\\
=:
& \mathcal{G}_{21} + \mathcal{G}_{22} + \mathcal{G}_{23} + \mathcal{G}_{24}.
\end{aligned}
\end{equation}

Recall that the terms in $\mathcal{G}_{21}, \mathcal{G}_{24}$ can be bounded by initial data and the estimates from Section \ref{sec:exterior}, we only consider terms in $\mathcal{G}_{22}, \mathcal{G}_{23}$.

We bound the $\mathcal{G}_{22}$ term. Using the Leibniz rule, we have
\begin{align*}
-\partial^I (nE) + n \partial^I E
=
-\sum_{I_1+I_2=I, \, |I_1|\geq 1, |I_2|\leq N} \partial^{I_1} n \partial^{I_2} E.
\end{align*}
By Proposition \ref{prop:KG-extra}, we get
\begin{align*}
| \partial^{I_2} E|
\lesssim
&{\tau^2 \over t^2} |\partial \partial \partial^{I_2} E|
+ 
{1\over t} \sum_{|I_3|\leq N, \,|J_1|\leq 1} |\partial \partial^{I_3} L^{J_1} E|
+
|\partial^{I_2} (nE)|
\\
\lesssim
&\sum_{|I_3|+|J_1|\leq N+1, \, |J_1|\leq 1} {\tau^2 \over t^2} |\partial \partial^{I_3} L^{J_1} E|
+
\sum_{|I_{21}| + |I_{22}| = N} |\partial^{I_{21}} n \partial^{I_{22}} E|,
\end{align*}
which, combining the smallness of $|n|$, further gives
\begin{equation}\label{eq:phi-01-in}
\begin{aligned}
&| \partial^{I_2} E|
\\
\lesssim
&\sum_{|I_3|+|J_1|\leq N+1, \, |J_1|\leq 1} {\tau^2 \over t^2} |\partial \partial^{I_3} L^{J_1} E|
+
\sum_{|I_{21}| + |I_{22}| \leq N, \, |I_{21}|\geq 1} |\partial^{I_{21}} n \partial^{I_{22}} E|.
\end{aligned}
\end{equation}
Therefore, we have
\begin{equation}\label{eq:phi-02-in}
\begin{aligned}
&\big\|  -\partial^I (nE) + n \partial^I E \big\|_{L^2(\mathcal{H}_{\widetilde{\tau}})} 
\\
\lesssim
& \sum_{\substack{|I_1|+|I_3|\leq N+1,\, |I_1|\geq 1 \\|I_3|+|J_1|\leq N+1, \,|J_1|\leq 1}} \big\| \partial^{I_1} n (\widetilde{\tau} / t)^2 \partial \partial^{I_3} L^{J_1} E \big\|_{L^2(\mathcal{H}_{\widetilde{\tau}})}  
\\
&+
\sum_{\substack{|I_1| + |I_{21}| + |I_{22}| \leq N+1, \, |I_1|\geq 1\\ |I_{21}| + |I_{22}| \leq N, \, |I_{21}|\geq 1}} \|\partial^{I_1} n \partial^{I_{21}} n \partial^{I_{22}} E  \|_{L^2(\mathcal{H}_{\widetilde{\tau}})}  
\\
\lesssim
& (C_1 \epsilon)^2 \widetilde{\tau}^{-{3\over 2}+\delta}.
\end{aligned}
\end{equation}
Consequently, we obtain
\begin{align*}
&\big\| \tau_-^{-4\delta} \big( -\partial^I (nE) + n \partial^I E\big) \cdot (\widetilde{\tau}/t) \partial_t \partial^I E \big\|_{L^1(\mathcal{H}_{\widetilde{\tau}})} 
\\
\lesssim
&\big\| \big( -\partial^I (nE) + n \partial^I E\big) \big\|_{L^2(\mathcal{H}_{\widetilde{\tau}})}  \| (\widetilde{\tau}/t) \partial_t \partial^I E \|_{L^2(\mathcal{H}_{\widetilde{\tau}})} 
\\
\lesssim
&(C_1 \epsilon)^3\widetilde{\tau}^{-{3\over 2}+2\delta}.
\end{align*}
Based on these estimates, we arrive at
\begin{align}
\mathcal{G}_{22}
\lesssim
(C_1 \epsilon)^3 \int_{\tau_0}^\tau \widetilde{\tau}^{-{3\over 2}+2\delta} \, \d\widetilde{\tau}
\lesssim
(C_1 \epsilon)^3.
\end{align}

We now bound the $\mathcal{G}_{23}$ term. We first note
\begin{align*}
&\big\| \tau_-^{-4\delta} |\partial^I E|^2 (\widetilde{\tau}/t) \partial_t n \big\|_{L^1(\mathcal{H}_{\widetilde{\tau}})} 
+ \big\| \tau_-^{-4\delta-1} |\partial^I E|^2 (\widetilde{\tau}/t) n \big\|_{L^1(\mathcal{H}_{\widetilde{\tau}})} 
\\
\lesssim
&\big\| \tau_-^{-1-4\delta} (\widetilde{\tau}/t)|\partial^I E|^2  \big\|_{L^1(\mathcal{H}_{\widetilde{\tau}})} \big(  \| \tau_- \partial_t n\|_{L^\infty(\mathcal{H}_{\widetilde{\tau}})}  
+ \|  n\|_{L^\infty(\mathcal{H}_{\widetilde{\tau}})}   \big)
\\
\lesssim
& C_1 \epsilon \big\| \tau_-^{-1-4\delta} (\widetilde{\tau}/t)|\partial^I E|^2  \big\|_{L^1(\mathcal{H}_{\widetilde{\tau}})}.
\end{align*}
Therefore, by smallness of $C_1 \epsilon$, the term $\mathcal{G}_{23}$ can be absorbed by the left hand side of the spacetime energy.

To conclude with the case of $\gamma = 2\delta$, we have
\begin{align}\label{eq:phi-0-in}
\mathcal{E}^{in}_1 (\partial^I E, \tau, 2\delta)
\lesssim
\epsilon^2 + (C_1 \epsilon)^3.
\end{align}

Next, we set $\gamma=0$, and want to derive 
\begin{align*}
\mathcal{E}^{in}_1 (\partial^I E, \tau, 0)
\lesssim
\epsilon^2 + (C_1 \epsilon)^3,
\end{align*}
with the aid of \eqref{eq:phi-0-in}.

We have
\begin{equation}
\begin{aligned}
&\mathcal{E}^{in}_1 (\partial^I E, \tau, 0)
\\
\lesssim
&\mathcal{E}^{in}_1 (\partial^I E, \tau_0, 0)
+
\int_{\tau_0}^\tau \big\| \big( -\partial^I (nE) + n \partial^I E\big) \cdot (\widetilde{\tau}/t) \partial_t \partial^I E \big\|_{L^1(\mathcal{H}_{\widetilde{\tau}})} \, \d\widetilde{\tau}
\\
&+
\int_{\tau_0}^\tau \big\| | \partial^I E|^2 \cdot (\widetilde{\tau}/t) \partial_t n \big\|_{L^1(\mathcal{H}_{\widetilde{\tau}})} \, \d\widetilde{\tau} 
\\
 &+
 \int_{\Sigma^{bd}_{t^{bd}}} \Big( \sum_a ({x_a \over r} \partial_t \partial^I E + \partial_a \partial^I E)^2 + |\partial^I E|^2 \Big) \, \d S
\\
=
&\mathcal{G}_{31} + \mathcal{G}_{32} + \mathcal{G}_{33} + \mathcal{G}_{34}.
\end{aligned}
\end{equation}

Note that $\mathcal{G}_{31}, \mathcal{G}_{34}$ can be bounded by the initial data and the estimates from Section \ref{sec:exterior}, we only treat terms in $\mathcal{G}_{32}, \mathcal{G}_{33}$ here.

We estimate the terms in $\mathcal{G}_{32}$. By the Leibniz rule, we have
\begin{align*}
-\partial^I (nE) + n \partial^I E
=
-\sum_{I_1+I_2=I, \, |I_1|\geq 1, |I_2|\leq N} \partial^{I_1} n \partial^{I_2} E.
\end{align*}
By \eqref{eq:phi-01-in} and \eqref{eq:phi-02-in}, we get
\begin{align}
\big\|  -\partial^I (nE) + n \partial^I E \big\|_{L^2(\mathcal{H}_{\widetilde{\tau}})} 
\lesssim
 (C_1 \epsilon)^2 \widetilde{\tau}^{-{3\over 2}+\delta}.
\end{align}
Therefore, we obtain
\begin{align*}
&\big\| \big( -\partial^I (nE) + n \partial^I E\big) \cdot (\widetilde{\tau}/t) \partial_t \partial^I E \big\|_{L^1(\mathcal{H}_{\widetilde{\tau}})} 
\\
\lesssim
&\big\| \big( -\partial^I (nE) + n \partial^I E\big) \big\|_{L^2(\mathcal{H}_{\widetilde{\tau}})}  \| (\widetilde{\tau}/t) \partial_t \partial^I E \|_{L^2(\mathcal{H}_{\widetilde{\tau}})} 
\\
\lesssim
&(C_1 \epsilon)^3\widetilde{\tau}^{-{3\over 2}+2\delta}.
\end{align*}
Based on these estimates, we arrive at
\begin{align}
\mathcal{G}_{32}
\lesssim
(C_1 \epsilon)^3 \int_{\tau_0}^\tau \widetilde{\tau}^{-{3\over 2}+2\delta} \, \d\widetilde{\tau}
\lesssim
(C_1 \epsilon)^3.
\end{align}

For the $\mathcal{G}_{33}$ term, we note
\begin{align*}
&\big\|  |\partial^I E|^2 (\widetilde{\tau}/t) \partial_t n\big\|_{L^1(\mathcal{H}_{\widetilde{\tau}})} 
\\
\lesssim
&\big\| \tau_-^{-1-4\delta} (\widetilde{\tau}/t)|\partial^I E|^2  \big\|_{L^1(\mathcal{H}_{\widetilde{\tau}})}   \big\| \tau_-^{1+4\delta} \partial_t n \big\|_{L^\infty(\mathcal{H}_{\widetilde{\tau}})}  
\\
\lesssim
& C_1 \epsilon \big\| \tau_-^{-1-4\delta} (\widetilde{\tau}/t)|\partial^I E|^2  \big\|_{L^1(\mathcal{H}_{\widetilde{\tau}})}.
\end{align*}
Therefore, with the smallness factor of $C_1 \epsilon$, the term $\mathcal{G}_{33}$ can be bounded by the spacetime integral in the energy $\mathcal{E}^{in}_1 (\partial^I E, \tau, 2\delta)$.

Gathering these estimates, the proof is done.
\end{proof}

\begin{proposition}[Refined estimates for $E$: II]\label{prop:refine-E-in002}
We have
\begin{align}\label{eq:refine-E-in002}
\mathcal{E}^{in}_1 (\partial^I L^J E, \tau)
\lesssim
\epsilon^2 + (C_1 \epsilon)^3 \log^{2|J| }(1+\tau),
\qquad
|I| + |J| \leq N+1, \quad |J|\leq N.
\end{align}
\end{proposition}

\begin{proof}
For the case $|J|=0$, it was proved in Proposition \ref{prop:refine-E-in001}. We now consider general cases   of $|I|+|J| \leq N+1$ except the top-order estimates of $|J|=N+1$. We prove it by an induction argument.

We assume \eqref{eq:refine-E-in002} is valid for all $|I|+|J|\leq N+1$ with $|J|\leq k$ where $0\leq k \leq N-1$. We now want to show \eqref{eq:refine-E-in002} is also true for all $|I|+|J|\leq N+1$ with $|J|\leq k+1$.


We apply Proposition \ref{prop:EE-in2} with $\gamma = 0$ to get
\begin{equation}
\begin{aligned}
&\mathcal{E}^{in}_1 (\partial^I L^J E, \tau, 0)
\\
\lesssim
&\mathcal{E}^{in}_1 (\partial^I L^J E, \tau_0, 0)
+
\int_{\tau_0}^\tau \big\| (-\partial^I L^J (n E) + n \partial^I L^J E) \cdot (\widetilde{\tau}/t) \partial_t \partial^I L^J E \big\|_{L^1(\mathcal{H}_{\widetilde{\tau}})} \, \d \widetilde{\tau}
\\
&+
\int_{\tau_0}^\tau \big\|  |\partial^I L^J E|^2 (\widetilde{\tau}/ t) \partial_t n \big\|_{L^1(\mathcal{H}_{\widetilde{\tau}})} \, \d \widetilde{\tau}
\\
&+  \int_{\Sigma^{bd}_{t^{bd}}} \Big( \sum_a ({x_a \over r} \partial_t \partial^I L^J E + \partial_a \partial^I L^J E)^2 + |\partial^I L^J E|^2 \Big) \, \d S
\\
=:
& \mathcal{G}_{41} + \mathcal{G}_{42} + \mathcal{G}_{43}+ \mathcal{G}_{44}.
\end{aligned}
\end{equation}

The estimates for terms in $\mathcal{G}_{41}, \mathcal{G}_{43},\mathcal{G}_{44}$ can be done as before, and we only treat the term $\mathcal{G}_{42}$.

By Leibniz rule, we have
\begin{align*}
\big|-\partial^I L^J (n E) + n \partial^I L^J E \big|
\leq
\sum_{\substack{I_1 + I_2 = I,\, J_1+J_2= J \\  |I_2|+|J_2|< |I|+|J|}} | \partial^{I_1}L^{J_1} n \, \partial^{I_2} L^{J_2} E|.
\end{align*}

We estimate the nonlinear terms for different cases of $I$ and $k$.

\textbf{Case I: $|I|=0, k\leq N-3$.}

In this case, we find
\begin{align*}
&\big\|  (-L^J (n E) + n L^J E)  \big\|_{L^2(\mathcal{H}_{\widetilde{\tau}})}
\\
\lesssim
& \sum_{\substack{|J_1|+|J_2|=k+1 \\ |J_1|\geq 1, \, |J_2|\leq k}} \| L^{J_1} n \|_{L^\infty(\mathcal{H}_{\widetilde{\tau}})} \|L^{J_2}  E \|_{L^2(\mathcal{H}_{\widetilde{\tau}})}
\\
\lesssim
&(C_1 \epsilon)^2 \widetilde{\tau}^{-1} \log^{k} (1+\widetilde{\tau}),
\end{align*}
in which we used the pointwise bounds for $n$ in Proposition \ref{prop:refine-n-in001} and $L^2$ bounds for $E$ from induction estimates.

\textbf{Case II: $|I|=0, k\geq N-2$.}

In this case, we first apply the Sobolev inequality in Proposition \ref{prop:Sobolev-in} with induction assumption to derive that
\begin{align*}
\sum_{|K|\leq 3} | L^{K} E|
\lesssim
C_1 \epsilon t^{-1} \log^5 (1+\tau).
\end{align*}

In succession, we find
\begin{align*}
&\big\|  (-L^J (n E) + n L^J E)  \big\|_{L^2(\mathcal{H}_{\widetilde{\tau}})}
\\
\lesssim
& \sum_{\substack{|J_1|+|J_2|=|J| \\ 1\leq |J_1|\leq N-2, \, |J_2|\leq k}} \| L^{J_1} n \|_{L^\infty(\mathcal{H}_{\widetilde{\tau}})} \|L^{J_2}  E \|_{L^2(\mathcal{H}_{\widetilde{\tau}})}
\\
&+
 \sum_{\substack{|J_1|+|J_2|=|J|\\  1\leq |J_1|\leq k+1, \, |J_2|\leq 3}} \| (\widetilde{\tau}/t)L^{J_1} n \|_{L^2(\mathcal{H}_{\widetilde{\tau}})} \|(t/\widetilde{\tau})L^{J_2}  E \|_{L^\infty(\mathcal{H}_{\widetilde{\tau}})}
\\
\lesssim
&(C_1 \epsilon)^2 \widetilde{\tau}^{-1} \log^{k} (1+\widetilde{\tau}),
\end{align*}
in which we used the estimates in Proposition \ref{prop:refine-n-in001} and  induction estimates.

\textbf{Case III: $|I|\geq 1, k\leq N-4$.}

We note that
\begin{align*}
&\big\|  (-\partial^I L^J (n E) + n \partial^I L^J E)  \big\|_{L^2(\mathcal{H}_{\widetilde{\tau}})}
\\
\lesssim
& \sum_{|I_1|+|I_2|=|I|, \, |I_2|\leq |I|-1} \| \partial^{I_1} n \partial^{I_2} L^{J} E \|_{L^2(\mathcal{H}_{\widetilde{\tau}})}
\\
&+ \sum_{|J_1|+|J_2|=|J|, \, |J_2|\leq k} \| L^{J_1}n \partial^{I} L^{J_2} E \|_{L^2(\mathcal{H}_{\widetilde{\tau}})}
\\
&+
\sum_{\substack{|I_1|+|I_2|=|I|, |J_1|+|J_2|=|J|\\ |I_1|\geq 1, |J_2|\leq k}} \| \partial^{I_1} L^{J_1}n \partial^{I_2} L^{J_2} E \|_{L^2(\mathcal{H}_{\widetilde{\tau}})}
\\
=
& \mathcal{A}_1 + \mathcal{A}_2 + \mathcal{A}_3.
\end{align*}

For $\mathcal{A}_1$, we note that
\begin{align*}
\mathcal{A}_1
\lesssim
&\sum_{\substack{|I_1|+|I_2|=|I|\\  1\leq |I_1|\leq |I_2|+|J|+1}} \|(\widetilde{\tau}/t) \partial^{I_1} n\|_{L^\infty(\mathcal{H}_{\widetilde{\tau}})} \|(t/\widetilde{\tau}) \partial^{I_2} L^{J} E \|_{L^2(\mathcal{H}_{\widetilde{\tau}})}
\\
&+
\sum_{\substack{|I_1|+|I_2|=|I|\\   |I_1|> |I_2|+|J|+1}} \big\|(\widetilde{\tau}/t) \tau_-^{1\over 2} \partial^{I_1} n \big\|_{L^2(\mathcal{H}_{\widetilde{\tau}})} \big\|(t/\widetilde{\tau}) \tau_-^{-{1\over 2}} \partial^{I_2} L^{J} E \big\|_{L^\infty(\mathcal{H}_{\widetilde{\tau}})}
\\
\lesssim
 & (C_1 \epsilon)^2 \widetilde{\tau}^{-{5\over 4}+\delta}.
\end{align*}

For $\mathcal{A}_2$, we have
\begin{align*}
\mathcal{A}_2
\lesssim
& \sum_{|J_1|+|J_2|=|J|, \, |J_2|\leq k} \| L^{J_1}n \|_{L^\infty(\mathcal{H}_{\widetilde{\tau}})} \| \partial^{I} L^{J_2} E \|_{L^2(\mathcal{H}_{\widetilde{\tau}})}
\\
\lesssim
&(C_1 \epsilon)^2 \widetilde{\tau}^{-1} \log^k(1+\widetilde{\tau}).
\end{align*}

We bound $\mathcal{A}_3$ to get
\begin{align*}
\mathcal{A}_3
\lesssim
&\sum_{\substack{|I_1|+|I_2|=|I|, |J_1|+|J_2|=|J|\\ |I_1|+|J_1|\leq |I_2|+|J_2|+1, \, |I_1|\geq 1, |J_2|\leq k}} \| \partial^{I_1} L^{J_1}n\|_{L^\infty(\mathcal{H}_{\widetilde{\tau}})} \| \partial^{I_2} L^{J_2} E \|_{L^2(\mathcal{H}_{\widetilde{\tau}})}
\\
+
&\sum_{\substack{|I_1|+|I_2|=|I|, |J_1|+|J_2|=|J|\\ |I_1|+|J_1|> |I_2|+|J_2|+1, \, |I_1|\geq 1, |J_2|\leq k}} \big\| (\widetilde{\tau}/t) \tau_-^{1\over 2} \partial^{I_1} L^{J_1}n \big\|_{L^2(\mathcal{H}_{\widetilde{\tau}})} \big\| (t/\widetilde{\tau}) \tau_-^{-{1\over 2}}\partial^{I_2} L^{J_2} E \big\|_{L^\infty(\mathcal{H}_{\widetilde{\tau}})}
\\
\lesssim
&(C_1 \epsilon)^2 \widetilde{\tau}^{-1} \log^k(1+\widetilde{\tau}).
\end{align*}

Therefore, we get
\begin{align}
\big\|  (-\partial^I L^J (n E) + n \partial^I L^J E)  \big\|_{L^2(\mathcal{H}_{\widetilde{\tau}})}
\lesssim
(C_1 \epsilon)^2 \widetilde{\tau}^{-1} \log^k(1+\widetilde{\tau}).
\end{align}

\textbf{Case IV: $|I|\geq 1, k\geq N-3$.}

In this case, the Sobolev inequality in Proposition \ref{prop:Sobolev-in} and induction assumption give us
\begin{align*}
\sum_{|I|+|K|\leq N-2, \,|K|\leq k-2} |\partial^I L^{K} E|
\lesssim
C_1 \epsilon t^{-1} \log^k (1+\tau).
\end{align*}

In succession, we get
\begin{align*}
&\big\|  (-\partial^I L^J (n E) + n \partial^I L^J E)  \big\|_{L^2(\mathcal{H}_{\widetilde{\tau}})}
\\
\lesssim
& \sum_{|I_1|+|I_2|=|I|, \, |I_2|\leq |I|-1} \| \partial^{I_1} n \partial^{I_2} L^{J} E \|_{L^2(\mathcal{H}_{\widetilde{\tau}})}
\\
& \sum_{|J_1|+|J_2|=|J|, \, |J_2|\leq k} \| L^{J_1}n \partial^{I} L^{J_2} E \|_{L^2(\mathcal{H}_{\widetilde{\tau}})}
\\
&+
\sum_{\substack{|I_1|+|I_2|=|I|, |J_1|+|J_2|=|J|\\ |I_1|\geq 1, |J_2|\leq k}} \| \partial^{I_1} L^{J_1}n \partial^{I_2} L^{J_2} E \|_{L^2(\mathcal{H}_{\widetilde{\tau}})}
\\
=
& \mathcal{B}_1 + \mathcal{B}_2 + \mathcal{B}_3.
\end{align*}

To bound $\mathcal{B}_1$, we note
\begin{align*}
\mathcal{B}_1 = &\sum_{|I_1|+|I_2|=|I|, \, |I_2|\leq |I|-1} \| \partial^{I_1} n \partial^{I_2} L^{J} E \|_{L^2(\mathcal{H}_{\widetilde{\tau}})}
\\
\lesssim
&\sum_{|I_1|+|I_2|=|I|, \, |I_2|\leq |I|-1} \| (\widetilde{\tau}/t)\partial^{I_1} n \|_{L^\infty(\mathcal{H}_{\widetilde{\tau}})} \|(t/\widetilde{\tau})\partial^{I_2} L^{J} E \|_{L^2(\mathcal{H}_{\widetilde{\tau}})}
\\
\lesssim
&C_1 \epsilon \widetilde{\tau}^{-{5\over 4}} \sum_{|I_2|\leq |I|-1} \|(t/\widetilde{\tau})\partial^{I_2} L^{J} E \|_{L^2(\mathcal{H}_{\widetilde{\tau}})}.
\end{align*}
Recall that we apply Proposition \ref{prop:KG-extra} to get
\begin{align*}
&\sum_{|I_2|\leq |I|-1} \|(t/\widetilde{\tau})\partial^{I_2} L^{J} E \|_{L^2(\mathcal{H}_{\widetilde{\tau}})}
\\
\lesssim
&\sum_{\substack{|I_3|+|J_1|\leq |I|+|J|+1\\|J_1|\leq |J|+1}} \| (\widetilde{\tau}/t)\partial^{I_3} L^{J_1} E \|_{L^2(\mathcal{H}_{\widetilde{\tau}})}
+
\sum_{|I_2|\leq |I|-1} \| (t/\widetilde{\tau})\partial^{I_2} L^{J} (nE) \|_{L^2(\mathcal{H}_{\widetilde{\tau}})}
\\
\lesssim
& C_1 \epsilon \widetilde{\tau}^\delta.
\end{align*}
Hence, we obtain
\begin{align}
\mathcal{B}_1
\lesssim 
(C_1 \epsilon)^2 \widetilde{\tau}^{-{5\over 4}+\delta}.
\end{align}

Next, we estimate $\mathcal{B}_2$. We note
\begin{align*}
\mathcal{B}_2
\lesssim
&\sum_{\substack{|J_1|+|J_2|=|J|\\ 1\leq |J_1|\leq N-3, |J_2|\leq k}} \| L^{J_1}n \|_{L^\infty(\mathcal{H}_{\widetilde{\tau}})} \| \partial^{I} L^{J_2} E \|_{L^2(\mathcal{H}_{\widetilde{\tau}})}
\\
&+
\sum_{|J_1|+|J_2|=|J|, \, |J_1|\geq N-3} \|(\widetilde{\tau}/t) L^{J_1}n \|_{L^2(\mathcal{H}_{\widetilde{\tau}})} \|(t/\widetilde{\tau}) \partial^{I} L^{J_2} E \|_{L^\infty(\mathcal{H}_{\widetilde{\tau}})}
\\
\lesssim
&(C_1 \epsilon)^2 \widetilde{\tau}^{-1} \log^k(1+\widetilde{\tau}).
\end{align*}

Now, we bound $\mathcal{B}_3$. We have
\begin{align*}
\mathcal{B}_3
\lesssim
&\sum_{\substack{|I_1|+|I_2|=|I|, |J_1|+|J_2|=|J|\\|I_1|+|J_1|\leq |I_2|+|J_2|+1, |I_1|\geq 1, |J_2|\leq k}}    \| \partial^{I_1}L^{J_1} n \|_{L^\infty(\mathcal{H}_{\widetilde{\tau}})} \|\partial^{I_2} L^{J_2} E \|_{L^2(\mathcal{H}_{\widetilde{\tau}})}
\\
&+
\sum_{\substack{|I_1|+|I_2|=|I|, |J_1|+|J_2|=|J|\\|I_1|+|J_1|> |I_2|+|J_2|+1, |I_1|\geq 1, |J_2|\leq k}}  \big\| (\widetilde{\tau}/t) \tau_-^{1\over 2} \partial^{I_1}L^{J_1} n \big\|_{L^2(\mathcal{H}_{\widetilde{\tau}})}  \big\| (t/\widetilde{\tau}) \tau_-^{-{1\over 2}} \partial^{I_2} L^{J_2} E \big\|_{L^\infty(\mathcal{H}_{\widetilde{\tau}})}
\\
\lesssim
& (C_1 \epsilon)^2 \widetilde{\tau}^{-1} \log^k(1+\widetilde{\tau}).
\end{align*}

Gathering the estimates for $\mathcal{B}_1, \mathcal{B}_2, \mathcal{B}_3$, we arrive at
\begin{align}
\big\|  (-\partial^I L^J (n E) + n \partial^I L^J E)  \big\|_{L^2(\mathcal{H}_{\widetilde{\tau}})}
\lesssim
(C_1 \epsilon)^2 \widetilde{\tau}^{-1} \log^k(1+\widetilde{\tau}).
\end{align}
By this, we derive
\begin{align*}
\mathcal{G}_{42}
\lesssim
(C_1 \epsilon)^2 \int_{\tau_0}^\tau \widetilde{\tau}^{-1} \log^k(1+\widetilde{\tau}) \, \d\widetilde{\tau}
\lesssim
(C_1 \epsilon)^2 \log^{k+1}(1+\tau).
\end{align*}

The proof is done.
\end{proof}

\begin{lemma}[Estimates for $\mathfrak{m}$: II]\label{lem:nD-L2-in03}
We have
\begin{align}
\|  (\tau/t)\partial\partial L^J \mathfrak{m} \|_{L^2(\mathcal{H}_\tau)}
\lesssim
&\epsilon + (C_1 \epsilon)^2 \tau^\delta,
\qquad
& |J|= N+1, \label{eq:nD-L2-in03}
\\
\|  (\tau/t) L^J n \|_{L^2(\mathcal{H}_\tau)}
\lesssim
&\epsilon + (C_1 \epsilon)^2 \tau^\delta,
\qquad
& |J|= N+1. \label{eq:nD-L2-in03b}
\end{align}
\end{lemma}

\begin{proof}
We apply Proposition \ref{prop:EE-in2} with $\gamma = 0$ to the equation of $\partial L^J \mathfrak{m}$, and find
\begin{equation}
\begin{aligned}
&\mathcal{E}^{in}(\partial L^J \mathfrak{m}, \tau)
\\
\lesssim
&\mathcal{E}^{in}(\partial L^J \mathfrak{m}, \tau_0)
+
\int_{\tau_0}^\tau \big\| \partial L^J |E|^2 \big\|_{L^2(\mathcal{H}_{\widetilde{\tau}})} \mathcal{E}^{in}(\partial L^J \mathfrak{m}, \widetilde{\tau})^{1\over 2}\, \d\widetilde{\tau}
\\
&+  \int_{\Sigma^{bd}_{t^{bd}}} \Big( \sum_a ({x_a \over r} \partial_t  \partial L^J \mathfrak{m} + \partial_a  \partial L^J \mathfrak{m})^2 \Big) \, \d S
\\
=: &\mathcal{G}_{51} + \mathcal{G}_{52} + \mathcal{G}_{53}.
\end{aligned}
\end{equation}
Since the terms in $\mathcal{G}_{51}, \mathcal{G}_{53}$ can be bounded by the initial data and the boundary estimates from Section \ref{sec:exterior}, we will only focus on the estimates on the term $\mathcal{G}_{52}$.

Applying the Leibniz rule, we find that
\begin{align*}
&\big\|\partial L^J |E|^2 \big\|_{L^2(\mathcal{H}_{\widetilde{\tau}})}
\\
\lesssim
&\sum_{\substack{|J_1|+|J_2|=|J|\\ 1\leq |J_1|\leq |J_2|\leq N }} \big( \|\partial L^{J_1} E \|_{L^\infty(\mathcal{H}_{\widetilde{\tau}})} \| L^{J_2} E \|_{L^2(\mathcal{H}_{\widetilde{\tau}})}
+
 \|\partial L^{J_2} E \|_{L^2(\mathcal{H}_{\widetilde{\tau}})} \| L^{J_1} E \|_{L^\infty(\mathcal{H}_{\widetilde{\tau}})} \big)
\\
+
& \|(\widetilde{\tau}/t) \partial L^{J} E \|_{L^2(\mathcal{H}_{\widetilde{\tau}})} \| (t/\widetilde{\tau}) E \|_{L^\infty(\mathcal{H}_{\widetilde{\tau}})}
+
 \|\partial  E \|_{L^\infty(\mathcal{H}_{\widetilde{\tau}})} \| L^{J} E \|_{L^2(\mathcal{H}_{\widetilde{\tau}})}
\\
=: &\mathcal{A}_1 + \mathcal{A}_2 + \mathcal{A}_3.
\end{align*}

By Proposition \ref{prop:refine-E-in002}, we get
\begin{align*}
\mathcal{A}_1
\lesssim
(C_1 \epsilon)^2 \widetilde{\tau}^{-1} \log^{N+3}(1+\widetilde{\tau}).
\end{align*}

By bootstrap assumption \eqref{eq:BA-in3} and Proposition \ref{prop:refine-E-decay}, we can show that
\begin{align*}
\mathcal{A}_2 + \mathcal{A}_3
\lesssim 
(C_1 \epsilon)^2 \widetilde{\tau}^{-1+\delta}. 
\end{align*}

In summary, we derive that
\begin{align*}
\big\|\partial L^J |E|^2 \big\|_{L^2(\mathcal{H}_{\widetilde{\tau}})}
\lesssim
(C_1 \epsilon)^2 \widetilde{\tau}^{-1+\delta},
\end{align*}
which further leads to \eqref{eq:nD-L2-in03}.

Combining \eqref{eq:nD-L2-in03} and Lemma \ref{lem:n0-in}, we obtain \eqref{eq:nD-L2-in03b}.
The proof is done.
\end{proof}

\begin{proposition}[Refined estimates for $E$: III]\label{prop:refine-E-in003}
We have
\begin{align}
\mathcal{E}^{in}_1 (L^J E, \tau)
\lesssim
\epsilon^2 + (C_1 \epsilon)^3 \tau^{2 \delta},
\qquad
 |J| = N+1.
\end{align}
\end{proposition}

\begin{proof}

We apply Proposition \ref{prop:EE-in2} with $\gamma = 0$ to get
\begin{equation}
\begin{aligned}
&\mathcal{E}^{in}_1 ( L^J E, \tau, 0)
\\
\lesssim
&\mathcal{E}^{in}_1 ( L^J E, \tau_0, 0)
+
\int_{\tau_0}^\tau \big\| (- L^J (n E) + n  L^J E) \cdot (\widetilde{\tau}/t) \partial_t  L^J E \big\|_{L^1(\mathcal{H}_{\widetilde{\tau}})} \, \d \widetilde{\tau}
\\
&+
\int_{\tau_0}^\tau \big\| |L^J E|^2 (\widetilde{\tau}/t) \partial_t n \big\|_{L^1(\mathcal{H}_{\widetilde{\tau}})} \, \d \widetilde{\tau}
\\
&+  \int_{\Sigma^{bd}_{t^{bd}}} \Big( \sum_a ({x_a \over r} \partial_t  L^J E + \partial_a  L^J E)^2 + | L^J E|^2 \Big) \, \d S
\\
=:
& \mathcal{A}_1 + \mathcal{A}_2 + \mathcal{A}_3 + \mathcal{A}_4.
\end{aligned}
\end{equation}

To bound the term $\mathcal{A}_2$, we first note that
\begin{align*}
&\big\| (- L^J (n E) + n  L^J E)\big\|_{L^2(\mathcal{H}_{\widetilde{\tau}})}
\\
\lesssim
&\sum_{|J_1|+|J_2|=N+1, \, |J_1|\geq 1} \| L^{J_1} n L^{J_2} E\|_{L^2(\mathcal{H}_{\widetilde{\tau}})}
\\
\lesssim
 &\sum_{|J_1|+|J_2|=N+1, \, 1\leq |J_1|\leq |J_2|} \| L^{J_1} n\|_{L^\infty(\mathcal{H}_{\widetilde{\tau}})} \| L^{J_2} E\|_{L^2(\mathcal{H}_{\widetilde{\tau}})}
\\
&+
\sum_{|J_1|+|J_2|=N+1, \, |J_1|\geq |J_2|\geq 1} \| (\widetilde{\tau}/t) L^{J_1} n \|_{L^2(\mathcal{H}_{\widetilde{\tau}})} \| (t/\widetilde{\tau})L^{J_2} E\|_{L^\infty(\mathcal{H}_{\widetilde{\tau}})}
\\
&+
 \| (\widetilde{\tau}/t) L^{J} n \|_{L^2(\mathcal{H}_{\widetilde{\tau}})} \| (t/\widetilde{\tau}) E\|_{L^\infty(\mathcal{H}_{\widetilde{\tau}})}
\\
\lesssim
&(C_1 \epsilon)^2 \widetilde{\tau}^{-1+\delta},
\end{align*}
in which we use rough decay of $E$ in \eqref{eq:E-improve-in}, sharp decay of $E$ in Proposition \ref{prop:refine-E-decay} and Lemma \ref{lem:nD-L2-in03}.

Thus, we further have
\begin{align*}
\mathcal{A}_2
\lesssim
&\int_{\tau_0}^{\tau} \big\| (- L^J (n E) + n  L^J E) \big\|_{L^2(\mathcal{H}_{\widetilde{\tau}})}  \|(\widetilde{\tau}/t) \partial_t  L^J E \|_{L^2(\mathcal{H}_{\widetilde{\tau}})}  \, \d\widetilde{\tau}
\\
\lesssim
& (C_1 \epsilon)^3 \int_{\tau_0}^{\tau}  \widetilde{\tau}^{-1+2\delta} \, \d\widetilde{\tau}
\lesssim (C_1 \epsilon)^3 \tau^{2\delta}.
\end{align*}

For the term $\mathcal{A}_3$, we note
\begin{align*}
&\big\| |L^J E|^2 (\widetilde{\tau}/t) \partial_t n \big\|_{L^1(\mathcal{H}_{\widetilde{\tau}})}
\\
\lesssim
&\| L^J E \|^2_{L^2(\mathcal{H}_{\widetilde{\tau}})}  \| (\widetilde{\tau}/t) \partial_t n \|_{L^\infty(\mathcal{H}_{\widetilde{\tau}})}
\\
\lesssim
&(C_1 \epsilon)^3 \widetilde{\tau}^{-{3\over 2}+2\delta}.
\end{align*}
This leads to
\begin{align*}
\mathcal{A}_3
\lesssim
(C_1 \epsilon)^3.
\end{align*}

To conclude, we get
\begin{align*}
&\mathcal{E}^{in}_1 ( L^J E, \tau, 0)
\lesssim
\epsilon^2 + (C_1 \epsilon)^3 \tau^{2\delta}.
\end{align*}

The proof is completed.
\end{proof}

To conclude this section, we gather the estimates for $n, E$ above to have the following refined result of the original bootstrap assumptions in \eqref{eq:BA-in}--\eqref{eq:BA-in3}.

\begin{proposition}
For $\tau \in [\tau_0, \tau_1 )$, the following refined estimates hold:
\begin{align}\label{eq:BA-in-re}
|\partial^I L^J n(t, x)| 
\leq
&{1\over 2} C_1 \epsilon t^{-{1\over 2}},
\qquad
&|I| + |J| \leq N-2,
\\
\mathcal{E}^{in}(\partial^I L^J n, \tau)
\leq
&{1\over 2}(C_1 \epsilon)^{2},
\qquad
&|I| + |J| \leq N, \label{eq:BA-in2-re}
\\
\mathcal{E}^{in}_1(\partial^I L^J E, \tau)
\leq
&{1\over 2}(C_1 \epsilon)^2 \tau^{2\delta},
\qquad
&|I| + |J| \leq N+1. \label{eq:BA-in3-re}
\end{align}
As a consequence, the solution pair $(n, E)$ exists globally in $\{ r\leq t-1\}$.
\end{proposition}
\begin{proof}
The proof follows from Propositions \ref{prop:refine-n-in001}, \ref{prop:refine-E-in001}, \ref{prop:refine-E-in002}, and \ref{prop:refine-E-in003}.
\end{proof}

\section{Scattering of the field $n$}\label{sec:scatter}

In this section and next one, we concentrate on the scattering problem of the 2D Klein-Gordon-Zakharov system.

\subsection{Scattering criteria and energy bounds}

We recall the following scattering criteria for inhomogeneous wave-type equations.

\begin{proposition}[Scattering criteria]\label{prop:scatter01}
Consider the wave-type equation
\begin{align}
-\Box u + m^2 u = F.
\end{align}
If 
\begin{align*}
\int_1^{+\infty} \| F \|_{H^k(\mathbb{R}^2)} \, \d s
< +\infty,
\end{align*}
for some $0\leq k \in \mathbb{N}$, then the solution $(u, \partial_t u)$ scatters in the sense that
\begin{align*}
\| \partial_t (u - u_l)\|_{H^k} + \| \nabla (u- u_l)\|_{H^k} 
+ m \| u-u_l\|_{H^k}
\lesssim
\int_t^{+\infty} \| F \|_{H^k} \, \d s
\to 0,
\quad
\text{as } t\to +\infty.
\end{align*}
In the above, $u_l$ satisfies $-\Box u_l + m^2 u_l = 0$.
\end{proposition}

In Section \ref{sec:interior}, we derived energy bounds of $n, E$ on hyperboloids. However, to apply Proposition \ref{prop:scatter01} to show scattering, we need to establish energy bounds of $n, E$ on constant $t$ slices.

Denote
$$
\Sigma^{in}_t = \{(s, x): s=t,\, r\leq s-1  \}.
$$
For a smooth function $\phi(t, x)$, we define its weighted energy on constant $t$ slices by
\begin{equation}
\begin{aligned}
\mathbb{E}^{in}_m (\phi, t, \gamma)
:=
&\int_{ \Sigma^{in}_t }  \langle r-t\rangle^{-2\gamma} \big( |\partial \phi|^2 + m^2 |\phi|^2 \big) \, \d x
\\
&+
\gamma \int_{t_0}^t \int_{ \Sigma^{in}_s } \langle r-s \rangle^{-2\gamma-1} \big( |G \phi|^2 + m^2 |\phi|^2 \big) \, \d x \d s
\end{aligned}
\end{equation}
as well as
\begin{align}
\mathbb{E}_m (\phi, t, \gamma)
:=
\mathbb{E}^{in}_m (\phi, t, \gamma)
+
\mathcal{E}^{ex}_m (\phi, t, \gamma).
\end{align}

\begin{proposition}[Energy bounds on $t$-slice]\label{prop:bound-t}
We have the following energy bounds on constant $t$ slices.

\begin{enumerate}
\item \emph{Energy bounds for $n$.}

\begin{equation}
    \begin{aligned}
        \| \partial n\|_{H^N}
        \lesssim
        &C_1 \epsilon,
        \\
        \sum_{|I|+|J|\leq N+1} \|\Gamma^I L^J_0 \ell\| 
+ \sum_{|I|+|J|\leq N} \langle t\rangle^\delta\|\partial \Gamma^I L_0^J \ell\|
\lesssim
&C_1 \epsilon \langle t\rangle^\delta,
\\
\sum_{\substack{|I|\geq 2,\, |J|\leq N-2 \\|I|+|J|\leq N+2}} \big( \big\|\langle t-r\rangle^{1\over 2} \partial^I \Gamma^J \mathfrak{m}\big\| 
+\big\|\langle t-r\rangle^{3\over 2} \partial\partial^I \Gamma^J \mathfrak{m}\big\|
\big)
\lesssim
&C_1 \epsilon.
    \end{aligned}
\end{equation}

\item \emph{Energy bounds for $E$.}

\begin{equation}
\begin{aligned}
\|\partial_t E\|_{H^{N+1}} + \|E\|_{H^{N+2}}
        \lesssim
        &C_1 \epsilon,
\\
\sum_{|I|\leq N+1} \big(\|\Gamma^I E\| + \|\partial \Gamma^I E\| \big) 
\lesssim 
&C_1 \epsilon \langle t\rangle^\delta,
\\
\big\|\langle t+r\rangle^{2} \langle t-r\rangle^{-2} E \big\|_{H^{7}}
\lesssim 
&C_1 \epsilon.
\end{aligned}
\end{equation}

\item \emph{Energy bounds for $nE$.}

\begin{equation}
\begin{aligned}    
\sum_{|I|\leq 7} \|\Gamma^I (n E)\| 
\lesssim 
&(C_1 \epsilon)^2 \langle t\rangle^{-1+\delta}.
\end{aligned}
\end{equation}

\item \emph{Pointwise bounds.}
\begin{align*}
\sum_{|I|+|J|\leq N-6}|\Gamma^I L_0^J \ell|
\lesssim 
&C_1 \epsilon \min \{\langle t+r\rangle^{-{1\over 2}}, \langle t+r\rangle^{-{1\over 2}} \langle t-r\rangle^{-{1\over 2}} \log \langle t\rangle \},
\\
\sum_{|I|\leq N-6} |\partial\partial\Gamma^I \mathfrak{m}| 
\lesssim
&C_1 \epsilon \langle t+r\rangle^{-{1\over 2}} \langle t-r\rangle^{-1},
\\
\sum_{|I|\leq N-8} |\partial^I E|
\lesssim
&C_1 \epsilon \langle t+r\rangle^{-{5\over 2}} \langle t-r\rangle^{{3\over 2}}.
\end{align*}
\end{enumerate}
\end{proposition}
\begin{proof}
    Note that the pointwise estimates are already derived in Sections \ref{sec:exterior} and \ref{sec:interior}. Since the energy bounds for $n, E$ in the exterior region are derived on constant $t$-slices in Section \ref{sec:exterior}, we only focus on their energy bounds on constant $t$-slices in the interior region.

\textbf{Case I: $n$ part.}

Recall the decomposition $n=\ell + \Delta \mathfrak{m}$, and the bounds for $\ell$ part follow from Proposition \ref{prop:con-EE} and Lemma \ref{lem:freewave}. Therefore, we only focus on the $\mathfrak{m}$ part.

Given $t>10$ fixed, we apply the divergence theorem on the equation \eqref{eq:div1} with $\phi = \partial^I \Gamma^J \mathfrak{m}, m=0, F=\partial^I \Gamma^J |E|^2, \gamma={1\over 2}$ with $|I|+|J|\leq N+2, \, |J|\leq N+1$
over the region 
\begin{align*}
    \mathcal{D}_A = \mathcal{D}_{[\tau_0, t]} \bigcap \{(s, x): s \leq t \}
\end{align*}
to get
\begin{align*}
    \mathbb{E}^{in}_0(\partial^I \Gamma^J \mathfrak{m}, t, {1\over 2})
    \lesssim
    &\mathcal{E}^{in}_0(\partial^I \Gamma^J \mathfrak{m}, \tau_0, {1\over 2})
    +
    \int_{\mathcal{D}_A} \big| \tau_-^{-1} \partial^I \Gamma^J|E|^2 \cdot \partial_t \partial^I \Gamma^J \mathfrak{m} \big| \, \d x \d s
    \\
    &+
    \int_{\Sigma^{bd}_t}  \sum_a ({x_a \over t} \partial_t  \partial^I L^J \mathfrak{m} + \partial_a  \partial^I L^J \mathfrak{m})^2  \, \d S
    \\
    \lesssim
    &\epsilon^2 
    +
    \int_{\mathcal{D}_{[\tau_0, t]}} \big| \tau_-^{-1} \partial^I \Gamma^J |E|^2 \partial_t \partial^I \Gamma^J \mathfrak{m} \big| \, \d x \d s
    \\
    \lesssim
    &\epsilon^2 
    +
    \int_{\tau_0}^t \big\| \tau_-^{-1} \partial^I \Gamma^J |E|^2 \cdot (\widetilde{\tau}/t) \partial_t \partial^I \Gamma^J \mathfrak{m} \big\|_{L^1(\mathcal{H}_{\widetilde{\tau}})} \, \d\widetilde{\tau},
\end{align*}
in which we used the bounds from Proposition \ref{prop:BDRY-4} to estimate the boundary contributions.

Therefore, employing the estimates in Lemma \ref{lem:nDelta-in}, \eqref{eq:E-improve-in}, and \eqref{eq:BA-in3-re}, we derive that
\begin{align*}
    \mathbb{E}^{in}_0(\partial^I \Gamma^J \mathfrak{m}, t, {1\over 2})
    \lesssim
    \epsilon^2 + (C_1 \epsilon)^3,
    \qquad
    |I|+|J|\leq N+2, \, |J|\leq N+1.
\end{align*}

Applying Proposition \ref{prop:wave-extra} (twice) yields the bounds for $\partial \partial \Gamma^I \mathfrak{m}$ (and $\partial \partial \partial \Gamma^I \mathfrak{m}$).

\textbf{Case II: $E$ part.}

We first look at the estimates for $\partial^I E$ with $|I|\leq N+1$. For any fixed $t>10$, we utilize the divergence identity \eqref{eq:div2} with $\phi = \partial^I E, m=1, h= n, F= -\partial^I (nE), \gamma = 0$ and the divergence theorem in the spacetime region $\mathcal{D}_A$ to obtain 
    \begin{align*}
    &\mathbb{E}^{in}_1(\partial^I E, t, 0)
    \\
    \lesssim
    &\mathcal{E}^{in}_1(\partial^I E, \tau_0, 0)
    +
    \int_{\mathcal{D}_A} \big( \big|(-\partial^I (nE)+n\partial^I E)\cdot \partial_t \partial^I E \big| + |\partial^I E|^2 |\partial_t n| \big) \, \d x \d s
    \\
    &+
 \int_{\Sigma^{bd}_t}  \sum_a ({x_a \over t} \partial_t \partial^I E + \partial_a \partial^I E)^2 + |\partial^I E|^2  \, \d S
    \\
    \lesssim
    &\epsilon^2 
    +
    \int_{\tau_0}^t \big\| (\widetilde{\tau}/t)\big( |(-\partial^I (nE)+n\partial^I E)\cdot \partial_t \partial^I E| + |\partial^I E|^2 |\partial_t n| \big)\big\|_{L^1(\mathcal{H}_{\widetilde{\tau}})} \, \d\widetilde{\tau},
\end{align*}
in which we used Proposition \ref{prop:BDRY-4}.

Using the estimates in Lemmas \ref{lem:n0-in}--\ref{lem:nDelta-in}, \eqref{eq:E-improve-in}, and \eqref{eq:BA-in3-re}, we get
\begin{align*}
    \mathbb{E}^{in}_1(\partial^I E, t, 0)
    \lesssim
    \epsilon^2 + (C_1 \epsilon)^3,
    \qquad
    |I|\leq N+1.
\end{align*}

Second, in a similar manner, we deduce that
\begin{align*}
    \mathbb{E}^{in}_1(\Gamma^I E, t, 0)
    \lesssim
    \epsilon^2 + (C_1 \epsilon)^3 \langle t\rangle^{2\delta},
    \qquad
    |I|\leq N+1.
\end{align*}

Applying Proposition \ref{prop:KG-extra} twice leads us to the bounds for $\langle t+r\rangle^2 \langle t-r\rangle^{-2} E$.

\textbf{Case III: $nE$ part.}

For $|I|\leq 7$, we simply have
\begin{align*}
    &\|\Gamma^I (nE)\|
    \\
    \lesssim
    &\sum_{\substack{|I_1|+|I_2|\leq |I| \\ |I_1|\leq |I_2|}} \big(\| (\tau_-/\tau_+)\Gamma^{I_1}n\|_{L^\infty} \| (\tau_+/\tau_-) \Gamma^{I_2}E\|
    + \| \Gamma^{I_2}n\| \|\Gamma^{I_1}E\|_{L^\infty}\big)
    \\
    \lesssim
    &(C_1 \epsilon)^2 \langle t\rangle^{-1+\delta}.
\end{align*}

    The proof is done.
\end{proof}

\subsection{Scattering for $n$}

We denote $Q_{\alpha \beta}(f, g) = \partial_\alpha f \cdot\partial_\beta g - \partial_\beta f \cdot\partial_\alpha g$, which represents null forms that are compatible with both wave and Klein-Gordon fields, and enjoys
$$
Q_{\alpha \beta}(f, g)
\lesssim
\langle t+r\rangle^{-1} \big( \big|\Gamma f\big| |\partial g| + |\partial f| \big|\Gamma g\big|  \big).
$$
Recall that $\Gamma \in \{  \partial_0, \partial_1, \partial_2, \Omega_{12}, L_1, L_2  \}$.

\begin{lemma}\label{lem:transform01}
Let 
\begin{equation}\label{eq:nono-trans}
\widetilde{n} = n + {1\over 4}\Delta |E|^2,
\end{equation}
then we have
\begin{equation}
\begin{aligned}
-\Box \widetilde{n}
=
&{1\over 2} \big(-\Box \Delta E + \Delta E\big)\cdot E
+
{3\over 2} \Delta E \cdot \big(-\Box E + E \big)
+
{3\over 2} \big(-\Box \partial^a E + \partial^a E \big)\cdot \partial_a E
\\
&+
{1\over 2} \partial^a E \cdot \big(-\Box \partial_a E + \partial_a E \big)
-
Q_{a\beta} (\partial^\beta \partial^a E, E)
-
Q_{\beta a} (\partial^a E, \partial^\beta E).
\end{aligned}
\end{equation}
\end{lemma}

\begin{proof}
Set $f^a = 2 \partial^a E$, $g=E$, and then
\begin{align*}
\Delta |E|^2
=
\partial_a (f^a \cdot g).
\end{align*} 
We find
\begin{align*}
-\Box \widetilde{n}
=
&\partial_a (f^a \cdot g)
-
{1\over 4} \Box \partial_a (f^a \cdot g)
\\
=
&\partial_a f^a \cdot g + f^a \cdot \partial_a g
-
{1\over 4} \big(\partial_a \Box f^a \cdot g + \partial_a f^a \cdot \Box g +  \Box f^a \cdot  \partial_a g +  f^a \cdot \Box \partial_a g
\\
&+
2 \partial_\beta \partial_a f^a \cdot \partial^\beta g + 2 \partial_\beta  f^a \cdot  \partial^\beta \partial_a g
\big)
\\
=
& \partial_a f^a \cdot g + f^a \cdot \partial_a g
+
{1\over 4} \big(-\Box \partial_a f^a + \partial_a f^a -\partial_a f^a \big)\cdot  g
+
{1\over 4} \partial_a f^a \cdot \big(-\Box g + g -g \big)
\\
&+
{1\over 4} \big(-\Box  f^a +  f^a - f^a \big)\cdot \partial_a g
+
{1\over 4}  f^a \cdot \big(-\Box \partial_a g + \partial_a g - \partial_a g \big)
\\
&-
{1\over 4} \big( 2\partial_a \partial^\beta f^a \cdot \partial_\beta g - 2  \partial_\beta \partial^\beta f^a \cdot \partial_a g
-2  (-\Box  f^a +  f^a - f^a )\cdot \partial_a g
\\
&+
2 \partial_\beta f^a \cdot \partial_a \partial^\beta g - 2 \partial_a f^a \cdot \partial_\beta \partial^\beta g
-
2 \partial_a f^a \cdot (-\Box g + g -g)
\big).
\end{align*}
After some cancellation, we arrive at
\begin{align*}
-\Box \widetilde{n}
=
& 
{1\over 4} \big(-\Box \partial_a f^a + \partial_a f^a  \big) \cdot g
+
{3\over 4} \partial_a f^a \cdot \big(-\Box g + g \big)
\\
&+
{3\over 4} \big(-\Box  f^a +  f^a  \big) \cdot \partial_a g
+
{1\over 4}  f^a \cdot \big(-\Box \partial_a g + \partial_a g \big)
\\
&-
{1\over 4} \big( 2\partial_a \partial^\beta f^a \cdot \partial_\beta g - 2  \partial_\beta \partial^\beta f^a \cdot \partial_a g
+
2 \partial_\beta f^a \cdot \partial_a \partial^\beta g - 2 \partial_a f^a \cdot \partial_\beta \partial^\beta g
\big).
\end{align*}

Finally, inserting $f^a = \partial^a E, g=E$ leads to the desired result.
\end{proof}

More generally, as long as the nonlinearities in the wave equation are of divergence form, we have an analogous result.

\begin{lemma}\label{lem:transform02}
Let $w$ solve
$$
-\Box w = \partial_\gamma (f g),
$$
then $\widetilde{w} : = w + {1\over 4} \partial_\gamma (f g)$ satisfies
\begin{equation}
\begin{aligned}
-\Box \widetilde{w}
=
&{1\over 4} \big( -\Box \partial_\gamma f + \partial_\gamma f  \big) g
+
{3\over 4} \partial_\gamma f \big( -\Box g + g  \big)
+
{3\over 4} \big( -\Box f + f  \big) \partial_\gamma g
\\
&+
{1\over 4} f \big( -\Box \partial_\gamma g + \partial_\gamma g  \big) 
-
{1\over 2} Q_{\gamma \alpha} (\partial^\alpha f, g)
-
{1\over 2} Q_{ \alpha \gamma} ( f, \partial^\alpha g).
\end{aligned}
\end{equation}
\end{lemma}
\begin{proof}
We note
\begin{align*}
-\Box \widetilde{w}
=
&-\Box \big(w + {1\over 4} \partial_\gamma (fg) \big)
\\
=
&-\Box w 
+ {1\over 4} (-\Box \partial_\gamma f)g
+{1\over 4} \partial_\gamma f (-\Box g)
- {1\over 2} \partial_\alpha \partial_\gamma f \partial^\alpha g
\\
&+{1\over 4 } f (-\Box \partial_\gamma g)
+ {1\over 4} (-\Box f) \partial_\gamma g
-{1\over 2} \partial_\alpha f \partial^\alpha \partial_\gamma g
\\
=
&
{1\over 4} \big( -\Box \partial_\gamma f + \partial_\gamma f  \big) g
+
{3\over 4} \partial_\gamma f \big( -\Box g + g  \big)
+
{3\over 4} \big( -\Box f + f  \big) \partial_\gamma g
\\
&+
{1\over 4} f \big( -\Box \partial_\gamma g + \partial_\gamma g  \big) 
-
{1\over 2} Q_{\gamma \alpha} (\partial^\alpha f, g)
-
{1\over 2} Q_{ \alpha \gamma} ( f, \partial^\alpha g).
\end{align*}

The proof is completed.
\end{proof}

We go back to proving linear scattering of the component $n$ in \eqref{eq:2D-KGZ}.

\begin{proposition}[Linear scattering of $n$]\label{prop:scatter-n}
There exist a pair of  initial data $(n^f_0, n^f_1) \in H^{N-1} \times H^{N-2}$ and a  function $n^f$ solving the linear wave equation
\begin{align*}
&-\Box n^f = 0,
\\
&(n^f, \partial_t n^f)(t_0) = (n^f_0, n^f_1),
\end{align*}
 such that
\begin{align}
\| \partial_t (n - n^f) \|_{H^{N-2}}
+
\| \nabla (n- n^f) \|_{H^{N-2}}
\lesssim
t^{-{1\over 2}}.
\end{align}
\end{proposition}

\begin{remark}
In this scattering result, we lose two derivatives of regularity, which is due to the derivative loss in the nonlinear transformation \eqref{eq:nono-trans}.
\end{remark}





\begin{proof}[Proof of Proposition \ref{prop:scatter-n}]

\

We set 
$$
\widetilde{n} = n + {1\over 4}\Delta |E|^2.
$$
We employ Lemma \ref{lem:transform01} to get 
\begin{equation}
\aligned
-\Box \widetilde{n}
&=
-{1\over 2}  \Delta (nE)  E
-
{3\over 2} \Delta E  (n E)
-
{3\over 2}  \partial^a (n E) \partial_a E
\\
&-
{1\over 2} \partial^a E  \partial_a (nE)
-
Q_{a\beta} (\partial^\beta \partial^a E, E)
-
Q_{\beta a} (\partial^a E, \partial^\beta E).
\endaligned
\end{equation}
Our strategy is to first demonstrate linear scattering of $\widetilde{n}$, and then to transfer this property to $n$ by using the fact that the difference between $n$ and $\widetilde{n}$ is a lower order term of quadratic type. The usage of this nonlinear transformation is the main reason why we lose some regularity and cannot cover the full range scattering in $(\partial_t n, \nabla n) \in H^{N}\times H^N$.

To show linear scattering of $\widetilde{n}$, we rely on Proposition \ref{prop:scatter01} and need to bound the source terms of the wave equation of $\widetilde{n}$.

We have
\begin{align*}
\| \Delta (nE) \cdot E \|_{H^{N-2}}
\lesssim
&\sum_{|I|\leq N-2}\| \partial^I (\Delta (nE) \cdot E) \|
\\
\lesssim &(C_1 \epsilon)^3 \langle t\rangle^{-{3\over 2}}.
\end{align*}
Similarly, we derive
\begin{align*}
&\| \Delta E \cdot (nE) \|_{H^{N-2}}
+
\| \partial^a (n E) \partial_a E \|_{H^{N-2}}
+
\| \partial^a E  \partial_a (nE) \|_{H^{N-2}}
\\
\lesssim 
&(C_1 \epsilon)^3 \langle t\rangle^{-{3\over 2}}.
\end{align*}
For the null form, we have
\begin{align*}
\| Q_{a\beta} (\partial^\beta \partial^a E, E) \|_{H^{N-2}}
\lesssim
&\sum_{|I|\leq N-2}\| \partial^I Q_{a\beta} (\partial^\beta \partial^a E, E) \|
\\
\lesssim &(C_1 \epsilon)^3 \langle t\rangle^{-{3\over 2}}.
\end{align*}
Similarly, we get
\begin{align*}
\| Q_{\beta a} (\partial^a E, \partial^\beta E) \|_{H^{N-2}}
\lesssim (C_1 \epsilon)^3 \langle t\rangle^{-{3\over 2}}.
\end{align*}

On the other hand, we can bound the difference term by

\begin{align*}
    \|\Delta |E|^2 \|_{H^{N-1}}
    \lesssim
    (C_1 \epsilon)^2 \langle t\rangle^{-1}.
\end{align*}

To conclude, we have
\begin{equation}
\begin{aligned}
    &\| \partial_t (n - n^f) \|_{H^{N-2}}
    +
    \| \nabla (n- n^f) \|_{H^{N-2}}
    \\
    \leq
    &\| \partial_t (\widetilde{n} - n^f) \|_{H^{N-2}}
    +
    \| \nabla (\widetilde{n}- n^f) \|_{H^{N-2}}
    +
    \|\partial\partial |E|^2\|_{H^{N-2}}
    \\
    \lesssim
    &\int_{t}^{+\infty} \|\Box \widetilde{n}\|_{H^{N-1}} \, \d s
    + 
    \|\partial\partial |E|^2\|_{H^{N-1}}
    \lesssim \langle t\rangle^{-{1\over 2}}.
\end{aligned}
\end{equation}

The proof is completed.
\end{proof}

\section{Scattering of the field $E$}\label{sec:S-E}


\subsection{Setting in Fourier space}

Denote
\begin{align*}
E_\pm = &(\partial_t \mp i \langle \nabla \rangle) E,
\\
\ell_\pm = &(\partial_t \mp i  |\nabla |) \ell,
\\
(nE)_{\pm} = &(\partial_t \mp i \langle \nabla \rangle)(nE)
\end{align*}
with the interpretation in the Fourier space as
\begin{align*}
\widehat{E_\pm} = &(\partial_t \mp i \langle \xi \rangle) \widehat{E},
\\
\widehat{\ell_\pm} = &(\partial_t \mp i | \xi| ) \widehat{\ell},
\\
\widehat{(nE)_\pm} = &(\partial_t \mp i \langle \xi \rangle) \widehat{nE}.
\end{align*}

We denote the profiles of $E_\pm$, $\ell_\pm$, and $(nE)_\pm$ by 
\begin{align*}
f_\pm = e^{\pm it \langle \nabla\rangle} E_\pm,
\qquad
g_\pm = e^{\pm it | \nabla|} \ell_\pm,
\end{align*}
and 
\begin{align*}
h_\pm = e^{\pm it \langle \nabla\rangle} (nE)_\pm.
\end{align*}
respectively.
Easily, one finds
\begin{align*}
E = {E_- - E_+ \over 2i \langle \nabla \rangle}
=i{e^{-it\langle\nabla\rangle} f_+ - e^{it\langle\nabla\rangle} f_- \over 2 \langle\nabla \rangle},
\\
\ell = {\ell_- - \ell_+ \over 2i |\nabla|}
= i{e^{-it|\nabla|} g_+ - e^{it|\nabla|} g_- \over 2 |\nabla|},
\\
nE = {(nE)_- - (nE)_+ \over 2i \langle \nabla \rangle}
=i{e^{-it\langle\nabla\rangle} h_+ - e^{it\langle\nabla\rangle} h_- \over 2 \langle\nabla \rangle}.
\end{align*}

Taking into account the equations of $E, \ell$, and $nE$, we find that
\begin{align*}
\partial_t f_\pm
=
&e^{\pm it\langle \nabla\rangle} (-\ell E -\Delta \mathfrak{m} E)
\\
=
&-{i\over 2} e^{\pm it\langle \nabla\rangle} \Big( \ell {e^{-it\langle \nabla\rangle} f_+ - e^{it \langle \nabla\rangle} f_- \over \langle \nabla\rangle} \Big) - e^{\pm it\langle \nabla\rangle} (\Delta \mathfrak{m} E),
\end{align*}
as well as
\begin{align*}
\partial_t g_\pm
= 0,
\end{align*}
and 
\begin{align*}
    \partial_t h_\pm
    =
    e^{\pm it\langle \nabla\rangle} \big( \Delta|E|^2 \cdot E - n^2E- 2\partial_\alpha n \partial^\alpha E \big),
\end{align*}
in which we used the fact that
\begin{align*}
    -\Box (nE) + nE
    =
    \Delta|E|^2 \cdot E - n^2E- 2\partial_\alpha n \partial^\alpha E. 
\end{align*}

The proof is based on the framework of the spacetime resonance method of Germain-Masmoudi-Shatah and Gustafson-Nakanishi-Tsai (see, for instance, \cite{GeMaSh, GuNaTs} etc.) and the Z-norm method of Ionescu-Pausader (see for instance \cite{Ionescu-P2} etc.). While these two approaches share many common features, they also differ in important aspects. For instance, the spacetime resonance method focuses on analyzing the geometry and size of resonance sets, whereas the Z-norm method is built around sharp estimates for oscillatory phase functions.

We define the Z-norm as
\begin{align}
\|E\|_Z =  \sup_{t\geq t_0} 2^{-C_E k^-} 2^{-D_E k^+} \| {P_k E_{+}} \|,
\end{align}
in which $C_E>0, D_E<0$, $k^- = \min \{ k, 0 \}$ and $k^+ = \max \{ k, 0\}$. We will take $C_E = \delta$ and $D_E = -\delta$, but we want to track them in the analysis so we use the notation $C_E, D_E$ instead.  One notices that $\|{P_k E_{+}} \| = \|{P_k E_{-}} \|$.

The rough idea is that
\begin{equation}
\begin{aligned}
\partial_t \widehat{f_+}
&=
-{i\over 2} e^{it\langle\xi\rangle} \int_{\mathbb{R}^2} \widehat{n}(t, \eta) \Big( {e^{-it\langle \xi-\eta\rangle} \widehat{f_+}(t, \xi-\eta) \over \langle \xi - \eta\rangle} - {e^{it \langle \xi-\eta\rangle} \widehat{f_-}(t, \xi-\eta) \over \langle \xi-\eta\rangle}  \Big)
\\
&= 
-{i\over 2} \int \widehat{\ell_L}(t, \eta) {e^{it{\xi\over \langle \xi\rangle} \cdot \eta} \over \langle \xi \rangle} \, \d\eta \, \widehat{f_+}(t, \xi)
+ \mathcal{M}_1 + \mathcal{M}_2 + \mathcal{M}_3 +\mathcal{M}_4.
\end{aligned}
\end{equation}
In the above, we use the definition for a regular function $u(t, x)$ that
\begin{align}
\widehat{u_L}(t, \eta) = \phi_{\leq 0}(\eta \langle t\rangle^p) \widehat{u}(t, \eta), 
\qquad
p = {3\over 4},
\end{align}
while
\begin{align}
u_H = u - u_L,
\quad
\text{ and }
\quad
\widehat{u_H}(t, \eta) = \phi_{\geq 1} (\eta \langle t\rangle^p) \widehat{u}(t, \eta), 
\end{align}
and we use the expressions
\begin{equation}
\begin{aligned}
\mathcal{M}_1
=
&{i\over 2} e^{it\langle\xi\rangle} \int \widehat{\ell}(t, \eta) {e^{it \langle \xi-\eta\rangle} \widehat{f_-}(t, \xi-\eta) \over \langle \xi-\eta\rangle} ,
\\
\mathcal{M}_2
=
&-{i\over 2} e^{it\langle\xi\rangle} \int \widehat{\ell_H}(t, \eta) {e^{-it\langle \xi-\eta\rangle} \widehat{f_+}(t, \xi-\eta) \over \langle \xi - \eta\rangle},
\\
\mathcal{M}_3
=
&-{i\over 2}  \int \widehat{\ell_L}(t, \eta) \Big( {e^{it(\langle\xi\rangle-\langle \xi-\eta\rangle)} \widehat{f_+}(t, \xi-\eta) \over \langle \xi - \eta\rangle} - {e^{it{\xi\over \langle \xi\rangle} \cdot \eta} \widehat{f_+}(t, \xi) \over \langle \xi \rangle}  \Big),
\\
\mathcal{M}_4
=
& - e^{it\langle\xi\rangle} \widehat{\Delta \mathfrak{m} E}.
\end{aligned}
\end{equation}

Denote
\begin{align}
\widehat{f_*} = e^{i \Theta(t, \xi)} \widehat{f_+},
\end{align}
with the real-valued phase correction
\begin{equation}\label{eq:ph-correction}
\begin{aligned}
\Theta(t, \xi) 
= 
&{1\over 2} \langle \xi\rangle^{-1} \int_{t_0}^t \int \widehat{\ell_L}(s, \eta) e^{is {\xi\over \langle \xi\rangle} \cdot \eta } \, \d\eta ds
\\
=&{2\pi^2} \langle \xi\rangle^{-1} \int_{t_0}^t \ell_L(s, {s\xi\over \langle \xi\rangle}) ds.
\end{aligned}
\end{equation}
This leads us to
\begin{align}
\partial_t \widehat{f_*}
=
e^{i \Theta(t, \xi)} \big( \mathcal{M}_1 + \mathcal{M}_2 + \mathcal{M}_3 + \mathcal{M}_4\big).
\end{align}

Our bootstrap assumptions are as follows
\begin{align}
\| E \|_Z \leq C_1 \epsilon.
\end{align}
Our goal is to show
\begin{align}
\big\| \phi_k(\xi) (\widehat{f_*}(t_2, \xi) - \widehat{f_*}(t_1, \xi)) \big\|
\lesssim (C_1 \epsilon)^2 2^{C_E k^-} 2^{D_E k^+} 2^{-\delta m},
\end{align}
for all $k \in \mathbb{Z}$ and $t_1, t_2 \in [2^{m}-2, 2^{m+1}+2]$ with $m\geq 10$, which further implies nonlinear (modified) scattering of $E$.
Once this is established, we denote
\begin{align}\label{eq:f-infty}
    \widehat{f_*}(\infty, \xi)
    =
    \widehat{f_*}(t_0, \xi) 
    + \int_{t_0}^{+\infty} e^{i \Theta(s, \xi)} \big( \mathcal{M}_1 + \mathcal{M}_2 + \mathcal{M}_3 + \mathcal{M}_4\big) \, \d s.
\end{align}

\subsection{Nonlinear scattering for $E$}

\subsubsection{Infinite estimates}

For convenience, for a sysmbol $m(\xi, \eta)$, we denote
\begin{align}
    T_{m}(u, v)
    =
    \mathcal{F}_\xi^{-1}\int_{\mathbb{R}^2} m(\xi, \eta)\,  \widehat{u}(\eta) \widehat{v}(\xi-\eta) \, \d\eta.
\end{align}

\begin{lemma}\label{lem:Theta}
We have
\begin{align*}
|\partial_t \Theta(t, \xi)|
&\lesssim
C_1 \epsilon t^{-p} \langle\xi\rangle^{-1},
\\
|\partial_t \partial_t \Theta(t, \xi)|
&\lesssim
C_1 \epsilon t^{-2p}\langle\xi\rangle^{-1}.
\end{align*}
\end{lemma}
\begin{proof}
    First, we compute
\begin{align*}
    \partial_t \Theta(t, \xi)
    =
    {1\over 2} \langle \xi\rangle^{-1} \int e^{it {\xi \over \langle \xi\rangle} \cdot \eta } \widehat{\ell_L}(\eta) \, \d\eta.
\end{align*}
Recall the bound
$$
|\widehat{\ell}(\eta)|
\lesssim
\epsilon |\eta|^{-1},
\qquad
\text{ for } |\eta|\leq 1,
$$
and we have
\begin{align*}
    |\partial_t \Theta|
    \lesssim
    \epsilon \langle\xi\rangle^{-1} \int_{|\eta|\leq 4\langle t\rangle^{-p}} |\eta|^{-1} \, \d\eta
    \lesssim
    \epsilon \langle \xi\rangle^{-1} \langle t\rangle^{-p}.
\end{align*}

    Second, we compute
    \begin{align*}
    \partial_t \partial_t \Theta(t, \xi)
    =
    &{1\over 2} \langle \xi\rangle^{-1} \int e^{it{\xi\over \langle \xi\rangle} \cdot \eta} 
    \big( i{\xi\over \langle \xi\rangle} \cdot \eta \, \widehat{\ell_L}(\eta)
    + \phi_{\leq 0}(\eta \langle t\rangle^p) \partial_t \widehat{\ell}(\eta)
    \\
    &\qquad\qquad+ p \nabla\phi_{\leq 0}(\eta \langle t\rangle^p)\cdot \eta \, t \langle t\rangle^{p-2} \widehat{\ell}(\eta)
     \big) \, \d\eta,
    \end{align*}
which leads us to
\begin{align*}
    |\partial_t \partial_t \Theta|
    \lesssim
    \epsilon \langle\xi\rangle^{-1} \int_{|\eta|\leq 4\langle t\rangle^{-p}} 1 \, \d\eta
    \lesssim
    \epsilon \langle \xi\rangle^{-1} \langle t\rangle^{-2p}.
\end{align*}

The proof is done.
\end{proof}

\begin{lemma}\label{lem:d-xi-L2}
    We have
    \begin{align}
        \big\|\phi_k(\xi) \partial_{\xi_a} \widehat{f_{+}}(\xi)\big\|_{H^7}
        \lesssim 
        &C_1 \epsilon 2^{-k^+} t^\delta,
        \\
        \big\|\phi_k(\xi) \partial_{\xi_a} \widehat{g_{+}}(\xi)\big\|_{H^7}
        \lesssim
        &C_1 \epsilon 2^{ k},
        \\
        \big\|\phi_k(\xi) \partial_{\xi_a} \widehat{h_{+}}(\xi)\big\|_{H^7}
        \lesssim 
        &C_1 \epsilon 2^{-k^+} t^{-1+2\delta}.
    \end{align}
\end{lemma}
\begin{proof}

Below we only consider the $\|\cdot \|_{L^2}$ bounds, and the $\|\cdot\|_{H^7}$ bounds can be shown in a similar way.

\textbf{Case I}: $\widehat{f_+}$.

By Lemma \ref{lem:d-xi}, we have
\begin{align*}
    \big|\langle \xi\rangle \partial_{\xi_a} \widehat{f_+}(\xi)\big|
\lesssim
\big|\widehat{L_a E_+}\big| + \big|\widehat{f_+}\big|
+
\big|\partial_{\xi_a} \widehat{nE} \big|.
\end{align*}
Note that
\begin{align*}
    \big\|  \widehat{P_k f_+}\big\| 
    \lesssim 
    C_1 \epsilon,
    \\
    \big\|\phi_k(\xi) \partial_{\xi_a} \widehat{nE} \big\|
    \lesssim
    \big\||x| \, nE \big\|
    \lesssim
    \langle t\rangle^\delta.
\end{align*}
By Lemma \ref{lem:comm-Fourier}, we find
\begin{align*}
    \big\| \widehat{L_a E_+}\big\|
    \lesssim
    \|\partial E\| + \|(\partial_t - i\langle \Lambda\rangle)L_a E\|
    \lesssim
    C_1 \epsilon \langle t\rangle^\delta.
\end{align*}
Gathering these estimates, we arrive at
\begin{align}
    \big\|\phi_k(\xi) \partial_{\xi_a} \widehat{f_{+}}(\xi)\big\|
    \lesssim 
    C_1 \epsilon 2^{-k^+} t^\delta.
\end{align}

\textbf{Case II}: $\widehat{g_+}$.

Recall that $\widehat{g_+}(\xi) = \widehat{n_1}(\xi) - i |\xi| \widehat{n_0}(\xi)$. Therefore
\begin{align*}
    &\big\|\phi_k(\xi) \partial_{\xi_a} \widehat{g_{+}}(\xi)\big\|
    \\
    =
    &\big\|\phi_k(\xi) \partial_{\xi_a} \widehat{n_1}(\xi)\big\|
    +
    \big\|\phi_k(\xi) \partial_{\xi_a} (|\xi| \widehat{n_0}(\xi))\big\|
    \\
    \lesssim
    & 2^k \big( \| x n_1\|_{L^1} + \|n_0\|_{L^1} + \|x n_0\| \big).
\end{align*}

\textbf{Case III}: $\widehat{h_+}$.  

The proof in this case is similar to that of \textbf{Case I}.

Employing the estimates in Proposition \ref{prop:bound-t}, we get
\begin{align*}
    \| h_+\|
    \lesssim
    \|\partial_t (nE)\| + \|nE\|_{H^1}
    \lesssim
    (C_1 \epsilon)^2 t^{-1+\delta}.
\end{align*}
Now, we look at
\begin{align*}
    &\|x (\partial_\alpha n \partial^\alpha E) \|
    \\
    \lesssim
    &\|x \partial \ell \partial E\| + \sum_{|J|=3}\| x \partial^J \mathfrak{m} \partial E\|
    \\
    \lesssim
     &\big\| {\langle t-r\rangle } \partial \ell \big\| \cdot \big\|x{\langle t-r\rangle^{-1} }\partial E \big\|_{L^\infty} + \sum_{|J|=3} \big\| {\langle t-r\rangle }\partial^J \mathfrak{m} \big\| \cdot \big\| x {\langle t-r\rangle^{-1} } \partial E \big\|_{L^\infty}
     \\
     \lesssim
     &(C_1 \epsilon)^2 \langle t\rangle^\delta.
\end{align*}
Similarly, we derive 
\begin{align*}
    \| x (\Delta|E|^2 \cdot E - n^2 E)\|
    \lesssim
    (C_1 \epsilon)^2 t^{-1+\delta}.
\end{align*}

Then, we have
\begin{equation}
\begin{aligned}
    &2^{k^+}\big\|\phi_k(\xi) \partial_{\xi_a} \widehat{h_{+}}(\xi)\big\|
    \\
    \lesssim
    &\|\partial(nE)\| + \| (\partial_t - i\langle \Lambda\rangle) L_a(nE)\|
    +
    \big\|\partial_{\xi_a} \mathcal{F}\big(\Delta|E|^2 \cdot E - n^2 E - 2\partial_\alpha n \partial^\alpha E \big) \big\|
    \\
    \lesssim &(C_1 \epsilon)^2 \langle t\rangle^{-1+2\delta}.
\end{aligned}
\end{equation}

The proof is completed.
\end{proof}


\begin{proposition}\label{prop:M1}
For all $m\gg 1$, $t_1, t_2 \in [2^m - 2, 2^{m+1}+2]$, and $k\leq \delta m$, it holds
\begin{align}\label{eq:M1}
\Big\| \phi_k(\xi) \int_{t_1}^{t_2} e^{i\Theta(s, \xi)} \mathcal{M}_1(s, \xi) \, \d s \Big\|
\lesssim
\epsilon^2 2^{-\delta m} 2^{C_E k^-} 2^{D_E k^+}.  
\end{align}

\end{proposition}

\begin{proof}
We note 
\begin{align*}
 \int_{t_1}^{t_2} e^{i\Theta(s, \xi)}\mathcal{M}_1(s, \xi) \, \d s 
= \mathfrak{M}_{1+} + \mathfrak{M}_{1-},
\end{align*}
in which
\begin{align*}
\mathfrak{M}_{1\pm}
=
{\mp 1\over 4} \int_{t_1}^{t_2} e^{i\Theta(s, \xi)}   \int  e^{is\Phi_{1\mp}} {\widehat{g_{\pm}}(\eta) \over |\eta|} { \widehat{f_-}(s, \xi-\eta) \over \langle \xi-\eta\rangle} \, \d\eta \d s,
\end{align*}
in which 
$$
\Phi_{1\pm} = \langle \xi\rangle + \langle \xi-\eta\rangle \pm |\eta|.
$$
Note that $\Phi_{1\pm} > 0$, and we integrate by part in time to get
\begin{align*}
\mathfrak{M}_{1\pm}
=
\mathcal{R}_{1\pm, 1} + \mathcal{R}_{1\pm, 2} + \mathcal{R}_{1\pm, 3},
\end{align*}
in which
\begin{align*}
\mathcal{R}_{1\pm, 1}
=
& \mp {1\over 4} e^{i\Theta(t_2, \xi)} \mathcal{F}{T_{m_{1\pm}} (g_\pm, f_-)} 
\pm {1\over 4} e^{i\Theta(t_1, \xi)} \mathcal{F}{T_{m_{1\pm}} (g_\pm, f_-)},
\\
\mathcal{R}_{1\pm, 2}
=
&\pm{1\over 4} \int_{t_1}^{t_2}  e^{i\Theta} \mathcal{F}T_{m_{1\pm}}(g_\pm, \partial_s f_-) \, \d s,
\\
\mathcal{R}_{1\pm, 3}
=
&\pm{1\over 4} \int_{t_1}^{t_2} i \partial_s \Theta(s, \xi) e^{i\Theta} \mathcal{F}T_{m_{1\pm}}(g_\pm, f_-) \, \d s,
\end{align*}
with
$$
m_{1\pm} = {e^{is\Phi_{1\mp}}\over i \Phi_{1\mp}} |\eta|^{-1} \langle \xi - \eta\rangle^{-1}.
$$

For the term $\mathcal{R}_{1\pm, 3}$, we integrate by part in time again to get
\begin{align*}
\mathcal{R}_{1\pm, 3}
=
\mathcal{R}_{1\pm, 4}
+
\mathcal{R}_{1\pm, 5}
+
\mathcal{R}_{1\pm, 6},
\end{align*}
in which 
\begin{align*}
\mathcal{R}_{1\pm, 4}
=
& \pm {i\over 4}  \partial_s \Theta(t_2, \xi) e^{i\Theta(t_2, \xi)} \mathcal{F}{T_{m_{2\pm}} (g_\pm, f_-)} 
 \mp {i\over 4}  \partial_s \Theta(t_1, \xi) e^{i\Theta(t_1, \xi)} \mathcal{F}{T_{m_{2\pm}} (g_\pm, f_-)},
\\
\mathcal{R}_{1\pm, 5}
=
&\mp{i\over 4} \int_{t_1}^{t_2}  \partial_s \Theta(s, \xi) e^{i\Theta} \mathcal{F}T_{m_{2\pm}}(g_\pm, \partial_s f_-) \, \d s,
\\
\mathcal{R}_{1\pm, 6}
=
&\mp{i\over 4} \int_{t_1}^{t_2} \big(i( \partial_s \Theta(s, \xi))^2 + \partial_s \partial_s \Theta(s, \xi)\big) e^{i\Theta} \mathcal{F}T_{m_{2\pm}}(g_\pm, f_-) \, \d s,
\end{align*}
with
$$
m_{2\pm} = {e^{is\Phi_{1\mp}}\over (i \Phi_{1\mp})^2} |\eta|^{-1} \langle \xi - \eta\rangle^{-1}.
$$

Let $p_1 = {1\over 10}$, and define for any regular function $u(t, x)$ that
\begin{align*}
    \widehat{u_{L_1}}(t, \eta) = \phi_{\leq 0}(\eta\langle t\rangle^{p_1}) \widehat{u}(t, \eta),
\end{align*}
while
\begin{align*}
    \widehat{u_{H_1}}(t, \eta) = \phi_{\geq 1}(\eta\langle t\rangle^{p_1}) \widehat{u}(t, \eta).
\end{align*}

For $\mathfrak{a}\in \{ 1, 2, 4, 5, 6\}$, we have
\begin{align*}
\mathcal{R}_{1\pm, \mathfrak{a}}
=
\mathcal{R}_{1\pm, \mathfrak{a}, L_1} + \mathcal{R}_{1\pm, \mathfrak{a}, H_1},
\end{align*}
in which with $A \in \{ L_1, H_1 \}$
\begin{align*}
\mathcal{R}_{1\pm, 1, A}
=
& \mp {1\over 4} e^{i\Theta(t_2, \xi)} \mathcal{F}{T_{m_{1\pm}} (g_{\pm, A}, f_-)} 
 \pm {1\over 4} e^{i\Theta(t_1, \xi)} \mathcal{F}{T_{m_{1\pm}} (g_{\pm, A}, f_-)},
\\
\mathcal{R}_{1\pm, 2, A}
=
&\pm{1\over 4} \int_{t_1}^{t_2}  e^{i\Theta} \mathcal{F}T_{m_{1\pm}}(g_{\pm, A}, \partial_s f_-) \, \d s,
\\
\mathcal{R}_{1\pm, 4, A}
=
& \pm {i\over 4}  \partial_s \Theta(t_2, \xi) e^{i\Theta(t_2, \xi)} \mathcal{F}{T_{m_{2\pm}} (g_{\pm, A}, f_-)} 
 \mp {i\over 4}  \partial_s \Theta(t_1, \xi) e^{i\Theta(t_1, \xi)} \mathcal{F}{T_{m_{2\pm}} (g_{\pm, A}, f_-)},
\\
\mathcal{R}_{1\pm, 5, A}
=
&\mp{i\over 4} \int_{t_1}^{t_2}  \partial_s \Theta(s, \xi) e^{i\Theta} \mathcal{F}T_{m_{2\pm}}(g_{\pm, A}, \partial_s f_-) \, \d s,
\\
\mathcal{R}_{1\pm, 6, A}
=
&\mp{i\over 4} \int_{t_1}^{t_2} \big(i( \partial_s \Theta(s, \xi))^2 + \partial_s \partial_s \Theta(s, \xi)\big) e^{i\Theta} \mathcal{F}T_{m_{2\pm}}(g_{\pm, A}, f_-) \, \d s.
\end{align*}

\textbf{Case I: Low frequency.}

In the low frequency case where $|\eta| \lesssim \langle s\rangle^{-p_1}$, we note that
$$
\Phi_{1\pm} \geq 1.
$$

\textbf{Subcase $I_1$: $\mathcal{R}_{1\pm, 1, L_1}$.}

We have
\begin{align*}
&\|\phi_k(\xi) \mathcal{R}_{1\pm, 1, L_1} \|
\\
\lesssim
&\sum_{k_1\geq k_2 + 4} \|\phi_k(\xi)\mathcal{F}{T_{m_{1\pm}} (P_{k_1} g_{\pm, L_1}, P_{k_2}f_-)} \|
\\
&+
\sum_{k_2\geq k_1 + 4} \|\phi_k(\xi)\mathcal{F}{T_{m_{1\pm}} (P_{k_1} g_{\pm,  L_1}, P_{k_2}f_-)} \|
\\
&+
\sum_{|k_1 - k_2| \leq 3} \|\phi_k(\xi)\mathcal{F}{T_{m_{1\pm}} (P_{k_1} g_{\pm, L_1}, P_{k_2}f_-)} \|
\\
= 
&\mathfrak{F}_{11} + \mathfrak{F}_{12} + \mathfrak{F}_{13}.
\end{align*}

We first estimate $\mathfrak{F}_{11}$. In this case, $|k_1 - k|\leq 6$ and $2^{k_2} \leq 2^{k_1} \lesssim 2^{-p_1 m}$.
We proceed to have
\begin{align*}
\mathfrak{F}_{11}
\lesssim
&\sum_{\substack{|k_1 -k|\leq 6, \, k_1\geq k_2 + 4}} 2^{-k_1} 2^{-k_2^+} \big\| \phi_{\leq 0}(\eta \langle t\rangle^{p_1}) \widehat{P_{k_1}g_\pm}\big \|  \big\| \widehat{P_{k_2}f_-} \big\|_{L^1}
\\
\lesssim
&(C_1 \epsilon)^2 \sum_{\substack{|k_1 -k|\leq 6 \\k_1\geq k_2 + 4, \, k_1\leq -p_1m+4}} 2^{k_2} 2^{-k_2^+} 2^{C_E k_2^-} 2^{D_E k_2^+}
\\
\lesssim
&(C_1 \epsilon)^2 2^{-(p_1 + \delta D_E) m} 2^{C_E k^-} 2^{D_E k^+},
\end{align*}
in which we used the fact that $k\leq \delta m$.
For the term $\mathfrak{F}_{12}$, we note $|k_2 - k|\leq 6$, $2^{k_1} \leq 2^{k_2}$ and  $2^{k_1}\lesssim 2^{-p_1 m}$, and we have
\begin{align*}
\mathfrak{F}_{12}
\lesssim
&\sum_{\substack{|k_2 -k|\leq 6, \,k_2\geq k_1 + 4}} 2^{-k_1} 2^{-k_2^+} \big\| \phi_{\leq 0}(\eta \langle t\rangle^{p_1}) \widehat{P_{k_1}g_\pm} \big\|_{L^1}  \big\| \widehat{P_{k_2}f_-} \big\|
\\
\lesssim
&(C_1 \epsilon)^2 \sum_{\substack{|k_2 -k|\leq 6 \\ k_2\geq k_1 + 4, \, k_1\leq -p_1m+4}} 2^{k_1} 2^{-k_2^+} 2^{C_E k_2^-} 2^{D_E k_2^+}
\\
\lesssim
&(C_1 \epsilon)^2 2^{-p_1 m} 2^{C_E k^-} 2^{(-1+D_E)k^+}. 
\end{align*}
We then bound $\mathfrak{F}_{13}$. In this case, $\min\{ k_1, k_2\} \geq k - 6$, $2^{k} \lesssim 2^{k_1}\lesssim 2^{-p_1 m}$, and $2^{k^+}\lesssim 1$. We get
\begin{align*}
\mathfrak{F}_{13}
\lesssim
&2^k \sum_{\min\{k_1, k_2\}\geq k-6, \,|k_1 - k_2| \leq 3} \big\|\phi_{[k-2, k+2]}(\xi)\mathcal{F}{T_{m_{1\pm}} (P_{k_1} g_{\pm, L_1}, P_{k_2}f_-)} \big\|_{L^\infty}
\\
\lesssim
&2^k \sum_{\min\{k_1, k_2\}\geq k-6, \,|k_1 - k_2| \leq 3} 2^{-k_1} 2^{-k_2^+} \big\| \phi_{\leq 0}(\eta \langle t\rangle^{p_1}) \widehat{P_{k_1}g_\pm} \big\|_{L^\infty}  \big\| \widehat{P_{k_2}f_-} \big\|_{L^1}
\\
\lesssim
&(C_1 \epsilon)^2 \sum_{\substack{\min\{k_1, k_2\}\geq k-6\\ |k_1- k_2| \leq  3, \, k_1\leq -p_1m+4}} 2^{k} 2^{-k_2^+} 2^{C_E k_2^-} 2^{D_E k_2^+}
\\
\lesssim
&(C_1 \epsilon)^2 2^{-p_1 m} 2^{C_E k^-} 2^{(-1+D_E)k^+}. 
\end{align*}

\textbf{Subcase $I_2$: $\mathcal{R}_{1\pm, 2, L_1}$.}

We now bound $\|\phi_k(\xi) \mathcal{R}_{1\pm, 2, L_1} \|$.
Note
\begin{align*}
& 2^{-m} \|\phi_k(\xi) \mathcal{R}_{1\pm, 2, L_1} \|
\\
\lesssim
&\sum_{k_1\geq k_2 + 4} \|\phi_k(\xi)\mathcal{F}{T_{m_{1\pm}} (P_{k_1} g_{\pm, L_1}, \partial_s (P_{k_2}f_-))} \|
\\
&+
\sum_{k_2\geq k_1 + 4} \|\phi_k(\xi)\mathcal{F}{T_{m_{1\pm}} (P_{k_1} g_{\pm, L_1}, \partial_s (P_{k_2}f_-))} \|
\\
&+
\sum_{|k_1 - k_2| \leq 3} \|\phi_k(\xi)\mathcal{F}{T_{m_{1\pm}} (P_{k_1} g_{\pm, L_1}, \partial_s (P_{k_2}f_-))} \|
\\
= 
&\mathfrak{F}_{21} + \mathfrak{F}_{22} + \mathfrak{F}_{23}.
\end{align*}
Recall from Proposition \ref{prop:bound-t}, that 
\begin{align*}
\|\partial_s f_\pm\|
\lesssim
\|n E\|
\lesssim
C_1^2 \epsilon^2 s^{-1+\delta}.
\end{align*}

In succession, we have
\begin{align*}
\mathfrak{F}_{21}
\lesssim
&\sum_{\substack{|k_1 -k|\leq 6, \,k_1\geq k_2 + 4}} 2^{-k_1} 2^{-k_2^+} \big\| \phi_{\leq 0}(\eta \langle t\rangle^{p_1}) \widehat{P_{k_1}g_\pm} \big\|  \big\| \partial_s (\widehat{P_{k_2}f_-}) \big\|_{L^1}
\\
\lesssim
&(C_1 \epsilon)^2 \sum_{\substack{|k_1 -k|\leq 6\\k_1\geq k_2 + 4, \, k_1\leq -p_1m+4}} 2^{k_2^-}   2^{-(1-\delta)m}
\\
\lesssim
&(C_1 \epsilon)^2 2^{-(1-\delta+ \delta D_E)m-(1-C_E) p_1 m} 2^{C_E k^-} 2^{ D_E k^+}. 
\end{align*}

\begin{align*}
\mathfrak{F}_{22}
\lesssim
&\sum_{\substack{|k_2 -k|\leq 6, \,k_2\geq k_1 + 4}} 2^{-k_1} 2^{-k_2^+} \big\| \phi_{\leq 0}(\eta \langle t\rangle^{p_1}) \widehat{P_{k_1}g_\pm} \big\|_{L^1}  \big\| \partial_s (\widehat{P_{k_2}f_-)} \big\|
\\
\lesssim
&(C_1 \epsilon)^2 \sum_{\substack{|k_2 -k|\leq 6\\k_2\geq k_1 + 4, \, k_1\leq -p_1m+4}} 2^{k_1} 2^{-k_2^+} 2^{-(1-\delta)m}
\\
\lesssim
&(C_1 \epsilon)^2 2^{-(1-\delta)m-(1-C_E) p_1 m} 2^{C_E k^-} 2^{-k^+}. 
\end{align*}

\begin{align*}
\mathfrak{F}_{23}
\lesssim
&2^k \sum_{\min\{k_1, k_2\}\geq k-6, \,|k_1 - k_2| \leq 3} \big\|\phi_{[k-2, k+2]}(\xi)\mathcal{F}{T_{m_{1\pm, L_1}} (P_{k_1} g_\pm, \partial_s (P_{k_2}f_-))} \big\|_{L^\infty}
\\
\lesssim
&2^k \sum_{\min\{k_1, k_2\}\geq k-6, \,|k_1 - k_2| \leq 3} 2^{-k_1} 2^{-k_2^+} \big\| \phi_{\leq 0}(\eta \langle t\rangle^{p_1}) \widehat{P_{k_1}g_\pm} \big\|_{L^\infty}  \big\| \partial_s (\widehat{P_{k_2}f_-)} \big\|_{L^1}
\\
\lesssim
&(C_1 \epsilon)^2 \sum_{\substack{\min\{k_1, k_2\}\geq k-6 \\|k_1- k_2| \leq  3, \, k_1\leq -p_1m+4}} 2^{k} 2^{-k_2^+} 2^{-(1-\delta)m}
\\
\lesssim
&(C_1 \epsilon)^2 2^{-(1-\delta)m-(1-C_E) p_1 m} 2^{C_E k^-} 2^{-k^+}. 
\end{align*}

\textbf{Subcase $I_3$: The rest.}

By the same treatment as $\|\phi_k(\xi) \mathcal{R}_{1\pm, 1, L_1} \|$ and $\|\phi_k(\xi) \mathcal{R}_{1\pm, 2, L_1} \|$, and the decay estimates
$$
 | \partial_s \Theta(s, \xi)|
\lesssim
C_1 \epsilon \, 2^{-p m},
$$
one easily obtains
$$
\|\phi_k(\xi) \mathcal{R}_{1\pm, 4, L_1} \|
\lesssim
(C_1 \epsilon)^2 2^{-(p+p_1 +\delta D_E) m} 2^{C_E k^-} 2^{D_E k^+},
$$
and
$$
\|\phi_k(\xi) \mathcal{R}_{1\pm, 5, L_1} \|
\lesssim
(C_1 \epsilon)^2 2^{-(p-\delta+ \delta D_E)m-(1-C_E) p_1 m} 2^{C_E k^-} 2^{ D_E k^+}.  
$$
For the term $\|\phi_k(\xi) \mathcal{R}_{1\pm, 6, L_1} \|$, one notices that
$$
|\partial_s \Theta(s, \xi)|^2 + | \partial_s\partial_s \Theta(s, \xi)|
\lesssim
C_1 \epsilon 2^{-2p m}.
$$
Therefore, a similar treatment to $\|\phi_k(\xi) \mathcal{R}_{1\pm, 1, L_1} \|$ yields
$$
\|\phi_k(\xi) \mathcal{R}_{1\pm, 6, L_1} \|
\lesssim
(C_1 \epsilon)^2 2^{-(2p+p_1 - 1+\delta D_E) m} 2^{C_E k^-} 2^{D_E k^+}. 
$$

\textbf{Case II: High frequency.}

In the high frequency case where $|\eta| \gtrsim \langle s\rangle^{-p_1}$, we recall that
$$
|\Phi_{1\pm}| 
\gtrsim 
{1\over \langle \xi\rangle} + {1\over \langle \xi-\eta\rangle},
\qquad
|\nabla_\eta \Phi_{1\pm}| 
\gtrsim 
\langle \xi-\eta\rangle^{-2},
\qquad
|\nabla \nabla \Phi_{1\pm}| 
\lesssim 
{1\over \langle\xi-\eta\rangle} + {1\over |\eta|}.
$$

\textbf{Subcase $II_1$: $\mathcal{R}_{1\pm, 1, H_1}$.}

We treat the term $\mathcal{R}_{1\pm, 1, H_1}$.
We first look at $\mathcal{F}{T_{m_{1+}} (g_{+, H_1}, f_-)}$, and integrate by part in $\eta$ leads us to
\begin{align*}
&\mathcal{F}{T_{m_{1+}} (g_{+, H_1}, f_-)}
\\
=
&\int {\nabla_\eta \Phi_{1-} \cdot \nabla_\eta e^{it\Phi_{1-}} \over -t \Phi_{1-} |\nabla_\eta \Phi_{1-}|^2} 
{\widehat{g_{+,H_1}}(\eta) \over |\eta|}
{\widehat{f}_-(\xi-\eta) \over \langle \xi-\eta\rangle} \, \d\eta
\\
=
&\int \nabla\cdot \Big( {\nabla_\eta \Phi_{1-} \over t \Phi_{1-} |\nabla_\eta \Phi_{1-}|^2} \Big) e^{it\Phi_{1-}} 
    {\widehat{g_{+,H_1}}(\eta) \over |\eta|}
{\widehat{f}_-(\xi-\eta) \over \langle \xi-\eta\rangle} 
\\
&+
{e^{it\Phi_{1+}} \nabla_\eta \Phi_{1-} \over t \Phi_{1-} |\nabla_\eta \Phi_{1-}|^2} \cdot
\Big( 
{\nabla_\eta\widehat{g_{+,H_1}}(\eta) \over |\eta|}
{\widehat{f}_-(\xi-\eta) \over \langle \xi-\eta\rangle} 
-
{\widehat{g_{+,H_1}}(\eta) \over |\eta|}
{\nabla_\eta \widehat{f}_-(\xi-\eta) \over \langle \xi-\eta\rangle} 
\\
&\qquad\quad-
{\eta\widehat{g_{+,H_1}}(\eta) \over |\eta|^3}
{\widehat{f}_-(\xi-\eta) \over \langle \xi-\eta\rangle} 
-
{ \widehat{g_{+,H_1}}(\eta) \over |\eta|}
{(\eta-\xi)\widehat{f}_-(\xi-\eta) \over \langle \xi-\eta\rangle^3} 
\Big)
\\
=
&\mathfrak{G}_{11} + \mathfrak{G}_{12}.
\end{align*}

Denote
$$
\Psi = \nabla\cdot \Big( {\nabla_\eta \Phi_{1-} \over -t \Phi_{1-} |\nabla_\eta \Phi_{1-}|^2} \Big) e^{it\Phi_{1-}}. 
$$
One easily finds that
$$
t|\Psi|
\lesssim
{|\nabla\nabla \Phi_{1-}| \over |\Phi_{1-}| |\nabla \Phi_{1-}|^2}
+
{1\over |\Phi_{1-}|^2}
\lesssim
\langle \xi-\eta\rangle^4 + {\langle\xi-\eta\rangle^5 \over |\eta|}.
$$
We proceed to get
\begin{equation}
\begin{aligned}
&\|\phi_k(\xi) \mathfrak{G}_{11} \|
\\
\lesssim
&\sum_{k_1\geq k_2 + 4} \Big\|\phi_k(\xi)\int \Psi {\widehat{P_{k_1} g_{+, H_1}}(\eta) \over |\eta|} {\widehat{P_{k_2}f_-}(\xi-\eta) \over \langle\xi-\eta\rangle}  \Big\|
\\
&+
\sum_{k_2\geq k_1 + 4} \Big\|\phi_k(\xi)\int \Psi {\widehat{P_{k_1} g_{+, H_1}}(\eta) \over |\eta|} {\widehat{P_{k_2}f_-}(\xi-\eta) \over \langle\xi-\eta\rangle}  \Big\|
\\
&+
\sum_{|k_1 - k_2| \leq 3} \Big\|\phi_k(\xi)\int \Psi {\widehat{P_{k_1} g_{+, H_1}}(\eta) \over |\eta|} {\widehat{P_{k_2}f_-}(\xi-\eta) \over \langle\xi-\eta\rangle}  \Big\|
\\
=
&\mathfrak{G}_{11a} + \mathfrak{G}_{11b} + \mathfrak{G}_{11c}.
\end{aligned}
\end{equation}
We have
\begin{align*}
    \mathfrak{G}_{11a}
    \lesssim
    &2^{-m} \sum_{\substack{|k_1 -k|\leq 6 \\ k_1\geq k_2 + 4, \, k_1\geq -p_1m -4}} (1+2^{-k_1}) 2^{4k_2^+} \big\|\widehat{P_{k_1} g_{+, H_1}}\big\|
    \big\|\widehat{P_{k_2}f_-}\big\|_{L^1}
    \\
    \lesssim
    &(C_1 \epsilon)^2 2^{-m} \sum_{\substack{|k_1 -k|\leq 6 \\k_1\geq k_2 + 4, \, k_1\geq -p_1m -4}} (1+2^{-k_1}) 2^{k_2^-} 2^{-2 k_2^+}
    \\
    \lesssim
    &(C_1 \epsilon)^2 2^{((1+C_E)p_1 + 2\delta - \delta D_E-1)m} 2^{C_E k^-} 2^{D_E k^+},
\end{align*}
in which we used $\|f_-\|_{H^7} \lesssim C_1 \epsilon$.

Similarly, we obtain
\begin{align*}
    \mathfrak{G}_{11b}
    \lesssim
    (C_1 \epsilon)^2 2^{((1+C_E)p_1 + 2\delta - \delta D_E-1)m} 2^{C_E k^-} 2^{D_E k^+}.
\end{align*}
For the term $\mathfrak{G}_{11c}$, by Young's inequality for convolution we get
\begin{align*}
    \mathfrak{G}_{11c}
    \lesssim
    &2^{-m} \sum_{\substack{|k_1 - k_2| \leq 3, \, k_1\geq -p_1m -4 \\ \min\{k_1, k_2\}\geq k-6}} 2^{k} \big(1+2^{-k_1}\big) 2^{-k_1} 2^{4k_2^+} \big\|\widehat{P_{k_1} g_{+, H_1}}\big\|
    \big\|\widehat{P_{k_2}f_-}\big\|
    \\
    \lesssim
    &(C_1 \epsilon)^2 2^{( p_1-1)m} 2^{ k^-} 2^{D_E k^+}.
\end{align*}

Next, we consider the term $\mathfrak{G}_{12}$.
Denote
$$
\Psi_2 = {e^{it\Phi_{1-}}\nabla_\eta \Phi_{1-} \over -t \Phi_{1-} |\nabla_\eta \Phi_{1-}|^2}.
$$
We notice that
$$
t|\Psi_2|
\lesssim
{1\over |\Phi_{1-}| |\nabla_\eta \Phi_{1-}|}
\lesssim
\langle \xi-\eta\rangle^{3}.
$$
We have
\begin{align*}
&\Big\| \phi_k(\xi) \int \Psi_2 \cdot  {\nabla_\eta\widehat{g_{+,H_1}}(\eta) \over |\eta|}
{\widehat{f}_-(\xi-\eta) \over \langle \xi-\eta\rangle} \Big\|
\\
\lesssim
    &\sum_{k_1\geq k_2 + 4} \Big\| \phi_k(\xi) \int \Psi_2 \cdot  {\phi_{k_1}(\eta)\nabla_\eta\widehat{ g_{+,H_1}}(\eta) \over |\eta|}
{\widehat{P_{k_2}f_-}(\xi-\eta) \over \langle \xi-\eta\rangle} \Big\|
    \\
     &+
    \sum_{k_2\geq k_1 + 4} \Big\| \phi_k(\xi) \int \Psi_2 \cdot  {\phi_{k_1}(\eta)\nabla_\eta\widehat{g_{+,H_1}}(\eta) \over |\eta|}
{\widehat{P_{k_2}f_-}(\xi-\eta) \over \langle \xi-\eta\rangle} \Big\|
    \\
     &+
    \sum_{|k_1- k_2| \leq 3} \Big\| \phi_k(\xi) \int \Psi_2 \cdot  {\phi_{k_1}(\eta)\nabla_\eta\widehat{g_{+,H_1}}(\eta) \over |\eta|}
{\widehat{P_{k_2}f_-}(\xi-\eta) \over \langle \xi-\eta\rangle} \Big\|
\\
=
&\mathfrak{G}_{21} + \mathfrak{G}_{22} + \mathfrak{G}_{23}.
\end{align*}
Note that
$$
\nabla_\eta \widehat{g_{+, H_1}}
=
\nabla_\eta \widehat{g_{+}} \phi_{\geq 1}(\eta \langle s\rangle^{p_1})
+
\langle s\rangle^{p_1} \widehat{g_{+}} \nabla_\eta \phi_{\geq 1}(\eta \langle s\rangle^{p_1}).
$$
We proceed to bound
\begin{align*}
\mathfrak{G}_{21}
\lesssim
&\sum_{\substack{|k_1 -k|\leq 6, \,k_1\geq k_2 + 4}} 2^{-m} 2^{-k_1} 2^{2k_2^+} \big\|\phi_{k_1}(\eta)\nabla_\eta\widehat{ g_{+,H_1}} \big\|
\big\| \widehat{P_{k_2}f_-} \big\|_{L^1}
\\
\lesssim
& (C_1 \epsilon)^2 \sum_{\substack{|k_1 -k|\leq 6\\k_1\geq k_2 + 4, \, k_1\geq -p_1 m -4}} \Big( 2^{-m} 2^{k_2^-} 2^{-3k^+_2} 
+ 2^{-(1-p_1)m} 2^{k_2^-} 2^{-3k^+_2} \Big)
    \\
    \lesssim
    &(C_1 \epsilon)^2 2^{-(1-p_1 + \delta D_E)m} 2^{C_E k^-} 2^{D_E k^+}.
\end{align*}
We also get
\begin{align*}
\mathfrak{G}_{22}
\lesssim
&\sum_{\substack{|k_2 -k|\leq 6, \,k_2\geq k_1 + 4}} 2^{-m} 2^{-k_1} 2^{2k_2^+} \big\|\phi_{k_1}(\eta)\nabla_\eta\widehat{ g_{+,H_1}}\big\|_{L^1}
\big\| \widehat{P_{k_2}f_-}\big\|
\\
\lesssim
& (C_1 \epsilon)^2 \sum_{\substack{|k_2 -k|\leq 6\\k_2\geq k_1 + 4, \, k_1\geq -p_1 m -4}} \Big( 2^{-m} 2^{k_1^-} 2^{-3k^+_2} 
+ 2^{-(1-p_1)m} 2^{k_1^-} 2^{-3k^+_2} \Big)
    \\
    \lesssim
    &(C_1 \epsilon)^2 2^{-(1-p_1)m} 2^{C_E k^-} 2^{D_E k^+}.
\end{align*}
For $\mathfrak{G}_{23}$, we have
\begin{align*}
\mathfrak{G}_{23}
\lesssim
&\sum_{\min\{k_1, k_2\}\geq k-6, \,|k_1 - k_2| \leq 3} 2^{-m} 2^{k} 2^{-k_1} 2^{2k_2^+} \Big\|\int |\phi_{k_1}(\eta)\nabla_\eta\widehat{ g_{+,H_1}}|(\eta)
 |\widehat{P_{k_2}f_-}|(\xi-\eta)\Big\|_{L^\infty}
 \\
 \lesssim
 &\sum_{\min\{k_1, k_2\}\geq k-6, \,|k_1 - k_2| \leq 3} 2^{-m} 2^{k} 2^{-k_1} 2^{2k_2^+} \big\| \nabla_\eta\widehat{P_{k_1} g_{+,H_1}}\big\|
 \big\|\widehat{P_{k_2}f_-}(\xi-\eta)\big\|
\\
    \lesssim
    &(C_1 \epsilon)^2 2^{-(1-p_1 )m} 2^{C_E k^-} 2^{D_E k^+}.
\end{align*}

\textbf{Subcase $II_2$: $\mathcal{R}_{1\pm, 2, H_1}$.}

Employing the equation of $\widehat{f}_-$, which reads
\begin{align*}
&\partial_s\widehat{f_-}(s, \xi-\eta)
\\
=
&e^{-is\langle \xi-\eta\rangle} (-n E)
=
-{i\over 2} {e^{-2is\langle\xi-\eta\rangle}\widehat{h_+}(\xi-\eta) -  \widehat{h_-}(\xi-\eta)  \over \langle\xi-\eta\rangle},
\end{align*}
we have
\begin{align*}
     \mathcal{F} T_{m_{1\pm}}(g_{\pm, H_1}, \partial_s f_-)
     =
     &{i\over 2}\int {\nabla_\eta \Phi_{2\mp} \cdot \nabla_{\eta} e^{is\Phi_{2\mp}}\over s\Phi_{1\mp} |\nabla_\eta \Phi_{2\mp}|^2} {\widehat{g_{\pm, H_1}}(\eta) \over |\eta|} { \widehat{h_+}(s, \xi-\eta) \over \langle \xi-\eta\rangle^2} \, \d\eta
     \\
     &+
     {1\over 2} \int {e^{is\Phi_{1\mp}} \over \Phi_{1\mp}}
    {\widehat{g_{\pm, H_1}}(\eta) \over |\eta|} {\widehat{h_-}(s, \xi-\eta) \over \langle \xi-\eta\rangle^2}
     \\
     =&
     \mathfrak{G}_{31} + \mathfrak{G}_{32} + \mathfrak{G}_{33},
\end{align*}
in which (recall $\Phi_{2\pm} = \langle \xi\rangle - \langle \xi-\eta\rangle \pm |\eta|$) we apply integration by part in $\eta$ and 
\begin{align*}
    \mathfrak{G}_{31}
    =
    &-{i\over 2}\int \nabla_\eta\cdot {\nabla_\eta \Phi_{2\mp} \over s\Phi_{1\mp} |\nabla_\eta \Phi_{2\mp}|^2} e^{is\Phi_{2\mp}} 
    {\widehat{g_{\pm, H_1}}(\eta) \over |\eta|} { \widehat{h_+}(s, \xi-\eta) \over \langle \xi-\eta\rangle^2} \, \d\eta,
    \\
    \mathfrak{G}_{32}
    =
    &-{i\over 2}\int {e^{it\Phi_{2\mp}} \nabla_\eta \Phi_{2\mp} \over s \Phi_{1\mp} |\nabla_\eta \Phi_{2\mp}|^2} \cdot
\Big( 
{\nabla_\eta\widehat{g_{\pm,H_1}}(\eta) \over |\eta|}
{\widehat{h_+}(\xi-\eta) \over \langle \xi-\eta\rangle^2} 
-
{\widehat{g_{\pm,H_1}}(\eta) \over |\eta|}
{\nabla_\eta \widehat{h_+}(\xi-\eta) \over \langle \xi-\eta\rangle^2} 
\\
&\qquad-
{\eta\widehat{g_{\pm,H_1}}(\eta) \over |\eta|^3}
{\widehat{h_+}(\xi-\eta) \over \langle \xi-\eta\rangle^2} 
-
2{ \widehat{g_{\pm,H_1}}(\eta) \over |\eta|}
{(\eta-\xi)\widehat{h_+}(\xi-\eta) \over \langle \xi-\eta\rangle^4} 
\Big) \, \d \eta,
\\
\mathfrak{G}_{33}
    =
    &{1\over 2} \int {e^{is\Phi_{1\mp}} \over \Phi_{1\mp}}
    {\widehat{g_{\pm, H_1}}(\eta) \over |\eta|} {\widehat{h_-}(s, \xi-\eta) \over \langle \xi-\eta\rangle^2} \, \d\eta.
\end{align*}
We note that 
\begin{align*}
    \Big|\nabla_\eta\cdot {\nabla_\eta \Phi_{2\mp} \over s\Phi_{1\mp} |\nabla_\eta \Phi_{2\mp}|^2} \Big|
    \lesssim
    &s^{-1}  (1+|\eta|^{-1}) \langle\xi-\eta\rangle^5,
    \\
    \Big|{e^{it\Phi_{2\mp}} \nabla_\eta \Phi_{2\mp} \over s \Phi_{1\mp} |\nabla_\eta \Phi_{2\mp}|^2}\Big|
    \lesssim
    &s^{-1}   \langle\xi-\eta\rangle^3.
\end{align*}
Thanks to the fast decay of $\|P_{k_2} h_\pm\|_{H^6}\leq \|h_\pm\|_{H^6} \lesssim (C_1 \epsilon)^2 2^{-(1-\delta)m}$ for all $k_2 \in \mathbb{Z}$, we have
\begin{align*}
\|\phi_{k}(\xi) e^{i\Theta} \mathfrak{G}_{31}\|
\lesssim
2^{(-2 -C_E p_1 + 3\delta)m} 2^{C_E k^-} 2^{D_E k^+},
\end{align*}
as well as
\begin{align*}
\|\phi_{k}(\xi) e^{i\Theta} \mathfrak{G}_{32}\|
\lesssim
2^{(-2 + (1+C_E)p_1 + 5\delta)m} 2^{C_E k^-} 2^{D_E k^+}.
\end{align*}

For the part involving $\mathfrak{G}_{33}$, we integrate by part in time one more time to get
\begin{align*}
    \int_{t_1}^{t_2} e^{i\Theta} \mathfrak{G}_{33} \, \d s
    =
    \mathfrak{G}_{33a} + \int_{t_1}^{t_2} e^{i\Theta} ( \mathfrak{G}_{33b}
    +\mathfrak{G}_{33c} + \mathfrak{G}_{33d}) \, \d s,
\end{align*}
in which 
\begin{align*}
    \mathfrak{G}_{33a}
    =
    &{1\over 2} e^{i\Theta(t_2)} \int {e^{is\Phi_{1\mp}} \over i \Phi_{1\mp}^2}  {\widehat{g_{\pm, H_1}}(t_2, \eta) \over |\eta|} {\widehat{h_-}(t_2, \xi-\eta) \over \langle \xi-\eta\rangle^2}
    \\
    &-
    {1\over 2} e^{i\Theta(t_1)} \int {e^{is\Phi_{1\mp}} \over i \Phi_{1\mp}^2}  {\widehat{g_{\pm, H_1}}(t_1, \eta) \over |\eta|} {\widehat{h_-}(t_1, \xi-\eta) \over \langle \xi-\eta\rangle^2},
    \\
    \mathfrak{G}_{33b}
    =
    &-{1\over 2} \partial_s\Theta \int {e^{is\Phi_{1\mp}} \over  \Phi_{1\mp}^2}  {\widehat{g_{\pm, H_1}}(s, \eta) \over |\eta|} {\widehat{h_-}(s, \xi-\eta) \over \langle \xi-\eta\rangle^2},
    \\
    \mathfrak{G}_{33c}
    =
    &-{1\over 2}  \int {e^{is\Phi_{1\mp}} \over  i \Phi_{1\mp}^2}  {\partial_s \widehat{g_{\pm, H_1}}(s, \eta) \over |\eta|} {\widehat{h_-}(s, \xi-\eta) \over \langle \xi-\eta\rangle^2},
    \\
    \mathfrak{G}_{33d}
    =
    &-{1\over 2}  \int {e^{is\Phi_{1\mp}} \over  i \Phi_{1\mp}^2}  { \widehat{g_{\pm, H_1}}(s, \eta) \over |\eta|} {\partial_s \widehat{h_-}(s, \xi-\eta) \over \langle \xi-\eta\rangle^2}.
\end{align*}

Recall Lemma \ref{lem:phase}, $|\Phi_{1\pm}|^{-2}$ does not cause trouble in bounding these terms.

With the aid of $\| h_\pm\|_{H^6} \lesssim (C_1 \epsilon)^2 2^{-(1-\delta)m}$ for all $k_2 \in \mathbb{Z}$, the terms involving $\mathfrak{G}_{33a}$ can be easily bounded with
\begin{align*}
    \|\phi_k(\xi) \mathfrak{G}_{33a}\|
    \lesssim
    (C_1 \epsilon)^3 2^{-(1-\delta)m} 2^{C_E k^-} 2^{D_E k^+}.
\end{align*}

Employing the additional decay of $\partial_t \Theta$ in Lemma \ref{lem:Theta}, we get
\begin{align*}  
\|\phi_k(\xi) \mathfrak{G}_{33b}\|
\lesssim
(C_1 \epsilon)^3 2^{-(1+p-\delta)m} 2^{C_E k^-} 2^{D_E k^+}.
\end{align*}

Next, note the simple fact that 
$$
\partial_s \widehat{g_{\pm, H_1}}(\eta)
=
p_1 s \langle s\rangle^{p_1-2} \widehat{g_\pm}(\eta) \nabla_\eta \phi_{\geq 1}(\eta\langle s\rangle^{p_1}) \cdot \eta,
$$
we derive
\begin{align*}  
\|\phi_k(\xi) \mathfrak{G}_{33c}\|
\lesssim
(C_1 \epsilon)^3 2^{-(2-\delta)m} 2^{C_E k^-} 2^{D_E k^+}.
\end{align*}

By the fast decay of $\|\partial_s {h_-}\|$, which reads  
\begin{equation}\label{eq:M1-001}
    \|\partial_s {h_-}\|
    \lesssim
    \| \Delta|E|^2 \cdot E - (n)^2 E - 2\partial_\alpha n \, \partial^\alpha E\|
    \lesssim
    (C_1 \epsilon)^3 s^{-2+2\delta},
\end{equation}
we obtain
\begin{align*}
\|\phi_k(\xi) \mathfrak{G}_{33d}\|
\lesssim
(C_1 \epsilon)^3 2^{-(2-3\delta)m} 2^{C_E k^-} 2^{D_E k^+}.
\end{align*}

\textbf{Subcase $II_3$: The rest.}

Recall in  Lemma \ref{lem:Theta} that $|\partial_s \Theta(s, \xi)| \lesssim C_1 \epsilon 2^{-p m}$.
Using the same strategy as we bound $\|\phi_k(\xi) \mathcal{R}_{1\pm, 1, H_1} \|$, one derives
\begin{align}
\|\phi_k(\xi) \mathcal{R}_{1\pm, 4, H_1} \|
\lesssim
&(C_1 \epsilon)^2 2^{-(p-(1+C_E)p_1- 3\delta )m} 2^{C_E k^-} 2^{D_E k^+},
\\
\|\phi_k(\xi) \mathcal{R}_{1\pm, 5, H_1} \|
\lesssim
&(C_1 \epsilon)^2 2^{-(p-2\delta)m}2^{ p_1 m} 2^{C_E k^-} 2^{D_E k^+}. 
\end{align}
By Lemma \ref{lem:Theta}, i.e.,
$$
|\partial_s \Theta(s, \xi)|^2 + | \partial_s\partial_s \Theta(s, \xi)|
\lesssim
C_1 \epsilon 2^{-2p m},
$$
 a similar treatment to $\|\phi_k(\xi) \mathcal{R}_{1\pm, 1, H_1} \|$ gives
\begin{align}
\|\phi_k(\xi) \mathcal{R}_{1\pm, 6, H_1} \|
\lesssim
&(C_1 \epsilon)^2 2^{-(2p-(1+C_E)p_1- 5\delta )m} 2^{C_E k^-} 2^{D_E k^+}.
\end{align}

To conclude, we get \eqref{eq:M1} by gathering the estimates established in \textbf{Case I} and \textbf{Case II}. The proof is completed.
\end{proof}

\begin{proposition}\label{prop:M2}
For all $m\gg 1$, $t_1, t_2 \in [2^m - 2, 2^{m+1}+2]$, and $k\leq \delta m$, it holds
\begin{align}\label{eq:M2}
\Big\| \phi_k(\xi) \int_{t_1}^{t_2} e^{i\Theta(s, \xi)}\mathcal{M}_2(s, \xi) \, \d s \Big\|
\lesssim
(C_1 \epsilon)^2 2^{-\delta m} 2^{C_E k^-} 2^{D_E k^+}. 
\end{align}

\end{proposition}

\begin{proof}
We note
\begin{align*}
\int_{t_1}^{t_2} e^{i\Theta(s, \xi)}\mathcal{M}_2(s, \xi) \, \d s 
=
\mathfrak{M}_{2+} + \mathfrak{M}_{2-},
\end{align*}
with
\begin{align*}
    \mathfrak{M}_{2\pm} 
    = \pm {1\over 4} \int_{t_1}^{t_2} e^{i\Theta(s, \xi)}
         \int e^{is\Phi_{2\mp}} {\widehat{g_{\pm, H}}(\eta) \over |\eta|} { \widehat{f_+}(s, \xi-\eta) \over \langle \xi-\eta\rangle} \, \d\eta ds,
\end{align*}
in which 
$$
\Phi_{2\pm}
=
\langle\xi\rangle - \langle\xi-\eta\rangle \pm |\eta|.
$$
For $|\eta|\gtrsim \langle t\rangle^{-p}$, we note that
\begin{align*}
    |\Phi_{2\pm}| 
    \gtrsim
    {|\eta| \over (\langle\eta\rangle + \langle \xi-\eta\rangle)^2},
    \quad
    |\nabla_\eta \Phi_{2\pm}| 
    \gtrsim
    \langle \xi-\eta\rangle^{-2},
    \quad
    |\nabla\nabla \Phi_{2\pm}| 
    \lesssim
    {1\over \langle\xi-\eta\rangle} + {1\over |\eta|}.
\end{align*}
In the following, we focus only on $\mathfrak{M}_{2+}$ as the term $\mathfrak{M}_{2-}$ can be treated in the same way.

We have by integration by part in time that
\begin{align*}
    \mathfrak{M}_{2+}
    =
    \mathcal{R}_{2+, 1}+\mathcal{R}_{2+, 2}+\mathcal{R}_{2+, 3}+\mathcal{R}_{2+, 4},
\end{align*}
where
\begin{align*}
\mathcal{R}_{2+, 1}
=
&{1\over 4} e^{i\Theta(t_2, \xi)} \mathcal{F} T_{m_{3-}}(g_{+, H}, f_+)
    -
    {1\over 4} e^{i\Theta(t_1, \xi)} \mathcal{F} T_{m_{3-}}(g_{+, H}, f_+),
    \\
    \mathcal{R}_{2+, 2}
    =
    &-{1\over 4} \int_{t_1}^{t_2} e^{i\Theta} \mathcal{F}T_{m_{3-}}(g_{+, H}, \partial_s f_+),
    \\
    \mathcal{R}_{2+, 3}
    =
    &-{1\over 4} \int_{t_1}^{t_2} e^{i\Theta} \mathcal{F}T_{m_{3-}}(\partial_s g_{+, H},  f_+),
    \\
    \mathcal{R}_{2+, 4}
=
&-{i\over 4} \int_{t_1}^{t_2} e^{i\Theta} \partial_s \Theta(s, \xi) \mathcal{F} T_{m_{3-}}(g_{+, H}, f_+)
\end{align*}
with 
$$
m_{3-} = {e^{is\Phi_{2-}} \over i\Phi_{2-}} |\eta|^{-1} \langle \xi-\eta\rangle^{-1}.
$$

\textbf{Case I: Estimates for $\|\phi_k(\xi)\mathcal{R}_{2+, 1}\|$.}

We note that
\begin{align*}
     \mathcal{F} T_{m_{3-}}(g_{+, H}, f_+)
     =
     &-\int {\nabla_\eta \Phi_{2-} \cdot \nabla_{\eta} e^{is\Phi_{2-}}\over is\Phi_{2-} |\nabla_\eta \Phi_{2-}|^2} {\widehat{g_{+, H}}(\eta) \over |\eta|} { \widehat{f_+}(s, \xi-\eta) \over \langle \xi-\eta\rangle} \, \d\eta
     \\
     =&
     \mathfrak{H}_{11} + \mathfrak{H}_{12},
\end{align*}
in which
\begin{align*}
    \mathfrak{H}_{11}
    =
    &\int \nabla_\eta\cdot {\nabla_\eta \Phi_{2-} \over is\Phi_{2-} |\nabla_\eta \Phi_{2-}|^2} e^{is\Phi_{2-}} 
    {\widehat{g_{+, H}}(\eta) \over |\eta|} { \widehat{f_+}(s, \xi-\eta) \over \langle \xi-\eta\rangle} \, \d\eta,
    \\
    \mathfrak{H}_{12}
    =
    &\int {e^{it\Phi_{2-}} \nabla_\eta \Phi_{2-} \over is \Phi_{2-} |\nabla_\eta \Phi_{2-}|^2} \cdot
\Big( 
{\nabla_\eta\widehat{g_{+,H}}(\eta) \over |\eta|}
{\widehat{f}_+(\xi-\eta) \over \langle \xi-\eta\rangle} 
-
{\widehat{g_{+,H}}(\eta) \over |\eta|}
{\nabla_\eta \widehat{f}_+(\xi-\eta) \over \langle \xi-\eta\rangle} 
\\
&\qquad-
{\eta\widehat{g_{+,H}}(\eta) \over |\eta|^3}
{\widehat{f}_+(\xi-\eta) \over \langle \xi-\eta\rangle} 
-
{ \widehat{g_{+,H}}(\eta) \over |\eta|}
{(\eta-\xi)\widehat{f}_+(\xi-\eta) \over \langle \xi-\eta\rangle^3} 
\Big).
\end{align*}
We note that 
\begin{equation}\label{eq:M2-001}
    \begin{aligned}
    \Big|\nabla_\eta\cdot {\nabla_\eta \Phi_{2-} \over s\Phi_{2-} |\nabla_\eta \Phi_{2-}|^2} \Big|
    \lesssim
    &s^{-1}  |\eta|^{-2} \big( \langle\eta\rangle^6 + \langle\xi-\eta\rangle^6  \big),
    \\
    \Big| {\nabla_\eta \Phi_{2-} \over s \Phi_{2-} |\nabla_\eta \Phi_{2-}|^2}\Big|
    \lesssim
    &s^{-1} |\eta|^{-1} \big( \langle\eta\rangle^4 + \langle\xi-\eta\rangle^4  \big).
\end{aligned}
\end{equation}
Similar to the way we bound $\mathfrak{G}_{11}$, one can get
\begin{align*}
    \|\phi_k(\xi) \mathfrak{H}_{11}\|
    \lesssim    
    (C_1 \epsilon)^2 2^{((1+C_E)p + 6\delta-1)m} 2^{C_E k^-} 2^{D_E k^+}.
\end{align*}

Similar to the way we bound $\mathfrak{G}_{12}$, recalling also that 
$$
\nabla_\eta \widehat{g_{+, H}}
=
\nabla_\eta \widehat{g_{+}} \phi_{\geq 1}(\eta \langle s\rangle^{p})
+
\langle s\rangle^{p} \widehat{g_{+}} \nabla_\eta \phi_{\geq 1}(\eta \langle s\rangle^{p}),
$$
one can get
\begin{align*}
    \|\phi_k(\xi) \mathfrak{H}_{12}\|
    \lesssim    
    &(C_1 \epsilon)^2 2^{(-1+(1+2C_E)p+ 4\delta )m} 2^{C_E k^-} 2^{D_E k^+}.
\end{align*}

\textbf{Case II: Estimates for $\|\phi_k(\xi)\mathcal{R}_{2+, 2}\|$.}

Recalling again that
$$
\partial_s\widehat{f_+}(s, \xi-\eta)
=
-{i\over 2} {\widehat{h_+}(\xi-\eta) - e^{2is\langle\xi-\eta\rangle} \widehat{h_-}(\xi-\eta)  \over \langle\xi-\eta\rangle}.
$$
This leads us to
\begin{align*}
     \mathcal{F} T_{m_{3-}}(g_{+, H}, \partial_s f_+)
     =
     &-\int {\nabla_\eta \Phi_{2-} \cdot \nabla_{\eta} e^{is\Phi_{2-}}\over s\Phi_{2-} |\nabla_\eta \Phi_{2-}|^2} {\widehat{g_{+, H}}(\eta) \over |\eta|} { \partial_s\widehat{f_+}(s, \xi-\eta) \over \langle \xi-\eta\rangle} \, \d\eta
     \\
     =&
     \mathfrak{H}_{21} + \mathfrak{H}_{22} + \mathfrak{H}_{23},
\end{align*}
in which
\begin{align*}
    \mathfrak{H}_{21}
    =
    &-{i\over 2}\int \nabla_\eta\cdot {\nabla_\eta \Phi_{2-} \over s\Phi_{2-} |\nabla_\eta \Phi_{2-}|^2} e^{is\Phi_{2-}} 
    {\widehat{g_{+, H}}(\eta) \over |\eta|} { \widehat{h_+}(s, \xi-\eta) \over \langle \xi-\eta\rangle^2} \, \d\eta,
    \\
    \mathfrak{H}_{22}
    =
    &-{i\over 2}\int {e^{it\Phi_{2-}} \nabla_\eta \Phi_{2-} \over s \Phi_{2-} |\nabla_\eta \Phi_{2-}|^2} \cdot
\Big( 
{\nabla_\eta\widehat{g_{+,H}}(\eta) \over |\eta|}
{\widehat{h}_+(\xi-\eta) \over \langle \xi-\eta\rangle^2} 
-
{\widehat{g_{+,H}}(\eta) \over |\eta|}
{\nabla_\eta \widehat{h}_+(\xi-\eta) \over \langle \xi-\eta\rangle^2} 
\\
&\qquad-
{\eta\widehat{g_{+,H}}(\eta) \over |\eta|^3}
{\widehat{h}_+(\xi-\eta) \over \langle \xi-\eta\rangle^2} 
-
2{ \widehat{g_{+,H}}(\eta) \over |\eta|}
{(\eta-\xi)\widehat{h}_+(\xi-\eta) \over \langle \xi-\eta\rangle^4} 
\Big) \, \d \eta,
\\
\mathfrak{H}_{23}
    =
    &{1\over 2} \int {e^{is\Phi_{1-}} \over \Phi_{2-}}
    {\widehat{g_{+, H}}(\eta) \over |\eta|} {\widehat{h_-}(s, \xi-\eta) \over \langle \xi-\eta\rangle^2} \, \d \eta.
\end{align*}
We recall \eqref{eq:M2-001}, 
and thanks to the fast decay of $ \|h_\pm\|_{H^7} \lesssim (C_1 \epsilon)^2 2^{(-1+\delta)m}$ for all $k_2 \in \mathbb{Z}$, we have
\begin{align*}
\|\phi_{k}(\xi) e^{i\Theta} \mathfrak{H}_{21}\|
\lesssim
2^{(-2 + (1+2C_E)p + 7\delta)m} 2^{C_E k^-} 2^{D_E k^+}.
\end{align*}
Together with the estimate of $\| \nabla_\eta \widehat{h_+}(\eta)\|_{H^7}$ in Lemma \ref{lem:d-xi-L2}, we get
\begin{align*}
\|\phi_{k}(\xi) e^{i\Theta} \mathfrak{H}_{22}\|
\lesssim
2^{(-2 + (1+2C_E)p + 5\delta)m} 2^{C_E k^-} 2^{D_E k^+}.
\end{align*}

For the part involving $\mathfrak{H}_{23}$, we integrate by part in time one more time to get
\begin{align*}
    \int_{t_1}^{t_2} e^{i\Theta} \mathfrak{H}_{23} \, \d s
    =
    \mathfrak{H}_{23a} + \int_{t_1}^{t_2} e^{i\Theta} ( \mathfrak{H}_{23b}
    +\mathfrak{H}_{23c} + \mathfrak{H}_{23d}) \, \d s,
\end{align*}
in which 
\begin{align*}
    \mathfrak{H}_{23a}
    =
    &{1\over 2} e^{i\Theta(t_2)} \int {e^{is\Phi_{1-}} \over i \Phi_{1-} \Phi_{2-}} {\widehat{g_{+, H}}(t_2, \eta) \over |\eta|} {\widehat{h_-}(t_2, \xi-\eta) \over \langle \xi-\eta\rangle^2}
    \\
    -
    &{1\over 2} e^{i\Theta(t_1)} \int {e^{is\Phi_{1-}} \over i \Phi_{1-} \Phi_{2-}} {\widehat{g_{+, H}}(t_1, \eta) \over |\eta|} {\widehat{h_-}(t_1, \xi-\eta) \over \langle \xi-\eta\rangle^2},
    \\
    \mathfrak{H}_{23b}
    =
    -&{1\over 2} \partial_s\Theta \int {e^{is\Phi_{1-}} \over \Phi_{1-} \Phi_{2-}} {\widehat{g_{+, H}}(s, \eta) \over |\eta|} {\widehat{h_-}(s, \xi-\eta) \over \langle \xi-\eta\rangle^2},
    \\
    \mathfrak{H}_{23c}
    =
    -&{1\over 2}  \int {e^{is\Phi_{1-}} \over i\Phi_{1-} \Phi_{2-}} {\partial_s \widehat{g_{+, H}}(s, \eta) \over |\eta|} {\widehat{h_-}(s, \xi-\eta) \over \langle \xi-\eta\rangle^2},
    \\
    \mathfrak{H}_{23d}
    =
    -&{1\over 2}  \int {e^{is\Phi_{1-}} \over i\Phi_{1-} \Phi_{2-}} {\widehat{g_{+, H}}(s, \eta) \over |\eta|} {\partial_s\widehat{h_-}(s, \xi-\eta) \over \langle \xi-\eta\rangle^2}.
\end{align*}
The terms involving $\mathfrak{H}_{23a}, \mathfrak{H}_{23b}, \mathfrak{H}_{23c}$, and $\mathfrak{H}_{23d}$ can be estimated directly, where we need \eqref{eq:M1-001} when bounding $\mathfrak{H}_{23d}$. Here, we omit the details and only illustrate the final bounds
\begin{align*}
    \|\phi_k(\xi) e^{i\Theta} \mathfrak{H}_{23a}\|
    \lesssim
    & (C_1 \epsilon)^2 2^{(-1+2C_E p + 7\delta)m} 2^{C_E k^-} 2^{D_E k^+}, 
    \\
    \|\phi_k(\xi) e^{i\Theta} \mathfrak{H}_{23b}\| 
    \lesssim
    & (C_1 \epsilon)^2 2^{(-1+(-1+2C_E) p+7\delta)m} 2^{C_E k^-} 2^{D_E k^+}, 
    \\
    \|\phi_k(\xi) e^{i\Theta} \mathfrak{H}_{23c}\| + \|\phi_k(\xi) e^{i\Theta}\mathfrak{H}_{23d}\|
    \lesssim
    & (C_1 \epsilon)^2 2^{(-2+(1+2C_E)p + 9\delta)m} 2^{C_E k^-} 2^{D_E k^+}.
\end{align*}

\textbf{Case III: Estimates for $\|\phi_k(\xi)\mathcal{R}_{2+, 3}\|$.}

Recalling 
$$
\partial_s \widehat{g_{+, H}}
=
ps\langle s\rangle^{p-2} \eta\cdot\nabla\phi_{\geq 1}(\eta \langle s\rangle^p) \widehat{g_+}(\eta).
$$
We find, by integration by part in $\eta$, that
\begin{align*}
&\mathcal{F}T_{m_{3-}}(\partial_s g_{+, H},  f_+)
    \\
    =
    &{p\over 4} \langle s\rangle^{p-2} \int {\nabla_\eta \Phi_{2_-}\cdot\nabla_\eta e^{is\Phi_{2_-}} \over \Phi_{2_-} |\nabla_\eta \Phi_{2_-}|^2 }
    {\eta \cdot \nabla \phi_{\geq 1}(\eta\langle s\rangle^p)\widehat{g_{+}}(\eta) \over |\eta|} { \widehat{f_+}(s, \xi-\eta) \over \langle \xi-\eta\rangle} 
    \\
    = &\mathfrak{H}_{31} + \mathfrak{H}_{32} + \mathfrak{H}_{33},
\end{align*}
where
\begin{align*}
    \mathfrak{H}_{31}
    =
    &-{p\over 4} \langle s\rangle^{p-2} \int \nabla\cdot \Psi_3 e^{is\Phi_{2-}}
    {\eta \cdot \nabla \phi_{\geq 1}(\eta\langle s\rangle^p)\widehat{g_{+}}(\eta) \over |\eta|} { \widehat{f_+}(s, \xi-\eta) \over \langle \xi-\eta\rangle},
    \\
    \mathfrak{H}_{32}
    =
    &-{p\over 4} \langle s\rangle^{p-2} \int e^{is\Phi_{2-}} \Psi_3 \cdot 
    \nabla (\eta \cdot \nabla \phi_{\geq 1}(\eta\langle s\rangle^p))
    {\widehat{g_{+}}(\eta) \over |\eta|} { \widehat{f_+}(s, \xi-\eta) \over \langle \xi-\eta\rangle},
    \\
    \mathfrak{H}_{33}
    =
    &-{p\over 4} \langle s\rangle^{p-2} \int e^{is\Phi_{2-}} (\eta \cdot \nabla \phi_{\geq 1}(\eta\langle s\rangle^p)) \Psi_3 \cdot 
    \nabla \Big(
    {\widehat{g_{+}}(\eta) \over |\eta|} { \widehat{f_+}(s, \xi-\eta) \over \langle \xi-\eta\rangle}\Big).
\end{align*}
with
$$
\Psi_3 = {\nabla_\eta \Phi_{2_-} \over \Phi_{2_-} |\nabla \Phi_{2_-}|^2}.
$$
We are able to get
\begin{align}
    \sum_{\mathfrak{b}=1, 2, 3}\|\phi_k(\xi) e^{i\Theta}\mathfrak{H}_{3\mathfrak{b}} \|
    \lesssim
    (C_1 \epsilon)^2 2^{(-2+ (1+2C_E)p +7\delta)m} 2^{C_E k^-} 2^{D_E k^+}. 
\end{align}

\textbf{Case IV: Estimates for $\|\phi_k(\xi)\mathcal{R}_{2+, 4}\|$.}

The strategy is very similar to that of \textbf{Case I}, except that we need to do integration by part in $\eta$ more than once for certain terms.
Recall that
\begin{align*}
     \mathcal{F} T_{m_{3-}}(g_{+, H}, f_+)
     =
     \mathfrak{H}_{11} + \mathfrak{H}_{12},
\end{align*}
and we further decompose 
\begin{align*}
    \mathfrak{H}_{12}
    =
    \mathfrak{H}_{41a} + \mathfrak{H}_{41b} + \mathfrak{H}_{42} + \mathfrak{H}_{43} + \mathfrak{H}_{44},
   \end{align*} 
where
\begin{align*}
    \mathfrak{H}_{41a}
    =
    &\int {e^{it\Phi_{2-}} \nabla_\eta \Phi_{2-} \over is \Phi_{2-} |\nabla_\eta \Phi_{2-}|^2} \cdot
{(\nabla_\eta\widehat{g_{+}})_H(\eta) \over |\eta|}
{\widehat{f}_+(\xi-\eta) \over \langle \xi-\eta\rangle}, 
\\
\mathfrak{H}_{41b}
    =
    &\langle s\rangle^p \int {e^{it\Phi_{2-}} \nabla_\eta \Phi_{2-} \over is \Phi_{2-} |\nabla_\eta \Phi_{2-}|^2} \cdot
{\nabla \phi_{\geq 1}(\eta\langle s\rangle^p) \widehat{g_{+}}(\eta) \over |\eta|}
{\widehat{f}_+(\xi-\eta) \over \langle \xi-\eta\rangle}, 
\\
\mathfrak{H}_{42}
    =
    &-\int {e^{it\Phi_{2-}} \nabla_\eta \Phi_{2-} \over is \Phi_{2-} |\nabla_\eta \Phi_{2-}|^2} \cdot
{\widehat{g_{+,H}}(\eta) \over |\eta|}
{\nabla_\eta \widehat{f}_+(\xi-\eta) \over \langle \xi-\eta\rangle}, 
\\
\mathfrak{H}_{43}
    =
    &-\int {e^{it\Phi_{2-}} \nabla_\eta \Phi_{2-} \over is \Phi_{2-} |\nabla_\eta \Phi_{2-}|^2} \cdot
{\eta\widehat{g_{+,H}}(\eta) \over |\eta|^3}
{\widehat{f}_+(\xi-\eta) \over \langle \xi-\eta\rangle}, 
\\
\mathfrak{H}_{44}
    =
    &-\int {e^{it\Phi_{2-}} \nabla_\eta \Phi_{2-} \over is \Phi_{2-} |\nabla_\eta \Phi_{2-}|^2} \cdot
{ \widehat{g_{+,H}}(\eta) \over |\eta|}
{(\eta-\xi)\widehat{f}_+(\xi-\eta) \over \langle \xi-\eta\rangle^3}.
\end{align*}
By Lemma \ref{lem:d-xi-L2}, we have
\begin{align*}
    \|\phi_{k_1}(\eta) \nabla_\eta \widehat{f_+}\|
    \lesssim
    &C_1 \epsilon \langle s\rangle^\delta 2^{-k_1^+},
    \\
    \|\phi_{k_2}(\eta) \nabla_\eta \widehat{g_+}\|
    \lesssim
    &C_1 \epsilon \langle s\rangle^\delta 2^{k_2}.
\end{align*}
For the terms involving $\mathfrak{H}_{41a}, \mathfrak{H}_{42}, \mathfrak{H}_{44}$, we have bounds
\begin{align*}
    &\|\phi_k(\xi) \partial_s \Theta \cdot \mathfrak{H}_{41a}\|
    +
    \|\phi_k(\xi) \partial_s \Theta \cdot \mathfrak{H}_{42}\|
    +
    \|\phi_k(\xi) \partial_s \Theta \cdot \mathfrak{H}_{44}\|
    \\
    \lesssim
    &(C_1 \epsilon)^2 2^{(-1 + (2C_E-1)p + 7\delta)m} 2^{C_E k^-} 2^{D_E k^+}.
\end{align*}
For the rest part involving $\mathfrak{H}_{11}, \mathfrak{H}_{41b}, \mathfrak{H}_{43}$, we integrate by part in $\eta$ again, with the help of Lemma \ref{lem:phase}, to get the desired bounds
\begin{align*}
    &\|\phi_k(\xi) \partial_s \Theta \cdot \mathfrak{H}_{11}\|
    +
    \|\phi_k(\xi) \partial_s \Theta \cdot \mathfrak{H}_{41b}\|
    +
    \|\phi_k(\xi) \partial_s \Theta \cdot \mathfrak{H}_{43}\|
    \\
    \lesssim
    &(C_1 \epsilon)^2 2^{(-2 + (2C_E+1)p + 7\delta)m} 2^{C_E k^-} 2^{D_E k^+}.
\end{align*}


Combining the estimates in \textbf{Case I}, \textbf{Case II}, \textbf{Case III}, and \textbf{Case IV}, we conclude \eqref{eq:M2}. The proof is finished.
\end{proof}

Before estimating $\mathcal{M}_3$, we need some lemmas from Ionescu-Pausader \cite[Chapter 2]{Ionescu-P2}. 

For $f\in H^1(\mathbb{R}^2)$ and $(l, j)\in \mathcal{J}$ defined in \eqref{eq:J}, denote
\begin{align*}
    (Q_{l,j}f)(x) = \widetilde{\phi}^{(l)}_j(x) P_l f(x),
    \qquad
    f_{l,j} = \widetilde{P}_l(Q_{l,j}f),
\end{align*}
with $\widetilde{P}_l = P_{[l-2, l+2]}$.
One easily sees that
\begin{align*}
    P_l f 
    = \sum_{j\geq -\min\{l, 0 \}} Q_{l,j}f
    = \sum_{j\geq -\min\{l, 0 \}} f_{l,j}.
\end{align*}

The following result is the two space dimensional analogue of \cite[Lemma 2.6]{Ionescu-P2}. But in $\mathbb{R}^2$, we have some loss in the estimates compared with those in $\mathbb{R}^{3}$.

\begin{lemma}\label{lem:Q}
    Given $f\in H^1(\mathbb{R}^2)$ and $l\in\mathbb{Z}$, let
    \begin{align}
        A_l = \|P_l f\|_{H^1} + \sum_a\|\phi_l(\xi) (\partial_{\xi_a} \widehat{f})(\xi)\|,
        \qquad
        B_l = \Big( \sum_{j\geq \max \{-l, 0 \}} 2^{2j} \|Q_{l,j}f\|^2  \Big)^{1\over 2}.
    \end{align}
    Then, one has
    \begin{eqnarray}
        B_l
        \lesssim
        \left\{
\begin{array}{lll}
\sum_{|l' - l|\leq 4} A_{l'},
\qquad\qquad &l\geq 1,
\\[1ex]
\langle l\rangle \sup_{l'} \|P_{l'} f\|_{H^1} + \sum_{|l' - l|\leq 6} A_{l'},
\qquad &l\leq 0.
\end{array}
\right.
\end{eqnarray}
\end{lemma}
The proof can be found in the appendix.

As a consequence, we have the following result.
\begin{lemma}\label{lem:Q_lj}
    We have
    \begin{align}
        \|Q_{l,j} f_+\|(t)
        \lesssim
        C_1 \epsilon \langle t\rangle^{\delta} 2^{-j} 2^{-\delta l^- -{1\over 2} l^+}.
    \end{align}
\end{lemma}

\begin{proposition}\label{prop:M3}
For all $m\gg 1$, $t_1, t_2 \in [2^m - 2, 2^{m+1}+2]$, and $k\leq \delta m$, it holds
\begin{align}\label{eq:M3}
\Big\| \phi_k(\xi) \int_{t_1}^{t_2} e^{i\Theta(s, \xi)} \mathcal{M}_3(s, \xi) \, \d s \Big\|
\lesssim
(C_1 \epsilon)^2 2^{-\delta m} 2^{C_E k^-} 2^{D_E k^+}.
\end{align}
\end{proposition}
\begin{proof}
Recall the decomposition
$$
f_+
=
\sum_{(l, j)\in \mathcal{J}} f_{+, l, j}, 
$$
and we have 
\begin{align*}
    \phi_k(\xi) \mathcal{M}_3
    =
    \sum_{(l, j)\in \mathcal{J}} \mathfrak{K}_{k; l, j}
\end{align*}
in which
\begin{align*}
    \mathfrak{K}_{k; l, j}
    =
    &-{i\over 2} \phi_k(\xi)  \int_{\mathbb{R}^2} \widehat{\ell_L}(s, \eta) \Big(e^{is(\langle\xi\rangle-\langle \xi-\eta\rangle)}  { \widehat{f_{+,l,j}}(s, \xi-\eta) \over \langle \xi - \eta\rangle} 
    - e^{is {\xi\over \langle\xi\rangle} \cdot \eta} { \widehat{f_{+,l,j}}(s, \xi) \over \langle \xi \rangle}  \Big) \, \d \eta
    \\
    =:&\mathfrak{L}_{11} + \mathfrak{L}_{12}.
\end{align*}
One important observation is that $\widehat{f_{+,l,j}}(\xi)$ is supported around frequency $2^l$.

We recall that
\begin{equation}
\begin{aligned}
    &\big\|\widehat{\ell_L}\big\|_{L^1}
    \lesssim
    C_1 \epsilon \langle s\rangle^{-p},
    \\
    &\big\|\widehat{\ell_L}\big\|
    \lesssim
    \|\ell\|
    \lesssim
    C_1 \epsilon \langle s\rangle^\delta.
\end{aligned}
\end{equation}

\textbf{Case I: Large $j\geq {1\over 2}m$.}

Applying Young's inequality for convolution,  we have
\begin{align*}
    \|\mathfrak{L}_{11}\|
    \lesssim
    \big\| \widehat{\ell_L}\big\|_{L^1} \big\| \widehat{f_{+,l,j}} \big\|
    \lesssim
    (C_1 \epsilon)^2 2^{-pm} 2^{-j} 2^{-\delta l^- -{1\over 2} l^+} 2^{\delta m},
\end{align*}
as well as the bound
\begin{align*}
    \|\mathfrak{L}_{11}\|
    \lesssim
    2^k \big\| \widehat{\ell_L}\big\| \big\| \widehat{f_{+,l,j}}\big\|
    \lesssim
    (C_1 \epsilon)^2 2^{k} 2^{-j} 2^{-\delta l^- - {1\over 2} l^+} 2^{2\delta m}.
\end{align*}
By interpolating these two estimates and the fact $j+l\geq 0$, we get
\begin{align*}
    \|\mathfrak{L}_{11}\|
    \lesssim
    &(C_1 \epsilon)^2 2^{C_E k^-} 2^{D_E k^+} 2^{(-1+2\delta)j} 2^{\delta l^- - {1\over 2} l^+} 2^{(-1+C_E)pm} 2^{3\delta m}.
\end{align*}

We also have
\begin{eqnarray}
    \|\mathfrak{L}_{12}\|
    \lesssim
    &\big\| \widehat{\ell_L}\big\|_{L^1} \big\|\phi_k(\xi) \widehat{f_{+,l,j}} \big\| \notag
    \\
    \lesssim
    &\left\{
\begin{array}{lll}
    (C_1 \epsilon)^2 2^{-pm} 2^{-j} 2^{-\delta l^- -{1\over 2} l^+} 2^{\delta m},
    \qquad\quad &|l-k|\leq 4,
    \\[1ex]
    0, \qquad\quad &|l-k|\geq 5.
\end{array}
\right.
\end{eqnarray}

Therefore, we obtain
\begin{equation}
\begin{aligned}
    \sum_{(l, j)\in \mathcal{J},\, j\geq {1\over 2}m}\|\mathfrak{K}_{k; l, j}\|
    \lesssim
    &\sum_{(l, j)\in \mathcal{J},\, j\geq {1\over 2}m} \big( \|\mathfrak{L}_{11}\| + \|\mathfrak{L}_{12}\| \big)
    \\
    \lesssim
    &(C_1 \epsilon)^2 2^{-(p+{1\over 2} -9\delta)m} 2^{C_E k^-} 2^{D_E k^+}.
\end{aligned}
\end{equation}

\textbf{Case II: Small $j\leq {1\over 2}m$.}

In this case, we need to utilize the cancellation in the integral. 

First, one finds that
\begin{align*}
    &\Big|e^{is(\langle\xi\rangle - \langle \xi-\eta\rangle)} 
    - e^{is {\xi\over \langle\xi\rangle} \cdot \eta} \Big|
    \\
    =
    &\Big|e^{is(\langle\xi\rangle - \langle \xi-\eta\rangle - {\xi\over \langle\xi\rangle} \cdot \eta)} 
    - 1 \Big|
    \\
    \lesssim
    &\Big|s(\langle\xi\rangle - \langle \xi-\eta\rangle - {\xi\over \langle\xi\rangle} \cdot \eta)\Big|
    \lesssim
    2^{-2pm + m}.
\end{align*}

Second, we note
\begin{align*}
    &\Big|{ \widehat{f_{+,l,j}}(s, \xi-\eta) \over \langle \xi - \eta\rangle} - { \widehat{f_{+,l,j}}(s, \xi) \over \langle \xi \rangle} \Big|
    \\
    \lesssim
    &\Big|\int \widehat{Q_{l,j} f_+}(\rho) \Big( {\widehat{\phi_{\leq j+4}}(\xi-\rho) \phi_{[l-2, l+2]}(\xi) \over \langle \xi\rangle} - {\widehat{\phi_{\leq j+4}}(\xi-\rho-\eta) \phi_{[l-2, l+2]}(\xi-\eta) \over \langle \xi-\eta\rangle} \Big) \, \d\rho\Big|.
\end{align*}
Some tedious calculation provides
$$
\Big|  {\widehat{\phi_{\leq j+4}}(\xi-\rho) \phi_{[l-2, l+2]}(\xi) \over \langle \xi\rangle} - {\widehat{\phi_{\leq j+4}}(\xi-\rho-\eta) \phi_{[l-2, l+2]}(\xi-\eta) \over \langle \xi-\eta\rangle}  \Big|
\lesssim
{2^{-pm} 2^j 2^{2j} \over  (1+ 2^j |\xi-\rho|)^{10}},
$$
which further leads us to
\begin{align*}
    \Big|{ \widehat{f_{+,l,j}}(s, \xi-\eta) \over \langle \xi - \eta\rangle} - { \widehat{f_{+,l,j}}(s, \xi) \over \langle \xi \rangle} \Big|
    \lesssim
    &\int |\widehat{Q_{l,j} f_+}|(\rho) {2^{-pm} 2^j 2^{2j} \over  (1+ 2^j |\xi-\rho|)^{10}} \, \d\rho
    \\
    \lesssim
    &\|\widehat{Q_{l,j} f_+}\|  \Big\| {2^{-pm} 2^j 2^{2j} \over  (1+ 2^j |\xi|)^{10}} \Big\|
    \\
    \lesssim
    & C_1 \epsilon 2^{-pm} 2^{\delta l^- - {1\over 2} l^+} 2^{(1+2\delta) j}.
\end{align*}
and consequently
\begin{align*}
    &\Big\|{ \widehat{f_{+,l,j}}(s, \xi-\eta) \over \langle \xi - \eta\rangle} - { \widehat{f_{+,l,j}}(s, \xi) \over \langle \xi \rangle} \Big\|
    \\
    \lesssim
    &\big\|\widehat{Q_{l,j} f_+}\big\| \int  {2^{-pm} 2^j 2^{2j} \over  (1+ 2^j |\xi-\rho|)^{10}} \, \d\rho
    \\
    \lesssim
    &2^{-pm} 2^j \big\|\widehat{Q_{l,j} f_+}\big\|
    \lesssim
    C_1 \epsilon 2^{-pm} 2^{\delta l^- - {1\over 2} l^+} 2^{2\delta j}.
\end{align*}
By using these two bounds, we derive
\begin{align*}
    &\sum_{(l,j)\in\mathcal{J}, \, j\leq {1\over 2}m}\|\mathfrak{K}_{k;l,j}\|
    \\
    \lesssim
    &\sum_{(l,j)\in\mathcal{J}, \, j\leq {1\over 2}m}\big\|\widehat{\ell_L}\big\|_{L^1}\Big( \Big\|\phi_k(\xi) \int |\widehat{Q_{l,j} f_+}|(\rho) {2^{-pm} 2^j 2^{2j} \over  (1+ 2^j |\xi-\rho|)^{10}} \, \d\rho \Big\|
    \\
    &\qquad\qquad
    +
    2^{(-2p+1)m}\Big\|\phi_k(\xi) {\widehat{f_{+, l, j}}(\xi)\over \langle\xi\rangle} \Big\|
    \Big)
    \\
    \lesssim
    &\sum_{\substack{(l, j)\in\mathcal{J}, \, j\leq {1\over 2}m}}
    (C_1 \epsilon)^2 \Big( 2^{-2pm} 2^{ (C_E + \delta)j} +  2^{(-3p+1)m} 2^{-{1\over 2}j}   \Big)2^{\delta l^- - {1\over 2} l^+} 2^{C_E k^-} 2^{D_E k^+} 2^{6\delta m}
    \\
    \lesssim
    &(C_1 \epsilon)^2 2^{C_E k^-} 2^{D_E k^+} 2^{(-3p+1+9\delta)m}.
\end{align*}

To conclude, we obtain
\begin{align}
    \|\phi_k(\xi) e^{i\Theta} \mathcal{M}_3\|
    \lesssim
    \sum_{(l, j)\in \mathcal{J}} \|\mathfrak{K}_{k; l,j}\|
    \lesssim
    (C_1 \epsilon)^2 2^{-\delta m} 2^{C_E k^-} 2^{D_E k^+}.
\end{align}

The estimates in \textbf{Case I} and \textbf{Case II} leads to \eqref{eq:M3}.
The proof is done.
\end{proof}

\begin{proposition}\label{prop:M4}
    For all $m\gg 1$, $t_1, t_2 \in [2^m - 2, 2^{m+1}+2]$, and $k\leq \delta m$, it holds
\begin{align}\label{eq:M4}
\Big\| \phi_k(\xi) \int_{t_1}^{t_2} e^{i\Theta(s, \xi)} \mathcal{M}_4(s, \xi) \, \d s \Big\|
\lesssim
\epsilon^2 2^{-\delta m} 2^{C_E k^-} 2^{D_E k^+}.
\end{align}
\end{proposition}
\begin{proof}
    We rely on the bounds in Proposition \ref{prop:bound-t} to derive
    \begin{align*}
        \| \Delta \mathfrak{m} E\|_{H^{8}}
        \lesssim
        \big\|(\langle s-r\rangle^{1\over 2} \Delta \mathfrak{m}) \cdot (\langle s-r\rangle^{-{1\over 2}} E)\big\|_{H^8}
        \lesssim
        (C_1 \epsilon)^2 s^{-{3\over 2}},
    \end{align*}
    as well as
    \begin{align*}
        \| \Delta \mathfrak{m} E\|_{L^{1}}
        \lesssim
        \big\|(\langle s-r\rangle^{1\over 2} \Delta \mathfrak{m}) \cdot (\langle s-r\rangle^{-{1\over 2}} E)\big\|_{L^1}
        \lesssim
        (C_1 \epsilon)^2 s^{-{1\over 2}}
    \end{align*}
    Thanks to the fast decay of the term $\Delta \mathfrak{m} E$, we have
    \begin{align*}
        \big\|\phi_k(\xi) \widehat{\Delta \mathfrak{m} E} \big\|
        \lesssim
        \| \Delta \mathfrak{m} E\|
        \lesssim
        (C_1 \epsilon)^2 2^{-{3\over 2}m}.
    \end{align*}

One the other hand, we get
\begin{align*}
    \big\|\phi_k(\xi) \widehat{\Delta \mathfrak{m} E} \big\|
    \lesssim
    \|\phi_k(\xi)\| \|{\Delta \mathfrak{m} E} \|_{L^1}
    \lesssim
    (C_1 \epsilon)^2 2^k 2^{-{m\over 2}}.
\end{align*}

Interpolating these two bounds yields
\begin{align}
    \big\|\phi_k(\xi) \widehat{\Delta \mathfrak{m} E} \big\|
    \lesssim
    2^{C_E k^-} 2^{D_E k^+} 2^{\big(-{3\over 2}(1-C_E) + {C_E\over 2} - D_E \delta\big)m}.
\end{align}

The proof is done.
\end{proof}

\subsubsection{Endgame}

We aim to show the following result, which further implies nonlinear scattering of the field $E$.

\begin{proposition}\label{prop:M-all}
For all $m\gg 1$, $t_1, t_2 \in [2^m - 2, 2^{m+1}+2]$, and $k\in \mathbb{N}$, one has
\begin{align}\label{eq:M-all}
    \big\| \phi_k(\xi) (\widehat{f_*}(t_2, \xi) - \widehat{f_*}(t_1, \xi)) \big\|
    \leq {1\over 2} C_1 \epsilon 2^{C_E k^-} 2^{D_E k^+} 2^{-\delta m}.
\end{align}
\end{proposition}
\begin{proof}

\textbf{Case I: $k\leq \delta m$.}

In this case, Propositions \ref{prop:M1}--\ref{prop:M4} yield the desired results.

\textbf{Case II: $k\geq \delta m$.}

In this case, we note that
\begin{align*}
\|n \, E\|_{H^7}
\lesssim
(C_1 \epsilon)^2 \langle s\rangle^{-1+\delta},  
\end{align*}
which leads us to
\begin{align}
    \|P_k (nE)\|
    \lesssim
    2^{-7k^+} \langle s\rangle^{-1+\delta}.
\end{align}
Therefore, one derives that
\begin{equation}
\begin{aligned}
    &\big\| \phi_k(\xi) (\widehat{f_*}(t_2, \xi) - \widehat{f_*}(t_1, \xi)) \big\|
    \\
    \lesssim
    &\big\|\phi_k(\xi)\widehat{f_+}(t_2, \xi)\big\| + \big\|\phi_k(\xi)\widehat{f_+}(t_1, \xi)\big\|
    \\
    \lesssim
    &2^{-7k^+} \Big(\|f_+(t_0, x)\|_{H^7}
    +\int_{t_0}^{t_1} \| nE\|_{H^7}\, \d s
    + \int_{t_0}^{t_2} \| nE\|_{H^7}\, \d s\Big)
    \\
    \lesssim
    &\big( \epsilon + (C_1 \epsilon)^2\big) 2^{-4k^+} 2^{-\delta m}.
\end{aligned}
\end{equation}

Thus, we complete the proof.
\end{proof}

Finally, we have following bounds for the phase correction $\Theta$.

\begin{proposition}\label{prop:Theta}
    Let $\xi$ be fixed, then it  holds that
\begin{align}
    \lim_{t\to+\infty}|\Theta(t, \xi)|
    =+\infty.
\end{align}
\end{proposition}
\begin{remark}
    In fact, we have a more precise description on $\Theta$, which reads
    \begin{align*}
        \big|\Theta(t, \xi) - \pi \widehat{n_1}(0) \log(t) \big|
        \lesssim
        C(\xi) |\widehat{n_1}(0)|.
    \end{align*}
\end{remark}

\begin{proof}[Proof of Proposition \ref{prop:Theta}]
  
    Recall that
    \begin{align*}
\Theta(t, \xi) 
= 
&{1\over 2} \langle \xi\rangle^{-1} \int_{t_0}^t \int  \phi_{\leq 0}(\eta \langle s\rangle^p) \widehat{\ell}(s, \eta) e^{is {\xi\over \langle \xi\rangle} \cdot \eta } \, \d\eta \d s
\\
=
&\mathfrak{U}_{1}(\xi) + \mathfrak{U}_2(\xi) +\mathfrak{U}_3(\xi),
\end{align*}
in which
\begin{align*}
    \mathfrak{U}_{1}(\xi)
    =
    &{1\over 2} \langle \xi\rangle^{-1} \int_{t_0}^t \int  \phi_{\leq 0}(\eta \langle s\rangle^p) \cos{((s-t_0)|\eta|)} \widehat{n_0}( \eta) e^{is {\xi\over \langle \xi\rangle} \cdot \eta } \, \d\eta \d s,
    \\
    \mathfrak{U}_{2}(\xi)
    =
    &{1\over 2} \langle \xi\rangle^{-1} \int_{t_0}^t \int  \phi_{\leq 0}(\eta \langle s\rangle^p) {\sin{((s-t_0)|\eta|)} \over |\eta|}\big(\widehat{n_1}(\eta) - \widehat{n_1}(0)\big) e^{is {\xi\over \langle \xi\rangle} \cdot \eta } \, \d\eta \d s,
    \\
    \mathfrak{U}_{3}(\xi)
    =
    &{1\over 2} \langle \xi\rangle^{-1} \int_{t_0}^t \int  \phi_{\leq 0}(\eta \langle s\rangle^p) {\sin{((s-t_0)|\eta|)} \over |\eta|}\widehat{n_1}( 0) e^{is {\xi\over \langle \xi\rangle} \cdot \eta } \, \d\eta \d s.
\end{align*}

    \textbf{Case I: Lower bound at $\xi=0$.}

We first focus on the simple case of $\xi = 0$. We abuse notation a bit by writing $\phi_{\leq 0}(|\eta|) = \phi_{\leq 0}(\eta)$ for instance.

One derives that
\begin{align*}
    &|\mathfrak{U}_1(0)| + |\mathfrak{U}_2(0)|
    \\
    \lesssim
    &\epsilon \int_{t_0}^t \int_{} \phi_{\leq 0}(\eta \langle s\rangle^p) \, \d\eta \d s
    \lesssim
    \epsilon.
\end{align*}

For $\mathfrak{U}_3$, we have
\begin{align*}
    \mathfrak{U}_3(0)
    =
    &\pi \widehat{n_1}(0) \int_{t_0}^t \int_{0}^{+\infty} \phi_{\leq 0}(|\eta| \langle s\rangle^p) \sin{((s-t_0)|\eta|)} \, \d|\eta| \d s
    \\
    =
    &\pi \widehat{n_1}(0) \int_{t_0}^t \int_{0}^{+\infty} \langle s\rangle^{-p} \phi_{\leq 0}(\rho) \sin{((s-t_0)\langle s\rangle^{-p}\rho)} \, \d\rho \d s
    \\
    =
    &\mathfrak{U}_{31} + \mathfrak{U}_{32},
\end{align*}
in which (with $A= (s-t_0) \langle s\rangle^{-p}$)
\begin{align*}
    \mathfrak{U}_{31}
    =&(1-p)^{-1}\pi \widehat{n_1}(0)\int_{0}^{+\infty} {\phi_{\leq 0}(\rho) \over \rho} \big( 1- \cos (A\rho) \big)\, \d\rho,
    \\
    \mathfrak{U}_{32}
    =
    &{p\over 1-p} \pi \widehat{n_1}(0)\int_{0}^{+\infty} {\phi_{\leq 0}(\rho) \over \rho} \int_{0}^{A \rho} \sin \tau {1+ t_0s \over (1-p) s^2 + t_0 p s + 1} \, \d\tau \d\rho,
\end{align*}
with $\tau = (s-t_0) \langle s\rangle^{-p} \rho$.

We further decompose
\begin{align*}
    \mathfrak{U}_{31}
    =
    \mathfrak{U}_{311}
    +
    \mathfrak{U}_{312},
\end{align*}
where
\begin{align*}
    \mathfrak{U}_{311}
    =
    &(1-p)^{-1}\pi \widehat{n_1}(0)\int_{0}^{1} {1 \over \rho} \big( 1- \cos (A\rho) \big)\, \d\rho,
    \\
    \mathfrak{U}_{312}
    =
    &(1-p)^{-1}\pi \widehat{n_1}(0)\int_{1}^{2} {\phi_{\leq 0}(\rho) \over \rho} \big( 1- \cos (A\rho) \big)\, \d\rho.
\end{align*}
One finds that 
\begin{align*}
    |\mathfrak{U}_{312}| 
    \lesssim |\widehat{n_1}(0)|,
    \\
    |\mathfrak{U}_{32}| 
    \lesssim |\widehat{n_1}(0)|.
\end{align*}

 A tedious calculation gives that
 \begin{align*}
     &\int_{0}^1 {1- \cos(A\rho) \over \rho} \, \d\rho
     \\
     =
     &-A\int_0^1 \log \rho \sin(A\rho) \, \d\rho
     \\
     =
     &-\int_0^A (\log \rho_1 - \log A)  \sin(\rho_1) \, \d\rho_1
     \\
     =
     &\log A \cdot (1-\cos A) + \mathfrak{U}_{41} + \mathfrak{U}_{42},
 \end{align*}
in which
\begin{align*}
    \mathfrak{U}_{41}
    =
    &-\int_0^1 \log \rho_1 \sin(\rho_1) \, \d\rho_1,
    \\
    \mathfrak{U}_{42}
    =
    &-\int_1^A \log \rho_1 \sin(\rho_1) \, \d\rho_1.
\end{align*}
Easily, we get $|\mathfrak{U}_{41}|\lesssim 1$, and we also obtain
\begin{align*}
    \mathfrak{U}_{42}
    =
    \log A \cdot \cos(A) - \int_1^A {\cos(\rho_1) \over \rho_1} \, \d\rho_1,
\end{align*}
where
\begin{align*}
    \Big|\int_1^A {\cos(\rho_1) \over \rho_1} \, \d\rho_1 \Big|
    \lesssim 1,
\end{align*}
which further yields
\begin{align*}
         &\int_{0}^1 {1- \cos(A\rho) \over \rho} \, \d\rho
= \log A + \mathcal{O}(1).
\end{align*}

    To conclude, we have
    \begin{align}
        \big|\Theta(t, 0) - (1-p)^{-1} \pi \widehat{n_1}(0) \log ((t-t_0) \langle t\rangle^{-p})\big|
        \lesssim
        |\widehat{n_1}(0)|.
    \end{align}

    \textbf{Case II: Lower bound for fixed $\xi$.}
Below the implicit constant in $\lesssim$ might depend on $\xi$, so in that case we use the notation $\lesssim_\xi$ instead.
One can get
\begin{align*}
    |\mathfrak{U}_1(\xi)| + |\mathfrak{U}_2(\xi)|
    \lesssim
    \epsilon \langle \xi\rangle^{-1}.
\end{align*}

We next, without loss of generality, consider the term
\begin{align*}
    \mathfrak{U}_{4}(\xi)
    =
    & \int_{t_0}^t \int  \phi_{\leq 0}(\eta  s^p) {\sin{(s|\eta|)} \over |\eta|} e^{is {\xi\over \langle \xi\rangle} \cdot \eta } \, \d\eta \d s,
\end{align*}
which is radial in $\xi$. We only need to consider the case of $\xi = (|\xi|, 0)$.
Therefore, we get
\begin{align*}
    \mathfrak{U}_{4}(\xi)
    =
    & \int_{t_0}^t \int  \phi_{\leq 0}(\eta  s^p) {\sin{(s|\eta|)} } e^{is {|\xi|\over \langle \xi\rangle}   |\eta| \cos\theta } \, \d\theta \d |\eta| \d s
    \\
    =
    &2\pi \int_{t_0}^t \int  \phi_{\leq 0}(\eta  s^p) {\sin{(s|\eta|)} } J_0(s|\eta| {|\xi|\over \langle \xi\rangle}) \,  \d|\eta| \d s,
    \end{align*}
in which $J_0$ is the Bessel function of the first kind of order zero.

We decompose $\mathfrak{U}_4$ by
\begin{align*}
    \mathfrak{U}_{4}(\xi)
    =
    \mathfrak{U}_{41}(\xi) + \mathfrak{U}_{42}(\xi),
\end{align*}
in which
\begin{align*}
    \mathfrak{U}_{41}(\xi)
    =
    2\pi\int_{t_0}^t \int_{0}^1  {\sin(s^{1-p}\rho) \over s^p} J_0\big(s^{1-p}\rho {|\xi|\over \langle \xi\rangle}\big) \, \d\rho \d s,
    \\
    \mathfrak{U}_{42}(\xi)
    =2\pi\int_{t_0}^t \int_{1}^2 \phi_{\leq 0}(\rho) {\sin(s^{1-p}\rho) \over s^p} J_0\big(s^{1-p}\rho {|\xi|\over \langle \xi\rangle}\big) \, \d\rho \d s.
\end{align*}
We note that
\begin{align*}
    \mathfrak{U}_{42}(\xi)
    =
    {2\pi \over 1-p} \int_1^2 {\phi_{\leq 0}(\rho) \over \rho} \int_{t_0^{1-p} \rho}^{t^{1-p}\rho} \sin(\tau) J_0\big(\tau {|\xi|\over \langle \xi\rangle}\big) \, \d\tau \d\rho.
\end{align*}
By Lemma \ref{lem:Bessel}, we know
\begin{align*}
    |\mathfrak{U}_{42}(\xi)|
    \lesssim_{\xi}
    1.
\end{align*}
On the other hand, we find that
\begin{align*}
    &\mathfrak{U}_{41}(\xi)
    \\
    =
    &2\pi \int_{t_0}^t s^{-1}\int_{0}^{s^{1-p}} \sin(\sigma) J_0\big(\sigma {|\xi|\over \langle \xi\rangle}\big) \, \d\sigma \d s
    \\
    =
    &2\pi \int_{t_0}^t s^{-1}\int_{0}^{+\infty} \sin(\sigma) J_0\big(\sigma {|\xi|\over \langle \xi\rangle}\big) \, \d\sigma \d s
    +
    2\pi \int_{t_0}^t s^{-1}\int_{s^{1-p}}^{+\infty} \sin(\sigma) J_0\big(\sigma {|\xi|\over \langle \xi\rangle}\big) \, \d\sigma \d s
    \\
    =:
    & \mathfrak{U}_{411}(\xi) + \mathfrak{U}_{412}(\xi).
\end{align*}
We employ Lemma \ref{lem:Bessel2} to derive
\begin{align*}
    |\mathfrak{U}_{412}(\xi)|
    \lesssim_\xi 1,
\end{align*}
and
\begin{align*}
    \mathfrak{U}_{411}(\xi)
    =
    2\pi \langle \xi\rangle \log(t). 
\end{align*}

To conclude, we arrive at
\begin{align*}
    |\mathfrak{U}_4(\xi) - 2\pi \langle\xi\rangle \log(t)|
    \lesssim_\xi 1,
\end{align*}
which leads us to
\begin{align}
    |\mathfrak{U}_3(\xi) - \pi \widehat{n_1}(0) \log(t)|
    \lesssim_\xi |\widehat{n_1}(0)|.
\end{align}

\textbf{Case III: Upper bound.}
For the upper bound, we only derive a rough one, but it covers the whole range of $\xi$.

Following Lemma \ref{lem:Theta}, we deduce that
\begin{align*}
    |\Theta(t, \xi)|
    \lesssim
    \epsilon \langle \xi\rangle^{-1} t^{1-p}.
\end{align*}

The proof is done.
\end{proof}

\subsection{Linear scattering for $E$}
In this part, we aim to prove Theorem \ref{thm:main2}, i.e., linear scattering of $E$ under the additional assumption $\int_{\mathbb{R}^2} n_1(x) \, \d x = 0$.
Denote
\begin{align}
    \mathcal{M}_0
    =
    &-{i\over 2} e^{it\langle\xi\rangle} \int \widehat{\ell_L}(t, \eta)  {e^{-it\langle \xi-\eta\rangle} \widehat{f_+}(t, \xi-\eta) \over \langle \xi - \eta\rangle}.
\end{align}

Our goal is to show the following result.
\begin{proposition}\label{prop:M0}
For all $m\gg 1$, $t_1, t_2 \in [2^m - 2, 2^{m+1}+2]$, and $k\leq \delta m$, it holds
\begin{align}\label{eq:M0}
\Big\| \phi_k(\xi) \int_{t_1}^{t_2}  \mathcal{M}_0(s, \xi) \, \d s \Big\|
\lesssim
(C_1 \epsilon)^2 2^{-\delta m} 2^{C_E k^-} 2^{D_E k^+}.
\end{align}
\end{proposition}

We first prepare some lemmas.
\begin{lemma}[Low frequency estimates for $\widehat{n_1}$]\label{lem:n1-low}
    Under the assumption \eqref{eq:n1=0}, we have
    \begin{align}
        |\widehat{n_1}(\eta)|
        \lesssim
        \epsilon |\eta|,
        \qquad
        \text{for } |\eta|\leq 1.
    \end{align}
\end{lemma}
\begin{proof}
    Thanks to the fast spacial decay of $n_1$ in physical space, one gets
    \begin{align*}
        \sum_a |\partial_{\eta_a} \widehat{n_1}|
        \lesssim
        \|\langle x\rangle n_1\|_{L^1}
        \lesssim
        \epsilon.
    \end{align*}
    On the other hand, the assumption $\int_{\mathbb{R}^2} n_1(x) \, \d x = 0$ implies
    \begin{align*}
        \widehat{n_1}(0,0) = 0.
    \end{align*}
Therefore, one concludes 
\begin{align}
    |\widehat{n_1}(\eta)|
    \lesssim
    \epsilon |\eta|,
\end{align}
    which completes the proof.
\end{proof}

\begin{proof}[Proof of Proposition \ref{prop:M0}]
    We decompose
    \begin{align*}
        \int_{t_1}^{t_2} \mathcal{M}_0(s, \xi) \, \d s
        =
        \mathfrak{M}_{0+} + \mathfrak{M}_{0-},
    \end{align*}
in which
\begin{align*}
     \mathfrak{M}_{0\pm} 
    = &\pm {1\over 4} \int_{t_1}^{t_2} 
         \int e^{is\Phi_{2\mp}} {\widehat{g_{\pm, L}}(\eta) \over |\eta|} { \widehat{f_+}(s, \xi-\eta) \over \langle \xi-\eta\rangle} \, \d\eta \d s
    \\
    = &\pm {1\over 4} \int_{t_1}^{t_2} \mathcal{F} T_{m_{0\mp}}(g_{\pm, L}, f_+) \, \d s,
\end{align*}
with $m_{0\pm} = e^{is\Phi_{2\pm}} |\eta|^{-1} \langle \xi-\eta\rangle^{-1}$.

We find that
\begin{align*}
    \|\phi_k(\xi) \mathcal{F} T_{m_{0\mp}}(g_{\pm, L}, f_+)\|
\lesssim
&\sum_{k_1\geq k_2 + 4} \|\phi_k(\xi)\mathcal{F}{T_{m_{0\mp}} (P_{k_1} g_{\pm,L}, P_{k_2}f_+)} \|
\\
&+
\sum_{k_2\geq k_1 + 4} \|\phi_k(\xi)\mathcal{F}{T_{m_{0\mp}} (P_{k_1} g_{\pm, L}, P_{k_2}f_+)} \|
\\
&+
\sum_{|k_1 - k_2| \leq 3} \|\phi_k(\xi)\mathcal{F}{T_{m_{0\mp}} (P_{k_1} g_{\pm, L}, P_{k_2}f_+)} \|
\\
= 
&\mathfrak{P}_{11} + \mathfrak{P}_{12} + \mathfrak{P}_{13}.
    \end{align*}

We use Lemma \ref{lem:n1-low} to find that
\begin{align*}
    \mathfrak{P}_{11}
    \lesssim
    &\sum_{\substack{|k_1 -k|\leq 6, \, k_1\geq k_2 + 4}} 2^{-k_1} 2^{-k_2^+} \big\| \phi_{\leq 0}(\eta \langle t\rangle^{p}) \widehat{P_{k_1}g_\pm} \big\|  \big\| \widehat{P_{k_2}f_+} \big\|_{L^1}
    \\
\lesssim
&(C_1 \epsilon)^2 \sum_{\substack{|k_1 -k|\leq 6 \\k_1\geq k_2 + 4, \, k_1\leq -p m+4}} 2^{k_1} 2^{k_2} 2^{-k_2^+} 2^{C_E k_2^-} 2^{D_E k_2^+}
\\
\lesssim
&(C_1 \epsilon)^2 2^{(-2p + \delta D_E)  m} 2^{C_E k^-} 2^{D_E k^+}. 
\end{align*}

Employing Lemma \ref{lem:n1-low}, we get
\begin{align*}
    \mathfrak{P}_{12}
    \lesssim
&\sum_{\substack{|k_2 -k|\leq 6, \,k_2\geq k_1 + 4}} 2^{-k_1} 2^{-k_2^+} \big\| \phi_{\leq 0}(\eta \langle t\rangle^{p}) \widehat{P_{k_1}g_\pm} \big\|_{L^1}  \big\| \widehat{P_{k_2}f_+} \big\|
\\
\lesssim
&(C_1 \epsilon)^2 \sum_{\substack{|k_2 -k|\leq 4 \\ k_2\geq k_1 + 4, \, k_1\leq -p_1m+4}} 2^{2k_1} 2^{-k_2^+} 2^{C_E k_2^-} 2^{D_E k_2^+}
\\
\lesssim
&(C_1 \epsilon)^2 2^{-2p m} 2^{C_E k^-} 2^{D_E k^+}. 
\end{align*}

We obtain
\begin{align*}
    \mathfrak{P}_{13}
    \lesssim
&2^k \sum_{\min\{k_1, k_2\}\geq k-6, \,|k_1 - k_2| \leq 3} \|\phi_{[k-1, k+1]}(\xi)\mathcal{F}{T_{m_{0\mp}} (P_{k_1} g_{\pm, L}, P_{k_2}f_+)} \|_{L^\infty}
\\
\lesssim
&2^k \sum_{\min\{k_1, k_2\}\geq k-6, \,|k_1 - k_2| \leq 3} 2^{-k_1} 2^{-k_2^+} \big\| \phi_{\leq 0}(\eta \langle t\rangle^{p}) \widehat{P_{k_1}g_\pm} \big\|_{L^\infty}  \big\| \widehat{P_{k_2}f_+} \big\|_{L^1}
\\
\lesssim
&(C_1 \epsilon)^2 \sum_{\substack{\min\{k_1, k_2\}\geq k-6\\ |k_1- k_2| \leq  3, \, k_1\leq -pm+4}} 2^{k} 2^{k_2} 2^{-k_2^+} 2^{C_E k_2^-} 2^{D_E k_2^+}
\\
\lesssim
&(C_1 \epsilon)^2 2^{-2p m} 2^{C_E k^-} 2^{D_E k^+}. 
\end{align*}

The proof is done.
\end{proof}


\section{Appendix}

\subsection{Proofs of Props. \ref{prop:Sobolev-in}, \ref{prop:Hardy-ex}, and Lemma \ref{lem:Q}}

\begin{proof}[Proof of Proposition \ref{prop:Sobolev-in}]
    Our proof is inspired by the work \cite{Psarelli}, but differs a bit from \cite{Hormander, Psarelli}. We note that it suffices to consider only large $t\gg 1$.

For a function $\phi(t, x)$ of spacetime, we restrict it to the hyperboloid $\mathcal{H}_\tau$, and $t= \sqrt{\tau^2 + |x|^2}$. 
Define 
$$
\cancel{\partial}_r 
= {L_r \over t}
= {x^a\over r} \cancel{\partial}_a,
\qquad
\cancel{\partial}_a
= {L_a \over t},
$$  
and one checks that
\begin{align*}
    \partial_a \big(\psi(\sqrt{1\tau^2+|x|^2}, x) \big)
    =
    (\cancel{\partial}_a \phi \big)(\sqrt{\tau^2+|x|^2}, x).
\end{align*}

\textbf{Case I: Large $r\geq {t\over 3}$.}

In this case, it holds that $r^{-1} \lesssim t^{-1}$.
By the Newton-Leibniz formula, one has
\begin{align*}
    &|\phi(\sqrt{\tau^2+|x|^2}, x)|^2 - |\phi(\sqrt{\tau^2+|x|^2}, x)|^2\big|_{|x|={\tau^2 -1\over 2}}
    \\
    =
    &-2\int_r^{{\tau^2 -1\over 2}} \phi \cancel{\partial}_r \phi \, \d\rho
    \\
    \lesssim
    &r^{-2} \Big(\int_0^{{\tau^2 -1\over 2}} \phi^2 \rho \, \d\rho \Big)^{1\over 2}
    \times \Big(\int_0^{{\tau^2 -1\over 2}} (\cancel{\partial}_r\phi)^2 \rho \, \d\rho \Big)^{1\over 2}.
\end{align*}
Recalling the Sobolev embedding on the circle $\mathbb{S}^1$
\begin{align*}
    |\phi|^2
    \lesssim
    \|\phi\|^2_{L^2(\mathbb{S}^1)} + \|\Omega \phi\|^2_{L^2(\mathbb{S}^1)},
\end{align*}
and this leads to the desired bounds on $\mathcal{H}_\tau$
\begin{align*}
&|\phi(t, x)|^2
\\
\lesssim
&t^{-2} \sum_a\Big(|\phi(t, x)|^2\big|_{|x|={{\tau^2 -1\over 2}}}
+\|\phi\|_{L^2(\mathcal{H}_2)} \|L_a \phi\|_{L^2(\mathcal{H}_2)}
+ \|\Omega \phi\|_{L^2(\mathcal{H}_2)} \|\Omega L_a \phi\|_{L^2(\mathcal{H}_2)}\Big).
\end{align*}

\textbf{Case II: Small $r\leq {t\over 2}$.}

In this case, we note the relation $\tau^{-1} \lesssim t^{-1}$.
We renormalize $\phi(\sqrt{\tau^2 + |x|^2}, x)$ by introducing the new variable $y = {x\over \tau}$ and consequently 
\begin{align*}
    \phi(\sqrt{\tau^2 + |x|^2}, x)
    =
    \phi(\tau \sqrt{1 + |y|^2}, \tau y).
\end{align*}
We then define a function on $\mathcal{H}_1$ by
\begin{align*}
    \psi(\sqrt{1+|y|^2}, y)
    =
    \phi(\tau \sqrt{1 + |y|^2}, \tau y),
\end{align*}
and we find that
\begin{align*}
    \partial_{y^a} \big( \psi(\sqrt{1+|y|^2}, y)  \big)
    =
    {\tau \over t} (L_a \phi )(\tau\sqrt{1+|y|^2}, \tau y).
\end{align*}
We apply the Sobolev embedding to get on $\mathcal{H}_\tau$ that
\begin{align*}
    &\phi(t, x)
    \\
    =
    &|\psi(\sqrt{1+|y|^2}, y)|^2
    \\
    \lesssim
    &\int_{\{z : \,|z |\leq |y| + 1\}} \Big( \psi^2(\sqrt{1 + |z|^2}, z)
    + \sum_a (\partial_{z^a} (\psi (\sqrt{1 + |z|^2}, z)))^2 \Big)  \, \d z
    \\
    \lesssim
    &\int_{\mathcal{H}_\tau} \Big( \phi^2(\tau \sqrt{1 + |\widetilde{y}|^2}, \tau \widetilde{y})
    + \sum_a (L_a \phi)^2(\tau \sqrt{1 + |\widetilde{y}|^2}, \tau \widetilde{y}) \Big)\, \d \widetilde{y}
    \\
    \lesssim
    & \tau^{-2} \Big( \|\phi\|_{L^2(\mathcal{H}_\tau)}^2 + \sum_a\|L_a \phi\|_{L^2(\mathcal{H}_\tau)}^2 \Big).
\end{align*}

We therefore conclude by combining the estimates established in \textbf{Case I} and \textbf{Case II}, and the proof is done.
\end{proof}

\begin{proof}[Proof of Proposition \ref{prop:Hardy-ex}]  

Let $\eta(r): \mathbb{R}^+ \to \R^+$ satisfy 
\begin{equation*}
\eta(r) =  
\begin{cases} 
  1, &\qquad  0\leq r \leq 1,  \\
  -\frac{1}{N_0}\big(r-(N_0+1)\big),  &\qquad  1 <r\leq N_0+1, \\
  0,  &\qquad  r>N_0 +1.
  \end{cases}
\end{equation*}
where $N_0$ is a large integer such that $N_0^{-1} < \gamma$.
By the choice of $\eta(r)$, we have 
$$r|\eta'(r)|\leq 1 + \frac{1}{N_0} < 1+\gamma. $$
Since $r-t \geq -\eta_0 \geq -\frac 32 $, we can get
\begin{align}
\partial_r\big((-\tau+ \frac 32)^{2\gamma +1} r\big) & = (2\gamma+1) (-\tau_-+\frac 32)^{2\gamma} r + (-\tau_-+\frac 32)^{2\gamma +1} \nonumber \\
& \geq (2\gamma+1) (-\tau_-+\frac 32)^{2\gamma} r.
\end{align}
Set $0<\epsilon <{\gamma\over 8}$.  Multiplying by $\eta(\epsilon r) |\phi|^2$ and integrating with respect to $(r, \theta)$ over the domain $\{(r,\theta) | r\geq t-\eta_0,\ 0\leq \theta \leq 2\pi \}$, we can get 
\begin{align}
& \int_0^{2\pi} \int_{\Sigma^{ex}_t} \partial_r\big((-\tau_-+\frac 32)^{2\gamma +1} r\big) \eta(\epsilon r) |\phi|^2 \d r \d \theta \nonumber \\
 \geq &(2\gamma +1)\int_0^{2\pi}\int_{\Sigma^{ex}_t}(-\tau_-+\frac 32)^{2\gamma}  \eta(\epsilon r)|\phi|^2 r\d r \d \theta. \label{ineq:hardy1}
\end{align}
Applying integration by parts, one can find 
\begin{align*}
& \int_{\Sigma^{ex}_t}\partial_r\big((-\tau_-+\frac 32)^{2\gamma +1} r\big) \eta(\epsilon r) |\phi|^2 
\d r  
\\
 = &- (\frac 32 -\eta_0)^{2\gamma+1}(t-\eta_0)_{+} \eta\big(\epsilon(t-\eta_0) \big)|\phi|^2(t-\eta_0,\theta) - \int_{\Sigma^{ex}_t}(-\tau_-+\frac 32)^{2\gamma +1}  \partial_r\big(\eta(\epsilon r) |\phi|^2 \big)
r\d r 
\\
 \leq &-\int_{\Sigma^{ex}_t}(-\tau_-+\frac 32)^{2\gamma +1}  \epsilon \eta'(\epsilon r) |\phi|^2 r\d r  
+ 2\int_{\Sigma^{ex}_t}(-\tau_-+\frac 32)^{2\gamma +1}  \eta(\epsilon r) |\phi| |\partial_r\phi| r\d r 
\\
 \leq & \int_{\Sigma^{ex}_t}(-\tau_-+\frac 32)^{2\gamma}  |\epsilon r \eta'(\epsilon r)| |\phi|^2 r\d r 
+\frac{3}{2} \epsilon \int_{r\geq t -\eta_0}(-\tau_-+\frac 32)^{2\gamma}  | \eta'(\epsilon r)| |\phi|^2 r\d r   
\\ 
&  + 2\int_{\Sigma^{ex}_t}(-\tau_-+\frac 32)^{2\gamma +1}  \eta(\epsilon r) |\phi| |\partial_r\phi| r\d r 
\\
 \leq &(1+\frac 54 \gamma) \int_{\Sigma^{ex}_t}(-\tau_-+\frac 32)^{2\gamma} |\phi|^2 r\d r 
+2\int_{\Sigma^{ex}_t}(-\tau_-+\frac 32)^{2\gamma +1}  |\phi| |\partial_r\phi| r\d r.
\end{align*}
Hence, integrating over the interval $[0,2\pi]$ with variable $\theta$ and using H\"older inequality and Young inequality, we can get
\begin{align*}
& \int_0^{2\pi} \int_{\Sigma^{ex}_t} \partial_r\big((-\tau_-+\frac 32)^{2\gamma +1} r\big) \eta(\epsilon r)|\phi|^2 \d r \d \theta  \\
 \leq &(1+\frac 54 \gamma)\int_{\Sigma^{ex}_t}(-\tau_-+\frac 32)^{2\gamma} |\phi|^2 \d x 
+ 2\int_{\Sigma^{ex}_t}(-\tau_-+\frac 32)^{2\gamma +1}   |\phi| |\partial_r\phi| \d x 
\\
 \leq &(1+\frac 54 \gamma)\| (-\tau_-+{3\over 2})^{\gamma} \phi \|^2_{L^2(\Sigma^{ex}_t) } 
+ \frac{\gamma}{4}\| (-\tau_-+{3\over 2})^{\gamma} \phi \|^2_{L^2(\Sigma^{ex}_t)} \\
&  +  \frac{4}{\gamma}\| (-\tau_-+{3\over 2})^{\gamma+1} \partial_r\phi \|^2_{L^2(\Sigma^{ex}_t)}.
\end{align*}  
This combining with \eqref{ineq:hardy1} infers that
\begin{align*}
& (2\gamma+1)\int_0^{2\pi}\int_{\Sigma^{ex}_t}(-\tau_-+\frac 32)^{2\gamma}  \eta(\epsilon r)|\phi|^2r\d r \d \theta 
\\
\leq &(1+\frac{3}{2}\gamma) \big\| (-\tau_-+{3\over 2})^{\gamma} \phi \big\|^2_{L^2(\Sigma^{ex}_t)} 
+ \frac{4}{\gamma}\big\| (-\tau_-+{3\over 2})^{\gamma+1} \partial_r\phi \big\|^2_{L^2(\Sigma^{ex}_t)}.
\end{align*}
Now let $\epsilon \to 0$ and using the relation $\langle \tau_- \rangle \simeq (-\tau_- + \frac 32)$ for $(x,t) \in \Sigma_{t}^{ex}$, we can get  the desired result.

\end{proof}

\begin{proof}[Proof of Lemma \ref{lem:Q}]

We claim that 
\begin{align} \label{eq:defbequiv}
B_l \simeq 2^{\max\{-l,0\}} \|P_l f\| + \sum_{a=1}^2 \|x_a P_l f\|.
\end{align}
Noticing the almost orthogonality of $\tilde{\phi}_{j}^{(l)}$ and the equality
\begin{align*}
\sum_{j\geq \max\{-l,\, 0\}} \widetilde{\phi}_{j}^{(l)} =1,
\end{align*}
one can get 
\begin{align*}
\frac 14 \leq \sum_{j\geq -l} \big( \widetilde{\phi}_{j}^{(l)} \big)^2 \leq 1.
\end{align*}
This   yields the desired result  \eqref{eq:defbequiv}.  

Let $l \geq 1$,  we can see that
\begin{equation}\label{ineq:weightplf}
\begin{aligned}
 \sum_{a=1}^2 \|x_a P_l f\| 
 &\lesssim   \sum_{a=1}^2 \big\|\partial_{\xi_a}(\phi_l(\xi) \widehat{f}(\xi))\big\| \\
 &\lesssim  \sum_{a=1}^2 \big\|(\partial_{\xi_a} \phi_l(\xi) ) \phi_{[l-4,l+4]}(\xi) \widehat{f}(\xi)\big\| 
 +  \big\|\phi_l(\xi) \partial_{\xi_a} \widehat{f}(\xi)\big\|  
 \\
 & \lesssim \sum_{|l'-l|\leq 4} A_{l'}.  
\end{aligned}
\end{equation}
For the case $l\leq 0$,  by \eqref{eq:defbequiv} we know
\begin{align*}
B_l \lesssim 2^{-l} \|P_l f\| + \sum_{a=1}^2 \|x_a P_l f\|.
\end{align*}
To obtain the desired result, it suffices to show that 
\begin{align}
2^{-l}  \|P_l f\| \lesssim  \langle l\rangle \sup_{l'} \|P_{l'} f\|_{H^1} + \sum_{|l' - l|\leq 6} A_{l'}. \label{ineq:negaplf}
\end{align}
Indeed, let $f_j = x_j f$, $j \in \{1,2\}$, hence, 
\begin{align*}
 f = \frac{1}{1+|x|^2} f + \sum_{j=1}^2 \frac{x_j}{1+|x|^2} f_j,
\end{align*}
and for any $k\in \mathbb{Z}$, 
\begin{align*}
\|P_k f\| + \sum_{j=1}^2 \|P_k f_j\| \lesssim A_k.
\end{align*}
Noticing that  for $j =1, 2$, we have 
\begin{align*}
\Big|\mathcal{F}\Big(\frac{1}{1+|x|^2}  \Big)(\xi) \Big| \lesssim |\xi|^{-1}, \ \ \  \Big|\mathcal{F}\Big(\frac{x_j}{1+|x|^2}  \Big)(\xi) \Big| \lesssim |\xi|^{-1}.
\end{align*}
Thus the proof of \eqref{ineq:negaplf} is reduced to showing with $g \in \{ |\widehat{f}|, \, |\widehat{f_j}|\}$
\begin{align}
2^{-l} \sum_{k\in \mathbb{Z}}\|\phi_l(\xi) \big((\phi_k\cdot g) \ast K \big)(\xi)\| \lesssim \langle l \rangle \sup_{k\in \mathbb{Z}}\|\phi_k \cdot g\|_{H^1} + \sum_{|k-l|\leq 6} \|\phi_k\cdot g\|,  \label{ineq:uppbdplf}
\end{align}
provided $\|\phi_k\cdot g\|_{H^1} \lesssim A_k,\ \forall \ k\in \mathbb{Z}$ and $K(\xi)= |\xi|^{-1}$.  In fact,
\begin{align*}
& 2^{-l} \sum_{k\in \mathbb{Z}}\|\phi_l(\xi) \big((\phi_k\cdot g) \ast K \big)(\xi)\| \\
& \leq 2^{-l} \sum_{|k-l|\leq 6}\|\phi_l(\xi) \big((\phi_k\cdot g) \ast K \big)(\xi)\| 
 +  2^{-l} \sum_{k\geq l+6}\|\phi_l(\xi) \big((\phi_k\cdot g) \ast K \big)(\xi)\|\\
& +  2^{-l} \sum_{k\leq l-6}\|\phi_l(\xi) \big((\phi_k\cdot g) \ast K \big)(\xi)\|\\
& \leq \mathfrak{I}_1 + \mathfrak{I}_2 + \mathfrak{I}_3.
\end{align*}
Using the support property of $\phi_k$, we can bound $\mathfrak{I}_1$ as follows,
\begin{align*}
\mathfrak{I}_1 & = 2^{-l} \sum_{|k-l|\leq 6} \big \|\phi_l(\xi) \big((\phi_k \cdot g) \ast (K\cdot \phi_{\leq l+10} \big)(\xi) \big\| \\
& \lesssim 2^{-l} \sum_{|k-l|\leq 6} \|\phi_k \cdot g\| \|K\cdot \phi_{\leq l+10}\|_{L^1} \\
& \lesssim \sum_{|k-l|\leq 6} \|\phi_k\cdot g\|.
\end{align*}
As for $\mathfrak{I}_2$, we have 
\begin{align*}
\mathfrak{I}_2 & =  2^{-l} \sum_{k\geq l+6}\|\phi_l(\xi) \big((\phi_k\cdot g) \ast (K\cdot \phi_{[k-4,k+4]}) \big)(\xi)\| \\
&  \lesssim \sum_{l+6\leq  k\leq 0} \|\phi_k\cdot g\| \|K\cdot \phi_{[k-4,k+4]}\| + \sum_{ k\geq 1} \|\phi_k\cdot g\| \|K\cdot \phi_{[k-4,k+4]}\| \\
& \lesssim \langle l \rangle \sup_{k\in \mathbb{Z}} \|\phi_k \cdot g\| + \sum_{k\geq 1}2^{-k} \|\nabla (\phi_k\cdot g)\| \\
& \lesssim \langle l \rangle \sup_{k\in \mathbb{Z}} \|\phi_k \cdot g\|_{H^1}.
\end{align*}
Finally we estimate $\mathfrak{I}_3$.  
\begin{align*}
\mathfrak{I}_3 & =2^{-l} \sum_{k\leq l-6}\|\phi_l(\xi) \big((\phi_k g) \ast (K\cdot \phi_{[l-4,l+4]}) \big)(\xi)\| \\
& \lesssim 2^{-l} \sum_{k\leq l-6} \|\phi_k g\|_{L^1} \|K\cdot \phi_{[l-4,l+4]}\|_{L^2} \\
& \lesssim 2^{-l}  \sum_{k\leq l-6} 2^k \|\phi_k g\| \\
& \lesssim \sup_{k\in \mathbb{Z}}  \|\phi_k g\|.
\end{align*}
This completes the proof of \eqref{ineq:uppbdplf}.  The desired result then follows from \eqref{eq:defbequiv}, \eqref{ineq:weightplf} and \eqref{ineq:uppbdplf}.
\end{proof}

\subsection{Proof of Prop. \ref{prop:cascade}}

\begin{proof}[Proof of Proposition \ref{prop:cascade}]\label{sec:App-2}

\textbf{Step 1: Estimates for $\ell$.}

This step is based on \cite[Chapter 4.3]{LiZh}.
    Recall that
    \begin{align*}
        \widehat{\ell}(t, \xi)
        =
        \cos((t-t_0)|\xi|) \widehat{n_0}(\xi)
        +
        {\sin((t-t_0)|\xi|) \over |\xi|} \widehat{n_1}(\xi)
        =: \mathfrak{A}_1 + \mathfrak{A}_2.
    \end{align*}
Easily, we know $\|\mathfrak{A}_1\| \lesssim \epsilon$. For $\mathfrak{A}_2$, we further decompose it as 
\begin{align*}
    \|\mathfrak{A}_2\|^2 = \mathfrak{A}_{21} + \mathfrak{A}_{22},
\end{align*}
where
\begin{align*}
    \mathfrak{A}_{21} 
    =& \int_{\{|\xi|\leq 1\}}{\sin^2((t-t_0)|\xi|) \over |\xi|^2} |\widehat{n_1}|^2(0) \, \d\xi,
    \\
     \mathfrak{A}_{22} 
    = &\int_{\{|\xi|\leq 1\}}{\sin^2((t-t_0)|\xi|) \over |\xi|^2} 
    \big( |\widehat{n_1}|^2(\xi) - |\widehat{n_1}|^2(0)\big) \, \d\xi
    +
    \int_{|\xi|\geq 1} {\sin^2((t-t_0)|\xi|) \over |\xi|^2} |\widehat{n_1}|^2(\xi) \, \d\xi.
\end{align*}
We easily get
$\mathfrak{A}_{22} \lesssim \epsilon^2$.
On the other hand, by the fact that
\begin{align*}
    &\int_{0}^{1} {\sin^2(s|\xi|) \over |\xi|} \, \d|\xi|
    =
    \int_{0}^{s} {\sin^2(\rho) \over \rho} \, \d\rho
    \\
    =
    &{1\over 2}\log s
    +
    \int_{0}^1 {\sin^2(\rho) \over \rho} \, \d\rho
    -
    {1\over 2}\int_1^s {\cos(2\rho) \over \rho} \, \d\rho,
\end{align*}
we derive that
\begin{align*}
    \big|\mathfrak{A}_{21} - {1\over 2}|\widehat{n_1}|^2(0) \log t\big|
    \lesssim
    \epsilon^2.
\end{align*}
This yields
\begin{align*}
    \big|\|\ell\|^2 - {1\over 2}|\widehat{n_1}|^2(0) \log t\big|
    \lesssim
    \epsilon^2.
    \end{align*}

\textbf{Step 2: Estimates for $n$.}

By Proposition \ref{prop:bound-t}, we have
\begin{align*}
    &\big|\|n\|^2 - {1\over 2}|\widehat{n_1}|^2(0) \log t\big|
    \\
    \leq
    &\big|\|\ell\|^2 - {1\over 2}|\widehat{n_1}|^2(0) \log t\big|
    + \| \partial\partial \mathfrak{m}\|^2
    \\
    \lesssim
    &\epsilon^2,
\end{align*}
which leads us to the desired result.

Therefore, the proof is completed.
\end{proof}


\section*{Acknowledgements} 

The authors would like to thank Jia Shen (Nankai), and the author S.D. additionally  thanks Bochen Liu (SUSTech), Qingtang Su (CAS), and Yi Zhou (Fudan) for helpful discussions. The authors also thank Jingya Zhao (SUSTech) for carefully reading the manuscript.






\end{document}